\documentclass[10pt,a4paper]{amsart}
\usepackage{amsmath,amsfonts,bm}
\usepackage{amssymb,enumitem,mathtools}
\usepackage{fullpage}
\usepackage{graphicx}
\usepackage{url}
\usepackage{appendix}
\usepackage{hyperref}

\makeatletter
\newcommand*{\transpose}{
  {\mathpalette\@transpose{}}
}
\newcommand*{\@transpose}[2]{
  \raisebox{\depth}{$\m@th#1\intercal$}
}

\usepackage{array,etoolbox}
\preto\tabular{\setcounter{magicrownumbers}{0}}
\newcounter{magicrownumbers}
\newcommand\rownumber{\stepcounter{magicrownumbers}\arabic{magicrownumbers}}

\preto\tabular{\setcounter{magicrownumbers2}{0}}
\newcounter{magicrownumbers2}
\newcommand\rownumbers{\stepcounter{magicrownumbers2}\arabic{magicrownumbers2}}

\newtheorem{theorem}{Theorem}[section]
\newtheorem{lemma}[theorem]{Lemma}
\newtheorem{proposition}[theorem]{Proposition}
\newtheorem{corollary}[theorem]{Corollary}
\theoremstyle{definition}

\newtheorem{example}[theorem]{Example}

\theoremstyle{remark}
\newtheorem{remark}[theorem]{Remark}

\newtheorem{question}[theorem]{Question}

\usepackage{multirow}
\usepackage{longtable}

\usepackage{pgfplots}
\pgfplotsset{width=10cm,compat=1.9}
\usepgfplotslibrary{groupplots}

\DeclareMathOperator{\tr}{tr}
\DeclareMathOperator{\Gal}{Gal}

\begin{document}

\title{Equiangular lines in Euclidean spaces: dimensions 17 and 18}

\date{}

\author{Gary R.W. Greaves}

\address{Division of Mathematical Sciences, 
  School of Physical and Mathematical Sciences, 
  Nanyang Technological University, 
  21 Nanyang Link, Singapore 637371, Singapore}
\email{gary@ntu.edu.sg}
\thanks{The first author was supported in part by the Singapore Ministry of Education Academic Research Fund (Tier 1);  grant numbers:  RG29/18 and RG21/20.}

\author{Jeven Syatriadi}
\address{Division of Mathematical Sciences, 
  School of Physical and Mathematical Sciences, 
  Nanyang Technological University, 
  21 Nanyang Link, Singapore 637371, Singapore}
\email{jeve0002@e.ntu.edu.sg}
\thanks{}

\author{Pavlo Yatsyna}
\address{Department of Algebra, 
Faculty of Mathematics and Physics,
Charles University, 
Sokolovsk\'a 83, 18600 Praha 8, Czech Republic}
\email{yatsyna@karlin.mff.cuni.cz}
\thanks{The third author was supported in part by the project PRIMUS/20/SCI/002 from Charles University.}

    \keywords{Equiangular lines, dimensions 17 and 18, compatible polynomials, eigenspace angles, polynomial interlacing}

\subjclass{Primary 05B40; Secondary 05B20}

\maketitle

\begin{abstract}
    We show that the maximum cardinality of an equiangular line system in 17 dimensions is 48, thereby solving a longstanding open problem.
    Furthermore, by giving an explicit construction, we improve the lower bound on the maximum cardinality of an equiangular line system in 18 dimensions to 57.
\end{abstract}

\section{Introduction}

A set of lines each passing through the origin of Euclidean space is called an \textbf{equiangular line system} if the angle between any pair of lines is the same.
Given $d$, we would like to know $N(d)$, the maximum cardinality of an equiangular line system in $\mathbb R^d$.
This problem dates back to a paper of Haantjes~\cite{haantjes48} from 1948, where the results $N(3)=N(4) = 6$ were reported.
The study of equiangular line systems largely developed in the 1970s due to the advent of the linear algebraic approach of Seidel et al.~\cite{lemmens73,vLintSeidel66,Seidel74}. 
In particular, in 1973, the value of $N(d)$ was known for all $d \leqslant 13$, $d=15$, and $d = 21,22,23$.

Recently, the authors solved the problem for dimensions $14$ and $16$ \cite{GSY21}.
However, up until now, and despite a considerable amount of research in the past 50 years, determining the value of $N(17)$ had remained an open problem.
Our main result is the following

\begin{theorem}
\label{thm:main}
$N(17) = 48$ and $N(18) \geqslant 57$.
\end{theorem}

Our work disproves part of a conjecture of Lin and Yu~\cite[Conjecture 3]{LinYu} that asserts $N(18)=56$; at the same time we verify another part of their conjecture that asserts $N(17)=48$.

Many of the largest constructions of equiangular line systems are found by carefully selecting blocks from the famous Witt design.
Recently, the lower bound for $N(18)$ was improved by Sz\" oll\H osi~\cite{Szo} and again by Lin and Yu~\cite{LinYu} by exactly this method.
In contrast, the configuration of 57 equiangular lines in $\mathbb R^{18}$ that is presented below was found using vectors from integer lattices.

In addition to the recent developments in $\mathbb R^{18}$, there have been many developments relating to the sequence $(N(d))_{d \in \mathbb N}$ in the last few years, including improvements to the upper bounds for $N(d)$ where $d = 14, 16, 17, 18, 19, 20$ \cite{Azarija75,Azarija95,GG18,GKMS16,GSY21,GreavesYatsyna19}.
There have also been various recent improvements to upper bounds for $N(d)$ for $d \geqslant 24$ using semidefinite programming, see \cite{GlazYu,KingTang,delaat,okudayu16,Yu17}. 

The asymptotic behaviour of $N(d)$ is quadratic in $d$ with a general upper bound of $d(d+1)/2$~\cite[Theorem 3.5]{lemmens73} and a general lower bound of $(32d^2 + 328d + 29)/1089$~\cite[Corollary 2.8]{GKMS16}.
One can also consider the related problem of, for fixed $\alpha \in (0,1)$, finding $N_\alpha(d)$, the maximum number of lines in $\mathbb R^d$ through the origin with pairwise angle $\arccos \alpha$. 
In a series of recent papers~\cite{balla18,bukh16,jiang1}, the problem of determining the value of $\lim_{d \to \infty} N_\alpha(d)/d$ was investigated extensively and completely resolved in \cite{jiang}.

In Table~\ref{tab:equi} below, including the improvements from this paper, we give the currently known values or lower and upper bounds for $N(d)$ for $d$ at most $43$.
(See Sequence A002853 in The On-Line Encyclopedia of Integer Sequences~\cite{oeis}.)
\begin{table}[ht]
	\begin{center}
	\setlength{\tabcolsep}{2pt}
	\begin{tabular}{c|ccccccccccccccccc}
		$d$  & 2 & 3        & 4           & 5  & 6  & 7--14 & 15 & 16 & 17 & 18 & 19 & 20 & 21 & 22 & 23 -- 41 & 42 & 43 \\\hline
		$N(d)$  & 3 & 6        & 6           & 10 & 16 & 28  & 36 & 40 & 48 & 57--60 & 72--74 & 90--94  & 126 & 176 & 276 & 276 -- 288 & 344
	\end{tabular}
	\end{center}
	\caption{Bounds for the sequence $N(d)$ for $2\leqslant d\leqslant 43$.  A single number is given in the cases where the exact number is known.  The improvements from this paper, $N(17) = 48$ and $N(18) \geqslant 57$ are included.}
	\label{tab:equi}
\end{table}

	Let $\mathcal L$ be an equiangular line system of cardinality $n$ in $\mathbb R^{d}$ with $n > d$.
	Suppose $\{ \mathbf{v}_1,\dots, \mathbf{v}_{n} \}$ is a set of unit  vectors spanning each line in $\mathcal L$.
	For any two distinct vectors $\mathbf v_i$ and $\mathbf v_j$ (with $i \ne j$) the
	inner product $\mathbf{v}_i^\transpose \mathbf{v}_j$ is equal to $\pm \alpha$ for some $\alpha \in (0,1)$.
	Thus, the Gram matrix $G$ for this set of vectors has diagonal entries equal to $1$ and off-diagonal entries equal to $\pm \alpha$.
	The $\{0,\pm 1\}$-matrix $S = (G-I)/\alpha$ is called the \textbf{Seidel matrix} corresponding to the set of lines $\mathcal L$.
	Note that changing the direction of our spanning vectors $\mathbf v_i$ corresponds to conjugating $S$ with a diagonal $\{\pm 1\}$-matrix.
	Furthermore, since $G$ is positive semidefinite with rank at most $d$, the smallest eigenvalue of $S$ is $-1/\alpha$ with multiplicity at least $n-d$.
    
    We follow the approach taken by the authors in \cite{GSY21}.
    First, using geometric and modular restrictions, we produce all polynomials that could potentially be the characteristic polynomial of a Seidel matrix that corresponds to 49 equiangular lines in $\mathbb R^{17}$.
    Once we have this list of candidate characteristic polynomials, we consider potential sets of characteristic polynomials for the corresponding principal submatrices for each polynomial in the list. 
    The extra difficulties we encountered in our endeavour to rule out 49 equiangular lines in $\mathbb R^{17}$ were twofold.
    Firstly, the number of candidate characteristic polynomials is much larger than the numbers in~\cite{GSY21} one obtains for dimensions $14$ and $16$.
    Secondly, there are candidate characteristic polynomials that pass all of the necessary conditions that were developed in~\cite{GSY21}.
    The main tool we developed to overcome these difficulties was the notion of compatibility for polynomials, which we define in Section~\ref{sec:compat}.
    To produce our exhaustive lists of candidate characteristic polynomials, we use the polynomial enumeration algorithm of \cite[Section 2.3]{GSY21}, which we have implemented in Magma~\cite{magma} and Mathematica~\cite{mathematica}.
    A supplementary file containing Magma code for all the computations performed in this paper is available online~\cite{github}. 
    The total running time of all the computations used in this paper is less than 9 hours running on a modern PC.
    
    The outline of this paper is as follows.
   In Section~\ref{sec:57}, we exhibit systems of 57 equiangular lines in $\mathbb R^{18}$ obtained from integer lattices.
   In Section~\ref{sec:ccp}, we produce all candidate characteristic polynomials of Seidel matrices that correspond to putative systems of 49 equiangular lines in $\mathbb R^{17}$.
   In Section~\ref{sec:compat}, we introduce the notion of compatibility for polynomials and in Section~\ref{sec:icp}, we define interlacing characteristic polynomials.
   From this point we are able to start ruling out some of the candidate characteristic polynomials from Section~\ref{sec:ccp}.
   Finally, in Sections~\ref{sec:angle0}, \ref{sec:common2simple}, \ref{sec:euler}, we gradually introduce more tools that allow us to rule out the remaining candidate characteristic polynomials.
   Tables of polynomials and certificates referenced in the main text are given in the appendix.

\section{A construction of 57 equiangular lines in 18 dimensions}
\label{sec:57}

In this section, we present some constructions of sets of 57 equiangular lines in $\mathbb R^{18}$.
The \textbf{Witt design}~\cite{Witt37} (see also \cite[Example 6.6]{taylor}) is a collection $\mathcal B$ of $759$ $8$-subsets of $P = \{0,\dots,23\}$.
Denote by $\mathbf e_i$ the standard basis vector of $\mathbb R^{P}$ whose $i$th entry is equal to $1$.
Let $\mathcal B_0$ be the subset of blocks of $\mathcal B$ that contain $0$.
Denote by $\mathbf e_P$ the all-ones vector: $\sum_{i \in P}\mathbf e_i$.
For each $B \in \mathcal B_0$, define the vector 
\[
\mathbf v_B \coloneqq 4\sum_{i \in B}\mathbf e_i-4 \mathbf e_0 -\mathbf e_P.
\]
Note that vectors $\mathbf v_B$ have just two distinct entries, equal to $-1$ and $3$.
For each $i \in \{1,\dots,23\}$, define the vector $\mathbf w_i \coloneqq 8\mathbf e_i+4 \mathbf e_0 -\mathbf e_P.$
Then the lines spanned by the vectors in $ \mathcal V := \{\mathbf v_B \; : \; B \in \mathcal B_0\} \cup \{ \mathbf w_i \; : \; i \in \{1,\dots,23\}\}$ form 276 equiangular lines in $\mathbb R^{23}$ (each vector in $\mathcal V$ is orthogonal to $4\mathbf e_0+\mathbf e_P$).
The previous largest known equiangular line systems consisting of 54 and 56 lines in  $\mathbb R^{18}$ were both found by ``searching inside the Witt design'' \cite{LinYu,Szo}, i.e., finding subsets $\mathcal W \subset \mathcal V$ such that the vectors in  $\mathcal W$ form a set of spanning vectors for a system of equiangular lines in $\mathbb R^{18}$.

In contrast, our constructions, presented below, consist of integer vectors of length $\sqrt{10}$.
We begin by constructing the set $L_0 = \{ \mathbf v \in \mathbb Z^{18} \; : \; \mathbf v^\transpose \mathbf v = 10 \} $, which has cardinality 36,808,740.
Next, we pick at random a vector $\mathbf v_1 \in L_0$.
Then we form the subset $L_1 \coloneqq \{ \mathbf v \in L_0 \; : \; \mathbf v^\transpose \mathbf v_1 = \pm 2\}.$
We again, pick at random a vector $\mathbf v_2 \in L_1$ and form the subset $L_2 \coloneqq \{ \mathbf v \in L_1 \; : \; \mathbf v^\transpose \mathbf v_2 = \pm 2\}.$
Repeat this process until we have $\mathbf v_1, \dots, \mathbf v_{18}$, where $\mathbf v_i \in L_{i-1}$ and $L_i \coloneqq \{ \mathbf v \in L_{i-1} \; : \; \mathbf v^\transpose \mathbf v_i = \pm 2\}$, for each $i \in \{1,\dots,18\}$.
If $\mathbf v_1, \dots, \mathbf v_{18}$ are not linearly independent then we start the process again.
Otherwise, we form a graph $G$ whose vertex set is $L_{18}$ where two vectors $\mathbf v$ and $\mathbf w$ are adjacent if $\mathbf v^\transpose \mathbf w = \pm 2$.
The vectors $\mathbf v_1, \dots, \mathbf v_{18}$ together with a clique from $G$ form spanning vectors for a system of equiangular lines in $\mathbb R^{18}$.
The sets of 57 vectors in Figures~\ref{fig:57linesInt}, \ref{fig:57lines1quad}, \ref{fig:57linesNoAsche}, and \ref{fig:57linesNoAsche2} were found by following the process described above.

\begin{figure}[htbp]
    \centering
\scalebox{0.80}{\parbox{.90\linewidth}{
	\begin{align*}
\left [
\begin{smallmatrix}
 2&2&0&0&+&0&0&0&0&0&0&0&0&0&0&0&0&-&0&-&0&-&0&0&+&0&0&+&+&+&+&+&+&+&+&+&-&+&0&0&+&+&-&0&0&0&0&+&-&-&0&0&0&-&0&0&+\cr
 0&0&2&2&0&0&0&0&0&0&0&0&0&-&0&0&-&+&0&0&+&0&0&-&-&-&-&0&0&-&-&+&0&+&0&0&0&0&+&-&+&-&-&+&0&+&-&0&0&0&0&-&-&0&0&-&0\cr
 0&-&+&0&2&2&-&0&0&0&0&0&-&0&0&0&+&0&0&0&-&0&+&0&+&-&-&0&0&+&-&0&0&0&0&+&0&-&-&+&+&-&0&-&+&+&+&0&+&0&0&-&+&0&0&0&0\cr
 0&-&+&-&0&0&2&2&+&-&-&0&-&0&0&0&0&0&0&+&+&-&0&0&-&0&-&0&0&0&+&0&+&0&+&-&-&0&-&0&0&0&-&0&+&0&0&+&0&+&0&0&+&0&0&0&0\cr
 +&0&0&0&0&0&0&0&2&2&0&0&0&0&+&-&0&0&0&0&0&+&+&+&-&+&-&-&-&0&+&0&-&0&0&-&0&+&-&0&+&0&0&+&-&0&0&+&+&0&0&-&0&+&0&0&+\cr
 -&0&0&0&+&-&0&0&+&-&2&+&0&0&+&-&0&0&-&0&0&+&+&-&-&-&0&0&0&0&0&0&+&+&+&+&+&+&+&+&-&-&0&0&0&0&0&0&0&0&0&0&0&0&+&+&+\cr
 0&-&0&-&0&0&-&-&0&0&0&2&0&0&0&+&-&0&+&-&+&0&+&-&0&0&-&0&0&0&0&+&-&-&0&0&0&+&+&+&0&+&-&-&0&0&0&0&0&+&+&+&+&+&-&-&+\cr
 0&0&+&0&0&0&0&0&0&0&-&-&2&+&0&0&-&0&+&+&-&0&0&0&-&+&0&-&0&+&0&+&0&0&-&+&+&0&+&-&0&+&0&-&+&0&0&0&+&0&+&-&0&-&0&+&+\cr
 0&0&-&0&0&+&0&0&-&+&+&0&0&2&0&0&0&0&0&0&0&+&0&+&-&0&0&+&0&-&-&+&+&+&-&0&-&+&0&-&0&0&0&-&0&+&-&0&-&+&-&-&+&+&-&+&0\cr
 0&-&0&+&+&0&+&-&0&0&0&0&0&+&2&+&0&0&+&0&0&+&-&0&0&-&0&-&+&0&0&-&0&+&-&-&+&0&0&0&0&+&-&0&0&-&+&+&-&0&-&0&0&0&+&0&0\cr
 +&0&+&-&-&+&+&0&0&0&0&-&-&0&-&2&0&0&0&-&+&0&+&+&0&-&+&-&-&+&0&-&-&+&0&+&0&+&+&0&-&0&0&0&-&+&+&0&0&+&0&0&-&0&0&+&0\cr
 -&+&0&+&0&0&-&+&0&0&0&0&0&-&-&0&2&-&+&+&0&+&0&+&0&0&-&-&+&+&0&0&0&0&0&-&0&0&+&0&0&0&0&-&-&+&-&+&-&+&+&+&0&0&+&0&0\cr
 0&0&0&0&0&0&0&0&0&0&0&+&-&-&0&-&0&2&0&-&0&0&+&0&0&0&+&0&+&+&0&0&-&0&-&-&-&-&+&-&+&+&-&0&0&0&0&-&0&0&-&0&+&-&+&+&0\cr
 -&+&0&0&-&0&0&+&-&+&+&0&0&0&+&0&0&+&-&-&-&-&-&0&0&0&0&0&0&0&0&0&0&0&0&0&0&0&0&0&+&+&+&+&+&+&+&+&+&+&+&+&+&+&+&+&+\cr
 0&0&0&0&0&+&-&+&0&0&0&-&0&0&0&0&0&-&+&-&+&0&0&-&-&+&0&+&+&+&+&-&0&+&-&0&+&-&0&0&0&-&0&+&-&0&0&0&+&-&-&+&+&+&-&0&+\cr
 0&0&0&0&0&+&0&0&-&+&0&0&+&0&0&-&+&0&+&0&+&+&+&-&0&-&-&0&+&-&+&0&0&-&+&0&-&0&0&-&-&0&+&0&+&-&+&-&0&+&-&0&0&-&0&+&+\cr
 +&0&+&0&+&0&0&0&0&0&+&0&-&+&-&0&0&0&+&+&0&0&0&0&0&+&+&+&-&0&-&+&-&0&+&0&+&-&0&-&0&0&+&+&0&-&+&+&-&+&0&+&0&0&+&-&+\cr
 0&0&0&-&0&+&0&+&+&-&-&+&0&0&0&0&-&-&-&0&+&+&-&+&+&0&0&+&+&0&-&-&-&+&0&0&0&0&0&-&+&0&+&0&+&-&-&-&0&0&+&0&-&+&+&0&+\cr
 \end{smallmatrix}
\right ]
\end{align*}
}}

    \caption{The matrix $F_1$ consisting of 57 (column) vectors in $\mathbb Z^{18}$.  The entries $\pm$ should be replaced with $\pm 1$ entries in the obvious way.}
    \label{fig:57linesInt}
\end{figure}

The Seidel matrix $S_{1} = {F_1^\transpose F_1}/2-5I_{57}$ corresponding to $F_1$
has characteristic polynomial $ \operatorname{Char}_{S_1}(x) = (x+5)^{39} (x-4) (x-7) (x-9)^2 (x-11)^9 (x-13)^4 (x-15)$.
Thus the 57 vectors in Figure~\ref{fig:57linesInt} span 57 equiangular lines in $\mathbb R^{18}$.

\begin{figure}[htbp]
    \centering
\scalebox{0.80}{\parbox{.90\linewidth}{
	\begin{align*}
\left [
\begin{smallmatrix}
 2&2&2&0&0&-&+&-&+&+&0&0&0&-&+&0&0&0&0&-&0&0&0&0&0&+&-&-&-&+&+&+&-&+&+&-&+&+&0&0&0&0&0&0&0&0&0&0&0&0&0&0&0&0&0&0&0\cr
 0&-&0&2&2&+&0&-&0&0&0&0&0&0&-&0&0&+&0&0&0&0&+&-&+&+&-&+&-&0&0&0&0&0&0&0&0&0&+&0&-&-&-&+&0&0&-&0&0&0&0&0&+&-&+&0&0\cr
 0&0&0&-&+&2&2&0&+&0&0&0&0&0&0&0&0&+&0&0&0&0&0&0&0&-&+&0&0&0&0&0&0&+&-&0&0&0&0&-&+&+&-&+&0&-&-&+&+&-&-&-&0&-&-&+&+\cr
 0&0&+&0&0&-&0&2&2&0&0&0&0&+&0&+&-&0&0&+&0&0&0&-&-&+&+&+&+&+&-&0&-&0&0&+&0&0&0&0&0&0&0&0&-&+&-&0&0&+&0&-&+&-&+&-&+\cr
 0&0&0&+&-&0&0&0&+&2&2&0&0&0&0&0&0&+&0&0&0&0&0&0&0&-&+&0&0&0&0&0&0&-&+&0&0&0&0&-&-&+&-&+&0&+&+&-&-&-&-&+&0&-&-&-&-\cr
 0&0&0&0&0&0&0&+&0&0&+&2&2&-&+&0&0&-&0&-&0&0&0&0&0&0&0&+&-&+&-&-&+&-&-&+&+&+&0&-&0&-&-&+&0&0&0&0&0&+&+&0&0&+&-&0&0\cr
 0&0&0&0&0&0&0&0&0&+&0&0&0&2&2&0&0&0&0&0&-&-&0&0&0&-&0&0&-&0&0&+&+&-&-&0&-&+&+&0&+&-&+&-&+&0&-&+&+&0&-&0&+&0&0&-&-\cr
 0&0&0&0&0&0&0&0&0&+&0&0&0&0&0&2&2&0&0&0&-&+&0&0&0&-&0&0&-&0&0&-&-&+&+&0&-&-&+&0&-&-&+&+&+&0&-&+&-&0&+&0&-&0&0&-&+\cr
 0&-&+&0&0&0&0&0&0&0&-&0&0&0&0&-&+&2&2&0&0&0&-&0&0&0&0&+&+&0&0&+&0&-&+&0&+&+&-&-&+&0&0&0&+&-&0&0&0&0&+&+&0&+&+&-&+\cr
 +&0&0&0&0&+&-&0&0&0&+&-&+&0&0&0&0&-&+&2&2&0&0&+&-&+&0&+&0&0&0&0&+&+&+&0&0&0&+&0&0&0&0&0&+&+&+&+&+&-&+&-&0&0&0&0&0\cr
 0&-&+&+&-&0&0&+&-&0&0&0&0&0&0&0&0&0&0&-&2&-&-&0&0&-&-&0&0&+&+&-&0&0&0&-&0&-&-&+&+&0&0&+&-&0&-&+&0&0&0&0&-&-&+&0&-\cr
 -&+&0&0&0&0&-&0&0&0&0&-&+&+&0&0&0&+&-&-&0&2&-&0&0&0&+&+&0&-&+&-&0&0&0&-&+&+&0&+&+&0&0&0&+&+&0&-&-&0&0&-&+&0&0&0&0\cr
 0&+&0&+&-&+&0&0&-&0&+&-&+&-&0&0&0&0&+&0&0&-&2&0&+&0&+&0&0&+&-&0&-&-&0&-&0&-&0&0&+&0&0&-&+&-&0&0&-&+&-&-&+&0&+&0&+\cr
 +&0&-&0&0&0&+&-&0&0&0&+&-&0&-&+&-&0&+&0&0&0&0&2&0&0&+&+&0&0&0&-&-&0&0&0&-&+&-&+&0&-&-&-&0&0&0&-&+&+&0&0&-&0&0&-&-\cr
 -&+&0&0&0&0&+&0&0&0&0&0&0&0&-&+&-&0&0&0&0&0&0&+&2&0&0&+&0&+&+&+&+&0&+&+&0&0&+&+&+&+&+&+&0&+&0&0&0&+&0&+&0&+&0&+&+\cr
 -&0&0&+&-&0&0&-&0&+&-&-&+&0&0&-&+&0&+&-&0&+&0&-&-&+&0&0&+&+&-&0&-&0&0&+&-&0&+&+&0&-&-&0&-&0&0&0&+&0&-&0&-&+&-&+&0\cr
 0&0&-&0&0&0&-&0&+&-&+&+&-&0&+&0&0&0&-&0&0&0&-&0&+&0&-&+&+&+&-&0&0&0&+&-&+&-&+&+&0&0&0&0&0&-&-&-&+&-&-&-&0&+&0&0&0\cr
 -&0&+&+&-&+&0&0&0&-&0&0&0&-&0&-&+&0&0&0&0&-&-&+&0&0&0&0&-&-&-&-&0&+&0&-&-&0&0&0&0&+&+&0&-&+&-&-&0&+&0&+&+&0&-&-&0\cr
\end{smallmatrix}
\right ]
\end{align*}
}}

    \caption{The matrix $F_2$ consisting of 57 (column) vectors in $\mathbb Z^{18}$.  The entries $\pm$ should be replaced with $\pm 1$ entries in the obvious way.}
    \label{fig:57lines1quad}
\end{figure}

The Seidel matrix $S_{2} = {F_2^\transpose F_2}/2-5I_{57}$ corresponding to $F_2$ in Figure~\ref{fig:57lines1quad}
has characteristic polynomial $ \operatorname{Char}_{S_2}(x) = (x+5)^{39} (x-7)^4 (x-9) (x-11)^5 (x-13)^6 (x^2-25x+152)$.

The systems of 57 equiangular lines in $\mathbb R^{18}$ given by $S_1$ and $S_2$ above were found by searching inside integer lattices as described above.
However, we do not know if the Seidel matrices $S_{1}$ and $S_2$ can also be found using lines spanned by vectors derived from blocks of the Witt design.

\begin{question}
\label{q2}
Can the systems of equiangular lines corresponding to $S_{1}$ and $S_2$ be found inside the Witt design?
\end{question}

A set of 72 equiangular lines in $\mathbb R^{19}$, known as \emph{Asche's lines} (see \cite[Example 5.19]{GKMS16}) is constructed by forming spanning vectors $\mathbf v_B$ from certain blocks $B \in \mathcal B$ of the Witt design.
While we cannot answer Question~\ref{q2}, we can find configurations of 57 equiangular lines that are not a subset of Asche's lines.
We exhibit two such configurations in Figures~\ref{fig:57linesNoAsche} and \ref{fig:57linesNoAsche2}.
The Seidel matrix corresponding to Asche's lines has characteristic polynomial $(x+5)^{53}(x-13)^{16}(x-19)^{3}$.
Thus, by interlacing (see Theorem~\ref{thm:cauchyinterlace}), the Seidel matrix corresponding to any subset of 57 lines of Asche's lines, would have $(x-13)$ as a factor of its characteristic polynomial.

\begin{figure}[htbp]
    \centering
\scalebox{0.80}{\parbox{.90\linewidth}{
	\begin{align*}
\left [
\begin{smallmatrix}
 2&2&0&+&0&0&-&+&0&0&0&0&0&0&0&-&0&0&-&0&0&0&0&+&-&-&+&-&-&-&-&+&-&-&+&+&+&+&+&+&+&0&0&0&0&0&0&0&0&0&0&0&0&0&0&0&0\cr
 0&0&2&2&0&0&0&0&+&0&0&-&+&0&0&-&0&0&-&0&0&0&0&-&+&-&+&-&+&+&+&-&-&0&0&0&0&0&0&0&0&-&-&-&-&-&-&+&+&+&0&0&0&0&0&0&0\cr
 -&+&0&0&2&2&0&0&0&0&+&0&-&+&0&+&0&0&0&+&0&-&0&-&-&-&+&0&0&0&0&0&0&-&+&-&0&0&0&0&0&+&-&-&+&0&0&0&0&0&+&-&+&0&0&0&0\cr
 0&0&-&0&+&-&2&2&+&0&0&0&+&0&-&0&-&0&0&0&+&0&0&+&0&0&0&-&+&-&0&0&0&0&0&0&-&-&0&0&0&+&+&0&0&-&-&-&0&0&-&0&0&-&-&0&0\cr
 0&0&0&0&-&0&0&0&2&2&0&0&0&+&0&-&-&+&0&0&0&-&0&0&+&+&0&-&0&0&-&0&0&-&+&0&+&-&-&0&0&0&0&+&0&+&+&0&+&-&+&-&-&+&0&+&0\cr
 0&0&+&0&0&0&0&0&0&0&2&2&-&0&-&0&0&0&+&0&-&-&0&0&0&0&-&0&+&0&-&-&0&+&0&+&+&0&0&0&0&-&-&0&-&-&0&0&-&-&-&-&0&0&-&-&-\cr
 0&0&-&0&0&0&0&0&0&0&0&+&2&2&+&0&0&0&+&-&-&0&-&0&-&+&0&0&0&0&0&0&0&-&0&+&0&-&+&-&+&-&0&-&0&0&-&+&0&0&0&+&+&+&0&-&+\cr
 -&+&+&-&0&0&-&+&+&-&0&0&0&0&2&0&0&+&-&0&0&0&0&+&-&0&0&+&0&0&0&0&-&0&-&0&+&-&0&-&0&+&+&+&0&-&+&0&+&0&0&0&0&0&-&-&-\cr
 0&0&0&+&0&0&+&-&+&0&+&0&0&+&+&2&0&0&0&0&+&-&0&0&0&+&0&0&-&+&0&0&-&+&+&+&0&+&+&0&-&0&+&+&+&-&0&+&0&+&0&0&0&-&-&+&+\cr
 -&+&0&0&0&+&0&0&0&+&+&-&0&0&0&0&2&0&0&0&0&0&0&+&+&+&-&0&+&+&0&+&-&0&+&-&-&0&0&+&+&-&0&0&0&+&0&-&+&-&0&+&+&-&-&-&0\cr
 0&0&0&0&-&-&0&0&0&0&0&0&0&0&0&+&+&2&0&+&0&0&+&0&0&0&-&0&0&0&+&+&-&-&0&-&+&0&-&-&+&+&0&0&-&0&-&+&-&-&-&-&+&0&+&+&+\cr
 0&0&0&0&-&+&0&0&0&+&0&-&+&0&0&0&-&-&2&+&0&0&0&0&-&-&0&-&+&+&-&+&-&0&-&0&0&+&+&-&-&0&+&0&-&0&+&-&0&0&0&-&+&0&+&0&0\cr
 0&0&0&+&+&0&+&-&-&+&0&0&0&+&0&0&-&0&0&2&0&-&+&0&0&0&-&0&-&-&+&0&0&+&-&0&+&-&0&+&+&0&0&0&-&0&0&-&+&+&0&+&+&+&0&0&-\cr
 +&-&0&0&0&0&0&0&0&+&+&0&+&-&0&0&0&-&0&0&2&0&0&+&-&+&+&+&0&+&+&+&0&0&0&0&0&0&-&+&+&0&-&0&0&-&+&+&0&0&-&-&0&+&+&-&-\cr
 0&0&0&0&0&0&0&0&0&0&+&0&0&+&-&0&0&0&0&+&+&2&-&0&-&0&-&0&0&-&-&-&-&0&0&0&0&+&-&0&+&-&+&-&+&+&0&+&0&+&-&0&-&0&0&+&-\cr
 0&0&+&-&-&+&+&-&0&0&0&-&0&0&+&0&+&0&0&0&-&0&2&+&0&0&0&-&+&-&0&0&0&-&-&+&0&0&-&+&-&0&0&-&+&0&0&0&-&+&-&0&0&+&-&0&0\cr
 +&-&-&0&0&0&0&0&0&+&0&0&0&0&0&0&0&-&-&-&0&+&+&+&0&-&-&+&+&0&-&-&-&-&0&-&+&-&0&-&0&0&-&+&0&-&-&0&-&0&+&0&+&-&0&+&-\cr
 +&-&0&-&0&+&-&+&-&0&-&+&0&0&0&+&0&+&0&0&0&0&+&-&0&0&0&-&0&0&0&+&-&0&+&-&-&0&+&0&0&-&0&+&-&0&-&0&+&+&-&-&-&+&-&0&-\cr
\end{smallmatrix}
\right ]
\end{align*}
}}
    \caption{The matrix $F_3$ consisting of 57 (column) vectors in $\mathbb Z^{18}$. The entries $\pm$ should be replaced with $\pm 1$ entries in the obvious way.}
     \label{fig:57linesNoAsche}
\end{figure}

\begin{figure}[htbp]
    \centering
\scalebox{0.80}{\parbox{.90\linewidth}{
	\begin{align*}
\left [
\begin{smallmatrix}
2&2&0&0&0&0&0&0&0&0&0&0&0&+&0&+&-&0&+&0&0&+&-&-&-&-&+&-&-&+&+&-&+&-&-&+&+&+&-&0&0&0&0&0&0&0&0&0&0&0&0&0&0&0&0&0&0\cr
 0&0&2&2&+&+&0&0&0&0&0&0&0&0&+&0&0&+&0&0&0&+&+&+&+&-&+&-&+&+&0&0&0&0&0&0&0&0&0&-&-&+&-&+&+&-&-&+&+&+&-&0&0&0&0&0&0\cr
 0&0&0&0&2&2&+&0&0&0&-&-&0&0&0&-&0&0&0&0&0&+&-&0&0&0&0&0&0&0&-&-&-&-&0&0&0&0&0&-&-&-&+&-&-&+&+&-&0&0&0&-&-&+&+&-&0\cr
 0&0&0&0&0&0&2&2&0&0&-&0&0&+&0&0&0&0&0&-&+&-&0&+&-&+&+&-&0&0&-&+&0&0&-&-&0&0&0&-&+&+&+&+&0&0&0&0&-&-&0&-&+&0&0&0&-\cr
 0&+&+&-&0&0&0&+&2&2&0&0&0&0&-&-&0&0&0&0&-&+&0&+&-&0&0&0&+&0&0&0&-&0&-&0&-&+&0&+&+&-&-&0&+&+&0&0&0&0&-&0&0&+&-&+&0\cr
 +&0&0&0&0&0&0&0&+&-&2&2&+&0&+&0&0&0&0&-&0&-&0&0&0&-&+&+&0&-&-&-&0&-&-&0&-&-&0&0&0&0&0&-&0&0&+&+&0&0&-&0&0&+&-&-&-\cr
 -&0&+&-&0&0&0&+&0&0&0&0&2&0&0&0&0&0&-&+&0&-&-&0&0&-&+&-&-&+&-&0&+&0&0&+&-&0&-&+&-&0&0&0&+&0&+&0&-&-&0&0&0&-&+&+&+\cr
 0&0&+&-&0&-&-&0&0&0&-&+&0&2&0&0&+&+&0&0&0&+&0&-&+&-&-&0&0&0&0&-&0&-&-&-&-&0&+&-&0&+&0&+&-&0&0&-&-&0&0&+&+&0&0&0&0\cr
 0&0&0&0&-&+&+&0&-&+&0&0&+&0&2&0&0&-&0&-&+&0&-&-&+&0&0&+&0&-&0&0&0&+&0&0&+&+&-&0&+&-&0&+&0&+&-&0&0&+&-&+&+&+&+&0&0\cr
 0&0&-&+&0&0&0&+&0&0&0&-&0&0&0&2&+&0&0&0&0&0&-&0&0&-&-&0&+&+&+&0&-&0&+&+&0&-&+&0&+&0&+&+&+&+&+&+&-&0&-&+&-&0&0&-&-\cr
 0&0&-&+&-&0&0&0&0&0&+&0&0&0&-&0&2&0&+&-&+&+&+&0&0&0&0&-&0&0&-&0&0&+&-&+&0&-&0&0&-&0&-&0&-&0&0&-&-&0&-&+&-&+&+&+&+\cr
 +&-&0&0&0&0&+&0&-&+&0&0&+&-&0&0&-&2&0&0&0&0&0&0&0&+&-&0&0&0&0&-&-&0&0&0&0&-&-&-&0&-&0&0&+&-&+&-&0&-&-&+&+&-&-&0&+\cr
 0&0&0&0&0&+&-&+&0&0&0&0&0&+&0&0&0&-&2&-&+&0&-&0&0&+&+&+&+&+&0&0&-&-&+&+&-&-&0&0&0&0&0&-&0&-&-&0&+&-&+&+&+&0&0&+&0\cr
 0&+&0&0&0&0&0&0&-&+&0&0&-&0&0&+&0&0&+&2&+&-&-&+&-&0&0&0&0&0&0&-&0&+&0&-&-&0&-&0&0&+&-&-&0&-&0&-&-&+&0&0&0&+&-&-&0\cr
 +&+&0&0&0&0&0&0&+&-&-&+&-&0&0&0&0&0&-&0&2&0&0&+&+&0&0&-&-&-&0&+&-&0&+&+&+&0&0&0&+&0&-&-&0&0&0&0&0&-&-&0&0&0&0&-&+\cr
 0&0&0&0&-&0&0&+&0&0&0&-&0&-&-&0&+&0&+&0&0&0&0&-&+&0&0&0&-&-&+&+&-&-&-&+&0&0&-&-&-&+&-&0&0&+&+&0&+&+&0&-&+&-&-&0&-\cr
 +&-&0&0&+&-&0&0&+&-&+&-&-&+&+&+&+&-&0&0&0&0&0&-&-&0&0&-&+&-&-&0&0&+&0&0&0&+&-&-&0&-&0&0&+&-&+&-&+&0&0&0&0&0&0&+&-\cr
 +&-&-&+&+&-&+&-&0&0&0&0&0&0&0&-&0&+&0&0&0&0&-&0&0&-&+&0&-&0&+&+&-&0&0&0&-&+&+&+&0&0&-&0&+&0&0&-&0&-&+&+&+&+&+&0&-\cr
\end{smallmatrix}
\right ]
\end{align*}
}}
    \caption{The matrix $F_4$ consisting of 57 (column) vectors in $\mathbb Z^{18}$. The entries $\pm$ should be replaced with $\pm 1$ entries in the obvious way.}
     \label{fig:57linesNoAsche2}
\end{figure}

The Seidel matrices $S_{3} = {F_3^\transpose F_3}/2-5I_{57}$ and $S_{4} = {F_4^\transpose F_4}/2-5I_{57}$ corresponding to $F_3$ and $F_4$ (see Figures~\ref{fig:57linesNoAsche} and \ref{fig:57linesNoAsche2})
have characteristic polynomials 
$\operatorname{Char}_{S_3}(x) =(x + 5)^{39} (x - 11)^{11} (x - 15) (x^2 - 19x + 76)(x^2 - 20x + 87)^2$ and $\operatorname{Char}_{S_4}(x) = (x+5)^{39} (x - 11)^{13} (x^2 - 18x + 57) (x^3 - 34x^2 + 361x - 1136)$.
We note that we are able to find numerous other configurations of 57 equiangular lines in $\mathbb R^{18}$ whose corresponding Seidel matrices have more than seven distinct eigenvalues.  
The four configurations given here in Figures~\ref{fig:57linesInt}, \ref{fig:57lines1quad}, \ref{fig:57linesNoAsche}, and \ref{fig:57linesNoAsche2}, are the four that we found that have precisely seven distinct eigenvalues.
We do not know if there are Seidel matrices corresponding to $57$ equiangular lines in $\mathbb R^{18}$ having fewer than seven distinct eigenvalues.

The \textbf{automorphism group} of a Seidel matrix $S$
is the group of all signed permutation matrices $P$ such that $S = PSP^\transpose$, where we do not make a distinction between $P$ and $-P$.
Using the computer algebra system Magma~\cite{magma}, we find that the Seidel matrices $S_{1}$, $S_{2}$, $S_3$ and $S_{4}$ each has a trivial automorphism group.

\section{Candidate characteristic polynomials for 49 equiangular lines in 17 dimensions}
\label{sec:ccp}

  In this section, we produce a list of polynomials that could potentially be the characteristic polynomial of a Seidel matrix that corresponds to a system of 49 equiangular lines in $\mathbb R^{17}$.
  Let $\mathcal L$ be an equiangular line system of cardinality $49$ in $\mathbb R^{17}$ and let $S$ be the Seidel matrix corresponding to $\mathcal L$.
	The Seidel matrix $S$ is a symmetric matrix with real entries, which means that each zero of its characteristic polynomial $\operatorname{Char}_S(x) \coloneqq \det(xI-S)$ is real.
	In other words, $\operatorname{Char}_S(x)$ is a \textbf{totally-real} polynomial.
	Moreover, since every entry of $S$ is an integer, each coefficient of $\operatorname{Char}_S(x)$ is also an integer.
	
	By \cite[Theorem 3.4]{lemmens73} together with \cite[Lemma 6.1]{vLintSeidel66}, \cite[Theorem 4.5]{lemmens73}, and \cite[Lemma 3.1]{GSY21}, the Seidel matrix $S$ must have smallest eigenvalue equal to $-5$ and $9$ must be an eigenvalue with multiplicity at least $4$.
	We record this result as a theorem.

	\begin{theorem}
	\label{thm:ev_mult}
	    Let $S$ be a Seidel matrix corresponding to $49$ equiangular lines in $\mathbb R^{17}$.
	    Then 
$$
\operatorname{Char}_S(x)=(x+5)^{32}(x-9)^4 \phi(x),
$$
for some monic polynomial $\phi$ of degree 13 in $\mathbb{Z}[x]$ all of whose zeros are greater than $-5$.
	\end{theorem}

Let $p(x)=\sum_{i=0}^n a_i x^{n-i}$ be a monic polynomial in $\mathbb{Z}[x]$.
Following~\cite{GSY21}, we say $p$ is \textbf{type $\mathbf 2$} if $2^i$ divides $a_i$ for all $i\geqslant 0$ and \textbf{weakly type $\mathbf 2$} if $2^{i-1}$ divides $a_i$ for all $i\geqslant 1$.

\begin{lemma}[{\cite[Lemma 2.7]{GSY21}}]
\label{lem:shiftType}
Let $S$ be a Seidel matrix of order $n$ and $k$ be an odd integer.
Then $\operatorname{Char}_S(x+k) = \operatorname{Char}_{S-kI}(x)$ is weakly type 2.
Furthermore, if $n$ is even then $\operatorname{Char}_S(x+k)$ is type 2.
\end{lemma}

The following lemma deals with the factorisation of type-2 and weakly-type-2 polynomials.

\begin{lemma}[{\cite[Lemma 2.8]{GSY21}}]\label{lem:type_factor}
Let $p\in \mathbb{Z}[x]$ be a monic polynomial.
Suppose $p=qr$ where $q,r \in \mathbb{Z}[x]$.
Then 
\begin{itemize}
    \item[(a)] $p$ is type 2 if and only if $q$ and $r$ are both type 2;
    \item[(b)] $p$ is weakly type 2 if and only if $q$ and $r$ are both weakly type 2 and at least one of them is type 2.
\end{itemize}
\end{lemma}

Denote by $\mathcal S_n$ the set of all Seidel matrices of order $n$.
Given a positive integer $e$, define the set $\mathcal P_{n,e} = \{ \operatorname{Char}_S(x) \mod 2^e \mathbb Z[x] \; : \; S \in \mathcal S_n \}$.
We will require the following upper bound on the cardinality of $\mathcal P_{n,e}$ for odd $n$.

\begin{theorem}[{\cite[Corollary 3.13]{GreavesYatsyna19}}]\label{thm:countCharPolySeidelOdd}
			Let $n$ be an odd integer and $e$ be a positive integer.
			Then the cardinality of $\mathcal P_{n,e}$ is at most $2^{\binom{e-2}{2}+1}$.
\end{theorem}

\begin{remark}
\label{rem:e7}
It turns out that for reasonably small values of $n$ and $e$ (e.g., $n \leqslant 49$ and $e \leqslant 7$), with $n$ large enough compared to $e$, equality holds in Theorem~\ref{thm:countCharPolySeidelOdd}.
For fixed $n$ and $e$, by randomly generating Seidel matrices of order $n$ and computing their characteristic polynomials modulo $2^e \mathbb Z[x]$, we can obtain a lower bound on the cardinality of $\mathcal P_{n,e}$.
If we find that this lower bound is equal to $2^{\binom{e-2}{2}+1}$ then, by Theorem~\ref{thm:countCharPolySeidelOdd}, we must have obtained all elements of $\mathcal P_{n,e}$.
In particular, for $n = 47$ and $49$, we can explicitly construct the set $\mathcal P_{n,7}$ using this method.
We use $\mathcal P_{49,7}$ in Lemma~\ref{lem:candpols} and $\mathcal P_{47,7}$ is used in Lemma~\ref{lem:last8_fourint} and Lemma~\ref{lem:last8_quad64}.
Note that, for $e$ larger than $7$, it becomes computationally infeasible to compute $\mathcal P_{n,e}$.
\end{remark}

Let $S$ be a Seidel matrix corresponding to an equiangular line system of cardinality $49$ in $\mathbb R^{17}$.
In view of Theorem~\ref{thm:ev_mult}, the next step is to find feasible polynomials for $\phi(x)=\sum_{t=0}^{13} b_t x^{13-t}$ where
$$
    \operatorname{Char}_S(x) = (x+5)^{32}(x-9)^{4} \phi(x).
$$
Obviously $b_0 = 1$.
And we can find $b_1$ and $b_2$, using $\tr S=0$ and $\tr S^2=49\cdot 48 = 2352$ together with Newton's identities: $b_1=-124$, and $b_2=7074$.

By Lemma~\ref{lem:shiftType}, the polynomial
$$
\operatorname{Char}_S(x-1)=(x+4)^{32}(x-10)^{4}\phi(x-1)
$$
is weakly type 2.
By Lemma~\ref{lem:type_factor}, the polynomial $\phi(x-1)$ is also weakly type 2.
Thus, we need to find all totally-real, integer polynomials $\phi(x)$ with the following properties:
\begin{enumerate}[label=(\roman*)]
\item $b_0=1$, $b_1=-124$, and $b_2=7074$,
\item $\phi(x-1)$ is weakly type 2.
\item $(x+5)^{32}(x-9)^{4} \phi(x)$ belongs to a congruence class in $\mathcal{P}_{49,7}$. (See Remark~\ref{rem:e7}.)\label{itm:iv}
\end{enumerate}

In order to enumerate the candidate characteristic polynomial for $S$, we employ an algorithm developed by McKee and Smyth (see \cite[Section 3]{MS04}, \cite[Section 4.3]{GreavesYatsyna19} or \cite[Section 2.3]{GSY21}).
We use this algorithm to generate totally-real (weakly)-type-2 polynomials whose top three coefficients are fixed, which comes up frequently throughout this paper.
Before we present the results of our enumeration of candidate characteristic polynomial for $S$, let us note that the number of totally-real, monic, integer polynomials of fixed degree and whose top three coefficients are also fixed is finite.
Indeed, suppose $p(x) = x^d + c_1x^{d-1} + c_2x^{d-2}+\dots = \prod_{i=1}^d (x-\lambda_i)$ is a totally-real integer polynomial.
Then, using Newton's identities, we have that $\lambda_i^2 \leqslant c_1^2-2c_2$ for each $i \in \{1,\dots,d\}$.
Furthermore, since $c_i$ is a (multivariate) polynomial in $\lambda_1,\dots,\lambda_d$, we can bound $c_i$ by functions of $c_1$ and $c_2$ for each $i \in \{3,\dots,d\}$.

We list the result of our computations in the following lemma.

\begin{lemma}
\label{lem:candpols}
    Let $S$ be a Seidel matrix corresponding to $49$ equiangular lines in $\mathbb R^{17}$.
     Then $\operatorname{Char}_S(x)$ is in one of the sets
        \begin{align*}
        \mathcal A &= \{(x+5)^{32}(x-9)^{16}(x-16)\} \\
        \mathcal B &= \{ (x+5)^{32}(x-8)(x-9)^8 (x^2-20x+95)^4, (x+5)^{32}(x-8)(x-9)^{12}(x^2-22x+113)^2 \} \\
            \mathcal C &= \left \{ \begin{array}{l}
                (x+5)^{32} (x-7) (x-9)^{14} (x-12) (x-15), \\
                (x+5)^{32} (x-7) (x-8) (x-9)^{12} (x-11)^2 (x-15), \\
                (x+5)^{32} (x-9)^{13} (x-11)^2 (x^2-21x+92), \\
                (x+5)^{32} (x-7)^2 (x-8) (x-9)^{10} (x-11)^2 (x-13)^2, \\
                (x+5)^{32} (x-9)^{13} (x-13)^2 (x^2-17x+64), \\
                (x+5)^{32} (x-9)^{12} (x-11)^3 (x^2-19x+72), \\
                (x+5)^{32} (x-7) (x-9)^{10} (x-11)^4 (x^2-19x+76), \\
                (x+5)^{32} (x-4) (x-9)^{10} (x-11)^6
            \end{array}
            \right \}
        \end{align*}
or is one of the 164 polynomials listed in Table~\ref{tab:164polydim17}, one of the 11 polynomials listed in Table~\ref{tab:11polydim17}, or one of the 8 polynomials listed in Table~\ref{tab:8polydim17}.
\end{lemma}

Using a modern PC, it takes about 2 hours and 30 minutes to find all 194 polynomials of Lemma~\ref{lem:candpols}.
The computations were executed in Magma~\cite{magma}, with the output independently verified using  Mathematica~\cite{mathematica}.

In the statement of Lemma~\ref{lem:candpols}, we have put the 194 polynomials of Lemma~\ref{lem:candpols} into six groups.
The grouping is based on the techniques we use to show that they cannot be the characteristic polynomial of any Seidel matrix.
First note that \cite[Remark 5.5.]{GG18} immediately precludes the existence of a Seidel matrix with characteristic polynomial in $\mathcal A$.
The rest of the paper is devoted to ruling out the remaining polynomials from Lemma~\ref{lem:candpols}.
From Section~\ref{sec:icp}, we will gradually introduce necessary conditions for a polynomial to be the characteristic polynomial of a Seidel matrix.
The next polynomials to be ruled out are those from the set $\mathcal B$ (see Lemma~\ref{lem:setB}).
Followed by the polynomials of Table~\ref{tab:164polydim17} (see Lemma~\ref{lem:firstTable}), Table~\ref{tab:11polydim17} (see Lemma~\ref{lem:secondTable}), and Table~\ref{tab:8polydim17} (see Lemma~\ref{lem:thirdTable}).
Since they require the most work, the polynomials in $\mathcal C$, will be dealt with last, in Lemmas~\ref{lem:last8_quad92}, \ref{lem:last8_fiveint}, \ref{lem:last8_sixint_ev15}, \ref{lem:last8_sixint_ev13},  \ref{lem:last8_fourint}, \ref{lem:last8_quad76}, \ref{lem:last8_quad72}, and \ref{lem:last8_quad64} where the sophistication of the techniques required for each polynomial steadily increases.   

\section{Angle vectors and compatibility}
\label{sec:compat}
In this section, we introduce angle vectors and define the notion of compatibility for polynomials.
Let $M$ be a real symmetric matrix of order $n$.
We write $\Lambda(M)$ for the set of distinct eigenvalues of $M$ and define the polynomial $$\operatorname{Min}_M(x) \coloneqq \prod_{\lambda \in \Lambda(M)} (x-\lambda),$$ which is the minimal polynomial of $M$.
For each $\lambda \in \Lambda(M)$, denote by $\mathcal{E}(\lambda)$ the eigenspace of $\lambda$ and let $\{\mathbf{e}_1,\dots,\mathbf{e}_n\}$ be the standard basis of $\mathbb R^n$.
Denote by $P_\lambda$ the orthogonal projection of $\mathbb R^n$ onto $\mathcal{E}(\lambda)$.
For a vector $\mathbf v$, we write $\mathbf v(i)$ to denote its $i$th entry.
\begin{theorem}[Spectral Decomposition Theorem]\label{thm:spec_decomp}
Let $M$ be a real symmetric matrix.
Then
$$
M=\sum_{\lambda \in \Lambda(M)}\lambda P_{\lambda}.
$$
\end{theorem}

Denote by $\bm \alpha_\lambda$ the \textbf{angle vector} for $\lambda \in \Lambda(M)$, that is, for each $i \in \{1,\dots,n\}$,
\[
    \bm \alpha_\lambda(i) = ||P_{\lambda} \mathbf{e}_i||.
\]
Define $$\operatorname{Quo}_M(x) \coloneqq \operatorname{Char}_M(x) / \operatorname{Min}_M(x),$$ and denote by $M[i]$ the principal submatrix of $M$ obtained by deleting its $i$th row and column.
		
		\begin{proposition}[See {\cite[(4.2.8)]{crs97}} or \cite{GoMK}]\label{prop:submatrix}
			Let $M$ be a real symmetric matrix of order $n$.
			Then, for each $i \in \{1,\dots,n\}$, we have
		$$\operatorname{Char}_{M[i]}(x)=\operatorname{Char}_M(x)\sum_{\lambda \in \Lambda(M)}\dfrac{\bm \alpha^2_\lambda(i)}{x-\lambda}.$$
		\end{proposition}

Cauchy's interlacing theorem, below, provides bounds for the eigenvalues of principal submatrices of $M$.

		\begin{theorem}[\cite{Cau:Interlace,Fisk:Interlace05,Hwang:Interlace04}]\label{thm:cauchyinterlace}
			Let $M$ be a real symmetric matrix having eigenvalues $\lambda_1 \leqslant \lambda_2 \leqslant \dots \leqslant \lambda_n$ and suppose $M[i]$, for some $i \in \{1,\dots,n \}$, has eigenvalues $\mu_1 \leqslant \mu_2 \leqslant \dots \leqslant \mu_{n-1}$.
			Then
			\[
				\lambda_1 \leqslant \mu_1 \leqslant \lambda_2 \leqslant \dots \leqslant \lambda_{n-1} \leqslant \mu_{n-1} \leqslant \lambda_n.
			\]
		\end{theorem}

	    Given $e \in \mathbb N$ and polynomials $f(x) = \prod_{i=0}^{e}(x-\lambda_i)$ and $g(x) = \prod_{i=1}^{e}(x-\mu_i)$ such that $\lambda_0 \leqslant \lambda_1 \leqslant \dots \leqslant \lambda_e$, and $\mu_1 \leqslant \mu_2 \leqslant \dots \leqslant \mu_e$, we say that $g$ \textbf{interlaces} $f$ if $\lambda_0 \leqslant \mu_1 \leqslant \lambda_1 \leqslant \dots \leqslant \mu_e \leqslant \lambda_e$.
Note that, if a polynomial $\mathfrak f(x)$ interlaces $\operatorname{Char}_{M}(x)$ then we can write $\mathfrak f(x) = \operatorname{Quo}_M(x)f(x)$, where $f(x)$ is a monic integer polynomial that interlaces $\operatorname{Min}_M(x)$.	

Given a polynomial $p(x) \in \mathbb Q[x]$, we denote its derivative by $p^\prime(x)$.
In the next result, we give a convenient expression for the entries of an angle vector.

\begin{lemma} \label{lem:anglesq_formula}
Let $M$ be a real symmetric matrix of order $n$.
Let $i \in \{1, \dots, n\}$ and suppose that $\operatorname{Char}_{M[i]}(x) = \operatorname{Quo}_M(x) \cdot f_i(x)$ for some polynomial $f_i(x) \in \mathbb{Z}[x]$.
Then, for each $\lambda \in \Lambda(M)$, we have
\begin{align*}
\bm \alpha_{\lambda}(i) = \sqrt{\frac{f_i(\lambda)}{\operatorname{Min}_M^\prime (\lambda)}}.
\end{align*}
\end{lemma}

\begin{proof}
By Proposition~\ref{prop:submatrix},
\begin{align*}
\operatorname{Char}_{M[i]}(x) = \operatorname{Char}_M(x) \sum_{\lambda \in \Lambda(M)} \frac{\bm \alpha_{\lambda}^2(i)}{x-\lambda}.
\end{align*}
Dividing both sides by $\operatorname{Quo}_M(x)$, we obtain
\begin{align*}
f_i(x) = \operatorname{Min}_M(x) \sum_{\lambda \in \Lambda(M)} \frac{\bm \alpha_{\lambda}^2(i)}{x-\lambda}.
\end{align*}
Thus, for each $\lambda \in \Lambda(M)$, we have
\begin{align*}
f_i(\lambda) = \bm \alpha_{\lambda}^2(i) \prod_{\mu \in \Lambda(M) \backslash \{\lambda\} } (\lambda-\mu).  
\end{align*}
On the other hand, we also have
\begin{align*}
\operatorname{Min}_M^\prime(x) = \sum_{\lambda \in \Lambda(M)} \prod_{\mu \in \Lambda(M) \backslash \{\lambda\} } (x-\mu).
\end{align*}
Therefore, $\bm \alpha_{\lambda}^2(i) = f_i(\lambda)/\operatorname{Min}_M^\prime (\lambda)$.
The statement of the lemma follows since, by definition, $\bm \alpha_{\lambda}(i)$ is nonnegative.
\end{proof}

	The next result generalises the fact that a unit eigenvector of a simple eigenvalue $\lambda$ can be expressed in terms of the angle vector of $\lambda$.
		See \cite{tao} for a survey on a related result.

\begin{lemma} \label{lem:anglesq_sqsum}
Let $M$ be a real symmetric matrix of order $n$ and let $\lambda$ be an eigenvalue of $M$ of multiplicity $e$.
Let $\mathbf{u}_1, \mathbf{u}_2, \dots, \mathbf{u}_{e}$ be an orthonormal basis for the eigenspace $\mathcal{E} (\lambda)$.
For all $i \in \{1,\dots,n\}$, we have that
\begin{align*}
\bm \alpha_{\lambda}^2(i) = \sum_{k=1}^{e} \mathbf{u}_k^2 (i).
\end{align*}
\end{lemma}

\begin{proof}
Firstly, we can write
\begin{align*}
P_{\lambda} = \mathbf{u}_1 \mathbf{u}_1^\transpose + \dots + \mathbf{u}_{e} \mathbf{u}_{e}^\transpose.
\end{align*}
Hence, for each $i \in \{1,\dots,n\}$ we have
\begin{align*}
P_{\lambda} \mathbf{e}_i = \mathbf{u}_1(i) \mathbf{u}_1 + \dots +  \mathbf{u}_{e}(i)\mathbf{u}_{e}.
\end{align*}
Therefore,
\begin{align*}
    \bm \alpha_{\lambda}^2(i) = \lVert P_{\lambda} \mathbf{e}_i \rVert^2 = \langle P_{\lambda} \mathbf{e}_i, P_{\lambda} \mathbf{e}_i \rangle = \sum_{k=1}^{e} \sum_{l=1}^{e} \mathbf{u}_k(i) \mathbf{u}_l(i) \langle \mathbf{u}_k, \mathbf{u}_l \rangle = \sum_{k=1}^{e} \mathbf{u}_k^2 (i)
\end{align*}
since $\mathbf{u}_1, \dots, \mathbf{u}_{e}$ are orthonormal.
\end{proof}

Two Seidel matrices $S_1$ and $S_2$ are called \textbf{switching-equivalent} if there exists a diagonal $\{\pm 1\}$-matrix $D$ such that $S_2 = DS_1D$.
Note that switching-equivalent matrices are similar and hence have the same characteristic polynomials.
By Lemma~\ref{lem:anglesq_formula}, switching-equivalent matrices also share the same angle vector for each eigenvalue.
Furthermore, a Seidel matrix $S$ can be ``switched''  (conjugated by a diagonal matrix of $\pm 1$s) so that the angle vector of a simple eigenvalue becomes a unit eigenvector.
\begin{corollary} \label{cor:alpha_eigenvec}
Let $S$ be a Seidel matrix and let $\lambda$ be a simple eigenvalue of $S$, i.e., with multiplicity $1$.
Then there exists a Seidel matrix $\hat{S}$ switching-equivalent to $S$ such that 
\[
\hat{S} \bm \alpha_\lambda = \lambda \bm \alpha_\lambda.
\]
\end{corollary}

Before we introduce compatibility for polynomials, we need the next result.
Denote by $\Sigma(M)$ the set of simple eigenvalues of $M$ and define the polynomial $$\operatorname{Sim}_M(x) \coloneqq \prod_{\lambda \in \Sigma(M)} (x-\lambda).$$
Note that $\operatorname{Sim}_M(x)$ is in $\mathbb Z[x]$.
    
\begin{lemma}[cf. {\cite[Lemma 4.3]{GSY21}}]\label{lem:integrality}
Let $M$ be an integer symmetric matrix of order $n$, let $\xi(x) \in \mathbb Z[x]$ be a factor of $\operatorname{Sim}_M(x)$, and set $q(x) = \operatorname{Min}_M(x)/\xi(x)$.
Denote by $\Xi$ the set of zeros of $\xi(x)$.
Then for all $i,j \in \{1,\dots,n\}$, there exists $\bm \delta \in \{ \pm 1 \}^{\Xi}$ such that
    \begin{equation*}
        \sum_{\lambda \in \Xi}  q(\lambda)\bm \delta(\lambda) \bm \alpha_\lambda(i) \bm \alpha_\lambda(j) \in \mathbb {Z}.
    \end{equation*}
\end{lemma}
\begin{proof}
For each $\lambda \in \Xi$, denote by $\mathbf u_\lambda$ a unit eigenvector for $\lambda$.
By Theorem~\ref{thm:spec_decomp}, we have 
$$q(M) =  \sum_{\lambda \in \Xi} q(\lambda)\mathbf u_\lambda \mathbf u_\lambda^\transpose.$$
The lemma then follows from Lemma~\ref{lem:anglesq_sqsum} together with the fact that $q(M)$ is an integer matrix. 
\end{proof}

By Lemma~\ref{lem:anglesq_formula}, there is a correspondence between the characteristic polynomial $\operatorname{Char}_{M[i]}(x)$ of a principal submatrix of $M$ and the set of angle vector entries $\{\bm \alpha_\lambda(i) \; : \; \lambda \in \Lambda(M) \}$.
Let $\mathfrak f(x) = \operatorname{Quo}_M(x) f(x)$ be a monic integer polynomial that interlaces $\operatorname{Char}_M(x)$.
For each $\lambda \in \Lambda(M)$, define the \textbf{angle} $\alpha_\lambda(\mathfrak f)$ of $\mathfrak f(x)$ with respect to $\lambda$ as
\[
\alpha_\lambda(\mathfrak f) \coloneqq \sqrt{\frac{f(\lambda)}{\operatorname{Min}_M^\prime(\lambda)}}.
\]
Note that, since $\mathfrak f(x)$ interlaces $\operatorname{Char}_M(x)$, we always have $f(\lambda)/\operatorname{Min}_M^\prime(\lambda) \geqslant 0$ when $\lambda \in \Lambda(M)$.
Furthermore, by Lemma~\ref{lem:anglesq_formula}, we have $\alpha_\lambda(\operatorname{Char}_{M[i]}) = \bm \alpha_\lambda(i)$ for all $\lambda \in \Lambda(M)$ and $i \in \{1,\dots,n\}$.
Now we can introduce the notion of \textit{compatibility} for polynomials.  
Let $\mathfrak g(x) \ne \mathfrak f(x)$ be a monic integer polynomial that interlaces $\operatorname{Char}_M(x)$.
In view of Lemma~\ref{lem:integrality}, given a factor $\xi(x) \in \mathbb Z[x]$ of $\operatorname{Sim}_M(x)$ with zero-set $\Xi$, we say that the polynomials $\mathfrak f(x)$ and $\mathfrak g(x)$ are $\xi$-\textbf{compatible} with respect to $M$ if there exists $\bm \delta \in \{\pm 1 \}^{\Xi}$ such that 
\begin{equation}
\label{eqn:integrality}
    \sum_{\lambda \in \Xi}  q(\lambda)\bm \delta(\lambda) \alpha_{\lambda}(\mathfrak f)\alpha_{\lambda}(\mathfrak g) \in \mathbb {Q}, \text{ where } q(x) = \operatorname{Min}_M(x)/\xi(x).
\end{equation}
If $\mathfrak f(x)$ and $\mathfrak g(x)$ are $\xi$-compatible with respect to $M$ for every factor $\xi(x) \in \mathbb{Z}[x]$ of $\operatorname{Sim}_M(x)$ then we say that $\mathfrak f(x)$ and $\mathfrak g(x)$ are \textbf{compatible} with respect to $M$.
When it is clear which matrix compatibility is taken with respect to, we merely say that $\mathfrak f(x)$ and $\mathfrak g(x)$ are compatible.
Note that our definition of compatibility differs from the corresponding definition in \cite{GSY21}.

We can strengthen the notion of compatibility for Seidel matrices in the following way.
Let $q(x) \in \mathbb Z[x]$ and let $S$ be a Seidel matrix of order $n$ odd.
Then, by \cite[Lemma 2.1]{GG18}, we have that the off-diagonal entries of $S^k$ are odd for all $k \geqslant 1$.
It follows that the parity of each off-diagonal entry of $q(S)$ is equal to the parity of $q(1)+q(0)$.
Let $\mathfrak f(x)$ and $\mathfrak g(x)$ be distinct monic integer polynomials that interlace $\operatorname{Char}_S(x)$.
Given a factor $\xi(x) \in \mathbb Z[x]$ of $\operatorname{Sim}_S(x)$ with zero-set $\Xi$, we say that the polynomials $\mathfrak f(x)$ and $\mathfrak g(x)$ are $\xi$-\textbf{Seidel-compatible} with respect to $S$ if there exists $\bm \delta \in \{\pm 1 \}^{\Xi}$ such that 
\begin{equation}
\label{eqn:paritySeidel}
    \sum_{\lambda \in \Xi}  q(\lambda)\bm \delta(\lambda) \alpha_{\lambda}(\mathfrak f)\alpha_{\lambda}(\mathfrak g) \equiv q(1)+q(0) \pmod 2, \text{ where } q(x) = \operatorname{Min}_S(x)/\xi(x).
\end{equation}
       
If $\mathfrak f(x)$ and $\mathfrak g(x)$ are $\xi$-Seidel-compatible with respect to $S$ for every factor $\xi(x) \in \mathbb{Z}[x]$ of $\operatorname{Sim}_S(x)$ then we say that $\mathfrak f(x)$ and $\mathfrak g(x)$ are \textbf{Seidel-compatible} with respect to $S$.
For a Seidel matrix $S$ of order $n$ even, we define Seidel-compatibility just the same as compatibility.
That is, $\mathfrak f(x)$ and $\mathfrak g(x)$ are $\xi$-Seidel-compatible with respect to $S$ if and only if $\mathfrak f(x)$ and $\mathfrak g(x)$ are $\xi$-compatible with respect to $S$.
For convenience with exposition below, we consider all polynomials to be Seidel-compatible with themselves.

In the sequel, we use repeatedly the following corollary of Lemma~\ref{lem:integrality}.

\begin{corollary}\label{cor:Seidel-compatible}
Let $S$ be a Seidel matrix of order $n$.
Let $X$ and $Y$ be principal submatrices of $S$ of order $n-1$.
Then $\operatorname{Char}_X(x)$ and $\operatorname{Char}_Y(x)$ are Seidel-compatible.
\end{corollary}
\begin{proof}
    If $\operatorname{Char}_X(x) = \operatorname{Char}_Y(x)$ then we are done.
    Hence, we suppose that $\operatorname{Char}_X(x) \ne \operatorname{Char}_Y(x)$.
    Note that there exist distinct $i, j \in \{1,\dots,n\}$ such that $X = S[i]$ and $Y = S[j]$.
    By definition and by Lemma~\ref{lem:anglesq_formula}, we have $\alpha_\lambda(\operatorname{Char}_X) = \bm \alpha_\lambda(i)$ and $\alpha_\lambda(\operatorname{Char}_Y) = \bm \alpha_\lambda(j)$.
    Let $\xi(x) \in \mathbb{Z}[x]$ be a factor of $\operatorname{Sim}_S(x)$ and let $q(x) = \operatorname{Min}_S(x) / \xi(x)$.
    By Lemma~\ref{lem:integrality}, $\operatorname{Char}_X(x)$ and $\operatorname{Char}_Y(x)$ are $\xi$-compatible with respect to $S$ and hence, they are compatible with respect to $S$.
    If $n$ is even then $\operatorname{Char}_X(x)$ and $\operatorname{Char}_Y(x)$ are Seidel-compatible by above definition.
    Otherwise, if $n$ is odd then, by Lemma~\ref{lem:integrality}, and by the parities of the off-diagonal entries of $q(S)$ (see  \cite[Lemma 2.1]{GG18}), there exists $\bm \delta \in \{ \pm 1 \}^{\Xi}$ such that  \eqref{eqn:paritySeidel} holds.
    Therefore, $\operatorname{Char}_X(x)$ and $\operatorname{Char}_Y(x)$ are Seidel-compatible.
\end{proof}

Empirically, we find that checking compatibility for polynomials in this paper can be quite computationally expensive, since it requires arithmetic in potentially high-degree number fields (as high as degree $120$ over $\mathbb{Q}$).
Next we establish some tools that will enable us to more efficiently check compatibility.

Let $M$ be an integer symmetric matrix and $\xi(x)$ be an irreducible factor of $\operatorname{Sim}_M(x)$ having zero-set $\Xi$.
Let $q(x) = \operatorname{Min}_M(x) / \xi(x)$,  $\lambda \in \Xi$, and suppose $\mathfrak f(x) = \operatorname{Quo}_M(x) f(x)$ and $\mathfrak g(x) = \operatorname{Quo}_M(x) g(x)$ are monic integer polynomials that interlace $\operatorname{Char}_M(x)$.
Note that we can write
\begin{equation}
\label{eqn:omegalambda}
q(\lambda) \bm \delta(\lambda) \alpha_{\lambda}(\mathfrak f)\alpha_{\lambda}(\mathfrak g) =  \bm \delta(\lambda) q(\lambda)  \frac{\sqrt{fg(\lambda)}}{|\operatorname{Min}_M^\prime(\lambda)|} = \frac{ \bm \delta^*(\lambda) \sqrt{fg(\lambda)}}{\xi^\prime(\lambda)}
\end{equation}
where $\bm \delta^*(\lambda) / \operatorname{Min}_M^\prime(\lambda) = \bm \delta(\lambda) / |\operatorname{Min}_M^\prime(\lambda)|$.
Since $f(\lambda)/\operatorname{Min}_M^\prime(\lambda) \geqslant 0$ and $g(\lambda)/\operatorname{Min}_M^\prime(\lambda) \geqslant 0$, we have $fg(\lambda) = f(\lambda) g(\lambda) \geqslant 0$.
Thus, the square root of $fg(\lambda)$ is a real number.

\begin{lemma} \label{lem:compatible_lagrange}
Let $M$ be an integer symmetric matrix and $\xi(x)$ be an irreducible factor of $\operatorname{Sim}_M(x)$ having zero-set $\Xi$.
Let $\mathfrak f(x) = \operatorname{Quo}_M(x) f(x)$ and $\mathfrak g(x) = \operatorname{Quo}_M(x) g(x)$ be distinct monic integer polynomials that interlace $\operatorname{Char}_M(x)$.
Suppose that there exists a polynomial $h(x) \in \mathbb{Q}[x]$ such that $h^2(\lambda) = fg(\lambda)$ for all $\lambda \in \Xi$.
Then $\mathfrak{f}(x)$ and $\mathfrak{g}(x)$ are $\xi$-compatible.
\end{lemma}

\begin{proof}
Let $\pi(x) \in \mathbb{Q}[x]$ be the unique polynomial of degree at most $|\Xi|-1$ such that $\pi(x) \equiv h(x) \mod \xi(x)$.
For each $\lambda \in \Xi$, we have that $\pi(\lambda) = h(\lambda) = \bm \delta^*(\lambda) \sqrt{fg(\lambda)}$ where $\bm \delta^* \in \{\pm 1 \}^\Xi$. 
Consider the $|\Xi|$ distinct interpolation points $(\lambda, \pi(\lambda))$ for $\lambda \in \Xi$.
Then we can write $\pi(x)$ as the unique interpolation polynomial in Lagrange form
$$
\pi(x) = \sum_{\lambda \in \Xi} \pi(\lambda) L_\lambda(x),
$$
where, for each $\lambda \in \Xi$, the polynomial $L_\lambda(x)$ is the Lagrange polynomial
$$
L_\lambda(x) = \prod_{\mu \in \Xi\backslash \{\lambda\}} \frac{x-\mu}{\lambda-\mu}.
$$
Let $\omega$ be the coefficient of $x^{|\Xi|-1}$ in $\pi(x)$ and by \eqref{eqn:omegalambda}, observe that
$$
\omega = \sum_{\lambda \in \Xi} \frac{\pi(\lambda)}{\xi^\prime(\lambda)} = 
\sum_{\lambda \in \Xi} \frac{\bm \delta^*(\lambda) \sqrt{fg(\lambda)}}{\xi^\prime(\lambda)} = \sum_{\lambda \in \Xi}  q(\lambda)\bm \delta(\lambda) \alpha_{\lambda}(\mathfrak f)\alpha_{\lambda}(\mathfrak g)
$$
where $q(x) = \operatorname{Min}_M(x) / \xi(x)$ and $\bm \delta^*(\lambda) / \operatorname{Min}_M^\prime(\lambda) = \bm \delta(\lambda) / |\operatorname{Min}_M^\prime(\lambda)|$ for all $\lambda \in \Xi$.
Since $\omega \in \mathbb Q$, the polynomials $\mathfrak{f}(x)$ and $\mathfrak{g}(x)$ are $\xi$-compatible.
\end{proof}

Let $K$ be the splitting field of $\xi(x)$ and let $\Gal(K/\mathbb{Q})$ be the Galois group of $K$ over $\mathbb{Q}$.
The Galois group $\Gal(K/\mathbb{Q})$ acts transitively on $\Xi$, the set of zeros of $\xi(x)$. 
Thus for all $\lambda,\mu \in \Xi$, there exists $\sigma \in \Gal(K/\mathbb{Q})$ such that $\sigma(\lambda) = \mu$.
We will use this fact in the proof of Proposition~\ref{pro:reducible_then_compatible} and Lemma~\ref{lem:irreducible_zero}.

\begin{proposition} \label{pro:reducible_then_compatible}
Let $M$ be an integer symmetric matrix and $\xi(x)$ be an irreducible factor of $\operatorname{Sim}_M(x)$ having zero-set $\Xi$.
Let $\mathfrak f(x) = \operatorname{Quo}_M(x) f(x)$ and $\mathfrak g(x) = \operatorname{Quo}_M(x) g(x)$ be distinct monic integer polynomials that interlace $\operatorname{Char}_M(x)$.
Let $\rho(x)$ be the minimal polynomial of $fg(\lambda)$ over $\mathbb{Q}$ for some $\lambda \in \Xi$ and suppose that $\rho(x^2)$ is reducible over $\mathbb{Q}$.
Then $\mathfrak{f}(x)$ and $\mathfrak{g}(x)$ are $\xi$-compatible.
\end{proposition}

\begin{proof}
Let $e$ be the degree of $\mathbb{Q}(\sqrt{fg(\lambda)})$ over $\mathbb{Q}$.
Note that the minimal polynomial of $\sqrt{fg(\lambda)}$ divides $\rho(x^2)$, which has degree $2 \deg \rho$.
This implies that $e$ divides $2 \deg \rho$ and moreover, $e < 2 \deg \rho$ since $\rho(x^2)$ is reducible over $\mathbb{Q}$.
On the other hand, $e \geqslant \deg \rho$ since $\mathbb{Q}(fg(\lambda)) \subseteq \mathbb{Q}(\sqrt{fg(\lambda)})$.
It follows that $e = \deg \rho$ and thus $\mathbb{Q}(\sqrt{fg(\lambda)}) = \mathbb{Q}(fg(\lambda)) \subseteq \mathbb{Q}(\lambda)$.
Hence, there exists a polynomial $h(x) \in \mathbb{Q}[x]$ such that $\sqrt{fg(\lambda)} = h(\lambda)$, which implies that $fg(\lambda) = h^2(\lambda)$.
Let $\mu \in \Xi$ and let $K$ be the splitting field of $\xi(x)$.
There exists $\sigma \in \Gal(K/\mathbb{Q})$ such that $\sigma(\lambda) = \mu$.
Hence, we obtain
$$
\sigma \left( fg \left( \lambda \right) \right) = \sigma \left( h^2(\lambda) \right) \implies fg \left( \mu \right) = h^2 (\mu).
$$
Therefore, we have $h^2(\lambda) = fg(\lambda)$ for all $\lambda \in \Xi$ and, by Lemma~\ref{lem:compatible_lagrange}, we conclude that $\mathfrak{f}(x)$ and $\mathfrak{g}(x)$ are $\xi$-compatible.
\end{proof}

Proposition~\ref{pro:reducible_then_compatible} provides us with a more computationally efficient method to show that two polynomials are compatible.
Furthermore, we can also use the above to check Seidel-compatibility.
Indeed, if $n$ is even, there is nothing more to do.
For $n$ odd, we can first construct the polynomial $h(x)$ from the proof of Proposition~\ref{pro:reducible_then_compatible} and the polynomial $\pi(x)$ in the proof of Lemma~\ref{lem:compatible_lagrange}.
Using \eqref{eqn:omegalambda}, we see that $\omega$ from the proof of Lemma~\ref{lem:compatible_lagrange} is the same as the left hand side of \eqref{eqn:paritySeidel}.
Thus, to check Seidel-compatibility, we can check if $\omega$ has the same parity as $q(1) + q(0)$, where $q(x) = \operatorname{Min}_M(x) / \xi(x)$.

\begin{example}
\label{ex:red2compat}
Suppose $S$ is a Seidel matrix with 
$$\operatorname{Char}_S(x) = (x+5)^{32}(x-9)^{13}(x-11)^2(x^2-21x+92).$$
Then we have
\begin{align*}
    \operatorname{Min}_S(x) &= (x+5)(x-9)(x-11)(x^2-21x+92), \\
    \operatorname{Quo}_S(x) &= (x+5)^{31}(x-9)^{12}(x-11), \\
    \operatorname{Sim}_S(x) &= x^2-21x+92.
\end{align*}
Let $f(x) = x^4-36x^3+454x^2-2356x+4241$ and $g(x) = x^4-36x^3+454x^2-2348x+4169$.
Then $\mathfrak f(x) = \operatorname{Quo}_S(x)f(x)$ and $\mathfrak g(x) = \operatorname{Quo}_S(x)g(x)$ both interlace $\operatorname{Char}_S(x)$.
Let $\lambda = (21-\sqrt{73})/2$ be a zero of $\operatorname{Sim}_S(x)$.
Then the minimal polynomial of $fg(\lambda)$ is $\rho(x) = x^2 - 10521x + 5308416$ and $\rho(x^2) = (x^2-123x+2304)(x^2+123x+2304)$.
By Proposition~\ref{pro:reducible_then_compatible}, $\mathfrak f(x)$ and $\mathfrak g(x)$ are compatible.
For $h(x) = 9x-33$ we have $fg(\lambda) = h^2(\lambda)$.
It follows that $\pi(x) = h(x) = 9x-33$ and $\omega = 9$.
Finally, since $q(1)+q(0) = 975$, where $q(x) = \frac{\operatorname{Min}_S(x)}{\operatorname{Sim}_S(x)}= x^3 - 15x^2 - x + 495$, we see that $\mathfrak f(x)$ and $\mathfrak g(x)$ are Seidel-compatible.
\end{example}

Next we develop tools to show that two polynomials are not Seidel-compatible.

\begin{lemma} \label{lem:irreducible_zero}
Let $M$ be an integer symmetric matrix and $\xi(x)$ be an irreducible factor of $\operatorname{Sim}_M(x)$ having zero-set $\Xi$ and splitting field $K$.
Let $\mathfrak f(x) = \operatorname{Quo}_M(x) f(x)$ and $\mathfrak g(x) = \operatorname{Quo}_M(x) g(x)$ be distinct monic integer polynomials that interlace $\operatorname{Char}_M(x)$.
Let $\rho(x)$ be the minimal polynomial of $fg(\lambda)$ over $\mathbb{Q}$ for some $\lambda \in \Xi$.
Suppose that $\deg \rho = |\Xi|$ and $\rho(x^2)$ is irreducible over $\mathbb{Q}$.
Further, suppose $\sqrt{fg(\lambda)} \not \in K$ and $q(x) = \operatorname{Min}_M(x) / \xi(x)$.
Then, for each $\bm \delta \in \{\pm 1 \}^{\Xi}$,
\[
\sum_{\lambda \in \Xi}  q(\lambda)\bm \delta(\lambda) \alpha_{\lambda}(\mathfrak f)\alpha_{\lambda}(\mathfrak g) \in \mathbb Q \implies \sum_{\lambda \in \Xi}  q(\lambda)\bm \delta(\lambda) \alpha_{\lambda}(\mathfrak f)\alpha_{\lambda}(\mathfrak g) =0.
\]
\end{lemma}

\begin{proof}
Let $\bm \delta \in \{\pm 1 \}^{\Xi}$ such that
\[
\sum_{\lambda \in \Xi}  q(\lambda)\bm \delta(\lambda) \alpha_{\lambda}(\mathfrak f)\alpha_{\lambda}(\mathfrak g) \in \mathbb Q,
\]
i.e., $\mathfrak{f}(x)$ and $\mathfrak{g}(x)$ are $\xi$-compatible.
For any $\lambda, \mu \in \Xi$ there exists $\sigma \in \Gal(K/\mathbb{Q})$ such that $\sigma(\lambda) = \mu$.
Since $\rho \left( fg(\lambda) \right) = 0$, we obtain $\rho\left( fg(\mu) \right) = 0$.
Since $\rho(x)$ is irreducible over $\mathbb{Q}$, we conclude that $\rho$ is the minimal polynomial of $fg(\lambda)$ over $\mathbb{Q}$ for all $\lambda \in \Xi$.

Next, for any $\lambda \in \Xi$ we have $\mathbb{Q}( fg(\lambda) ) \subseteq \mathbb{Q} (\lambda)$.
Since the degree of $\rho(x)$ is equal to $|\Xi|$ and $\rho(x)$ is the minimal polynomial of $fg(\lambda)$, we have $\mathbb{Q}( fg(\lambda) ) = \mathbb{Q} (\lambda)$.
Now suppose that there exist distinct $\mu, \nu \in \Xi$ such that $fg(\mu) = fg(\nu)$.
We can write $\mu = \pi (fg(\mu))$ for some $\pi(x) \in \mathbb{Q}[x]$, since $\mathbb{Q}( fg(\mu) ) = \mathbb{Q} (\mu)$.
Take $\sigma \in \Gal(K/\mathbb{Q})$ such that $\sigma(\mu) = \nu$, which yields us $\nu = \pi (fg(\nu))$.
However, this leads us to $\mu = \pi (fg(\mu)) = \pi (fg(\nu)) = \nu$, which is a contradiction.
It follows that
$$ \rho(x) = \prod_{\lambda \in \Xi} \left( x-fg(\lambda) \right). $$
Additionally, $K$ is the splitting field of $\rho(x)$.

Let $L$ be the splitting field of $\rho(x^2)$ over $\mathbb{Q}$, which contains $K$.
We have a tower of fields $\mathbb{Q} \subseteq K \subseteq L$ where $L$ is Galois over $\mathbb{Q}$ so $L$ is Galois over $K$.
Hence, if $\sigma \in \Gal(L/K)$ then $\sigma\left( \sqrt{fg(\lambda)} \right)$ is equal to either $\sqrt{fg(\lambda)}$ or $-\sqrt{fg(\lambda)}$ for each $\lambda \in \Xi$.
Since $K/\mathbb Q$ is a normal extension, if $\sqrt{fg(\lambda)}$ does not belong to $K$ for some $\lambda \in \Xi$ then $\sqrt{fg(\lambda)}$ does not belong to $K$ for all $\lambda \in \Xi$.

For a fixed $\lambda \in \Xi$, there exists $\sigma \in \Gal(L/K)$ such that $\sigma\left( \sqrt{fg(\lambda)} \right) = -\sqrt{fg(\lambda)}$.
Otherwise, if $\sqrt{fg(\lambda)}$ is fixed by all elements of $\Gal(L/K)$ then $\sqrt{fg(\lambda)} \in K$, which is a contradiction.
Using \eqref{eqn:omegalambda}, we write
\[
\omega = \displaystyle \sum_{\lambda \in \Xi} \omega_\lambda \text{ where }
\omega_\lambda = \displaystyle \frac{\bm \delta^*(\lambda) \sqrt{fg(\lambda)}}{\xi^\prime(\lambda)}
\]
and $\bm \delta^*(\lambda) / \operatorname{Min}_M^\prime(\lambda) = \bm \delta(\lambda) / |\operatorname{Min}_M^\prime(\lambda)|$ for all $\lambda \in \Xi$.
Take any $\sigma \in \Gal(L/K)$ such that $\sigma(\omega_\mu) = -\omega_\mu$ for some $\mu \in \Xi$.
Then we can partition the set $\Xi$ into two sets $I_+$ and $I_-$ such that $\sigma(\omega_\lambda) = \omega_\lambda$ for all $\lambda \in I_+$ and $\sigma(\omega_\mu) = -\omega_\mu$ for all $\mu \in I_-$.
This implies that $\sum_{\mu \in I_-} \omega_\mu = 0$ and $\sum_{\lambda \in I_+} \omega_\lambda = \omega$ since $\sigma(\omega) = \omega$.
We then apply the same procedure on $I_+$, where we take any $\sigma \in \Gal(L/K)$ such that $\sigma(\omega_\mu) = -\omega_\mu$ for some $\mu \in I_+$.
Since \(|\Xi|\) is finite, after finitely many steps on partitioning $I_+$, we arrive at the last subset $\mathcal{I}$ where the only possible partition is $\mathcal{I}_+ = \emptyset$ and $\mathcal{I}_- = \mathcal{I}$.
Therefore, we conclude that $\omega = \sum_{\lambda \in \Xi} \omega_\lambda = 0$.
\end{proof}
 
Now we have the following corollary of Lemma~\ref{lem:irreducible_zero}.

\begin{corollary} \label{cor:incompatible}
Let $S$ be a Seidel matrix of order $n$ odd, $\xi(x)$ be an irreducible factor of $\operatorname{Sim}_S(x)$ having zero-set $\Xi$, and let $q(x) = \operatorname{Min}_S(x)/\xi(x)$.
Let $\mathfrak f(x) = \operatorname{Quo}_S(x) f(x)$ and $\mathfrak g(x) = \operatorname{Quo}_S(x) g(x)$ be distinct monic integer polynomials that interlace $\operatorname{Char}_S(x)$.
Let $\rho(x)$ be the minimal polynomial of $fg(\lambda)$ over $\mathbb{Q}$ for some $\lambda \in \Xi$.
Suppose that $\deg \rho = |\Xi|$ and $\rho(x^2)$ is irreducible over $\mathbb{Q}$.
Let $G$ and $H$ be the Galois groups of $\rho(x)$ and $\rho(x^2)$ over $\mathbb{Q}$, respectively, and suppose that $|G| < |H|$.
If $q(1)+q(0)$ is odd then $\mathfrak{f}(x)$ and $\mathfrak{g}(x)$ are not Seidel-compatible.
\end{corollary}
\begin{proof}
Let $K \subseteq L$ be splitting fields of $\rho(x)$ and $\rho(x^2)$ over $\mathbb{Q}$, respectively.
If $K = L$ then $G \cong H$.
Since $|G| < |H|$, we conclude that $K$ is a proper subfield of $L$ and it follows that $\sqrt{fg(\lambda)}$ is not in $K$.
Otherwise, if $\sqrt{fg(\lambda)} \in K$ then $\sqrt{fg(\lambda)} \in K$ for all $\lambda \in \Xi$, which will imply that $K=L$.
Next, suppose that $\mathfrak f(x)$ and $\mathfrak g(x)$ are Seidel-compatible.
Then, by definition, there exists $\bm \delta \in \{\pm 1\}^\Xi$ such that
\[
\omega \coloneqq \sum_{\lambda \in \Xi}  q(\lambda)\bm \delta(\lambda) \alpha_{\lambda}(\mathfrak f)\alpha_{\lambda}(\mathfrak g) \equiv q(1) + q(0) \pmod 2.
\]
This implies that $\omega$ is rational.
By Lemma~\ref{lem:irreducible_zero}, we must have $\omega = 0$. 
Therefore, $q(1) + q(0)$ is even.
\end{proof}

\begin{example}
\label{ex:notCompatible}
Suppose $S$ is a Seidel matrix with 
$$\operatorname{Char}_S(x) = (x+5)^{32}(x-9)^{14}(x-11)(x^2-23x+116).$$
Then we have
\begin{align*}
    \operatorname{Min}_S(x) &= (x+5)(x-9)(x-11)(x^2-23x+116), \\
    \operatorname{Quo}_S(x) &= (x+5)^{31}(x-9)^{13}, \\
    \operatorname{Sim}_S(x) &= (x-11)(x^2-23x+116).
\end{align*}
Let $f(x) = x^4-38x^3+508x^2-2810x+5363$ and $g(x) = x^4-38x^3+508x^2-2802x+5291$.
Then $\mathfrak f(x) = \operatorname{Quo}_S(x)f(x)$ and $\mathfrak g(x) = \operatorname{Quo}_S(x)g(x)$ both interlace $\operatorname{Char}_S(x)$.
Let $\xi(x) = x^2-23x+116$ be an irreducible factor of $\operatorname{Sim}_S(x)$ and $\lambda = (23-\sqrt{65})/2$ be a zero of $\xi(x)$.
Then the minimal polynomial of $fg(\lambda)$ is $\rho(x) = x^2 - 11105x + 1433600$ and $\rho(x^2)$ is irreducible over $\mathbb{Q}$.
Furthermore, the Galois groups of $\rho(x)$ and $\rho(x^2)$ are $G=S_2$ and $H=D_4$, respectively.
Hence $2 = |G|< |H| = 8$.
Finally, $q(1)+q(0) = 975$, where $q(x) = \frac{\operatorname{Min}_S(x)}{\operatorname{Sim}_S(x)}= x^3 - 15x^2 - x + 495$.
Therefore, by Corollary~\ref{cor:incompatible}, the polynomials $\mathfrak f(x)$ and $\mathfrak g(x)$ are not Seidel-compatible.
\end{example}

Note that, in Corollary~\ref{cor:incompatible}, we cannot take the condition $|G| < |H|$ for granted, since for example, we have the polynomials $\rho_1(x) =  x^3 - 11x^2 + 27x - 13$, $\rho_2(x) =  x^4 - 14x^3 + 34x^2 - 14x + 1$, and $\rho_3(x) = x^4-14x^3+45x^4-29x+4$.
These polynomials have the property that $\rho_i(x^2)$ is irreducible over $\mathbb{Q}$ and the Galois group of $\rho_i(x)$ is isomorphic to that of $\rho_i(x^2)$ for each $i \in \{1,2,3\}$.
Furthermore, the polynomials $\rho_1(x^2)$, $\rho_2(x^2)$, and $\rho_3(x^2)$ have the Galois groups $S_3$, $D_4$, and $S_4$, respectively.

\section{Interlacing characteristic polynomials}
\label{sec:icp}

Our main approach for showing that a Seidel matrix $S$ having a certain spectrum does not exist is to consider the principal submatrices of $S$ and their characteristic polynomials. 
The next result is a condition on the sum of the characteristic polynomials of principal submatrices of a matrix.

\begin{theorem}[{\cite[Page 116]{thompson1968principal}}]\label{thm:sumofsubpoly}
Let $M$ be a real symmetric matrix of order $n$.
Then
\begin{equation}
    \label{eqn:thm}
    \sum_{i=1}^n \operatorname{Char}_{M[i]}(x)= \operatorname{Char}_M^\prime(x).
\end{equation}
\end{theorem}

Since every Seidel matrix $S$ of order $n$ has zero trace and the trace of $S^2$ is $n(n-1)$, we have the following.

		\begin{lemma}[{\cite[Lemma 5.4]{GreavesYatsyna19}}]\label{lem:coeff_of_interlacepoly}
		Let $S$ be a Seidel matrix of order $n$. 
		Suppose $S$ has minimal polynomial $\operatorname{Min}_S(x) = \sum_{i=0}^e a_i x^{e-i}$.
		Then, for all $j \in \{1,\dots,n\}$,
		\[
		\operatorname{Char}_{S[j]}(x)=\operatorname{Quo}_{S}(x)\sum_{i=0}^{e-1} b_i x^{e-1-i},
		\]
		where $b_0 = 1$, $b_1 = a_1$, $b_2 = a_2+n-1$, and $b_i \in \mathbb Z$ for $i \in \{3,\dots,e-1\}$.
		\end{lemma}

Let $p(x) \in \mathbb Z[x]$ be a monic totally-real polynomial of degree $n$ and suppose $S$ is a Seidel matrix such that $\operatorname{Char}_S(x) = p(x)$.
By Lemma~\ref{lem:coeff_of_interlacepoly}, for each $i\in \{1,\dots,n\}$, we have
$\operatorname{Char}_{S[i]}(x)=\operatorname{Quo}_S(x)\cdot f(x)$ for some polynomial $f(x)=\sum_{t=0}^{e-1} b_t x^{e-1-t}$ where
$b_0=1$, $b_1=a_1$, and $b_2=a_2+n-1$.
We want to find an exhaustive list of all possibilities for the polynomial $\operatorname{Char}_{S[i]}(x)$.

Define the polynomial $\operatorname{Rad}(p, x)$ derived from $p(x)$ as
\[
\operatorname{Rad}(p, x) \coloneqq \frac{p(x)}{\gcd(p(x),p^\prime (x))} .
\]
Note that $\operatorname{Min}_S(x) = \operatorname{Rad}(p, x)$ and $\operatorname{Quo}_{S}(x) = \gcd(p(x),p^\prime (x))$.

By Lemma~\ref{lem:shiftType}, the polynomial $\operatorname{Char}_{S[i]}(x-1)$ is weakly type 2 and is type 2 if $n-1$ is even.
By Lemma~\ref{lem:type_factor}, the polynomial $f(x-1)$ is also weakly type 2 and is type 2 if $n-1$ is even.
Thus, we need to find all totally-real, integer polynomials $f(x)=\sum_{t=0}^{e-1} b_t x^{e-1-t}$ with the following properties:
\begin{enumerate}[label=(\roman*)]
\item $b_0=1$, $b_1=a_1$, $b_2=a_2+n-1$,
\item $f(x)$ interlaces $\operatorname{Rad}(p, x)$,
\item $f(x-1)$ is weakly type 2 and is type 2 if $n-1$ is even,
\item $\gcd(p(x),p^\prime (x))\cdot f(x)$ is in a congruence class of $\mathcal{P}_{n-1,7}$, if $n-1$ is odd. (See Remark~\ref{rem:e7}.)
\end{enumerate}

Let $F$ be the set of polynomials satisfying these properties.
Each polynomial in the set 
$$\mathfrak F \coloneqq \{ \gcd(p(x),p^\prime (x))\cdot f(x) \; : \; f(x) \in F \}$$
is called an \textbf{interlacing characteristic polynomial} for $p(x)$.
In many of the lemmas below, we will need to find all the interlacing characteristic polynomials of certain polynomials.
To do this, we employ the polynomial generation algorithm of \cite[Section 2.3]{GSY21} to construct $\mathfrak F$.

\begin{lemma}
\label{lem:setB}
There does not exist a Seidel matrix $S$ whose characteristic polynomial is equal to any of the polynomials in $\mathcal B$ (from Lemma~\ref{lem:candpols}).
\end{lemma}
\begin{proof}
    Both polynomials in $\mathcal B$ do not have any interlacing characteristic polynomials. 
\end{proof}

The \textbf{coefficient vector} of a polynomial $h(x)=\sum_{t=0}^{n-1} c_{t} x^{n-1-t}$ of degree $n-1$ is defined to be the (row) vector $(c_0,c_1,\dots,c_{n-1})$.
Given a set $H$ of polynomials each of degree $n - 1$, the \textbf{coefficient matrix} $A(H)$ is defined as the $|H| \times n$ matrix whose rows are the coefficient vectors for each polynomial in $H$.
We write $\mathbf x \geqslant \mathbf 0$ to indicate that all entries of the vector $\mathbf x$ are nonnegative.
The polynomial equation \eqref{eqn:thm} can be viewed as a linear system: $\mathbf x^\transpose A = \mathbf b^\transpose$, where $\mathbf x \geqslant \mathbf 0$.
Indeed, for a real symmetric matrix $M$ of order $n$, let $X = \{ \operatorname{Char}_{M[i]}(x) \; : \; i \in \{1,\dots,n\}\}$ and let $\mathbf x$ be the vector indexed by elements of $X$ such that the entry indexed by $\mathfrak p (x) \in X$ equals the cardinality of the set $\{ i \in \{1,\dots,n\} \; : \; \operatorname{Char}_{M[i]}(x) = \mathfrak p (x) \}$.
If $A = A(X)$ is the coefficient matrix for the polynomials in $X$ and $\mathbf b$ is the coefficient vector for $\operatorname{Char}_M^\prime(x)$, then \eqref{eqn:thm} becomes $\mathbf x^\transpose A = \mathbf b^\transpose$.

Farkas' Lemma (see \cite{farkas} or \cite[Theorem 4.1]{GSY21}) allows us to demonstrate that there is no vector $\mathbf x \geqslant \mathbf 0$ satisfying $\mathbf x^\transpose A = \mathbf b^\transpose$, by finding a vector $\mathbf y \in \mathbb R^n$ such that $A\mathbf{y}\geqslant\mathbf{0}$ and $\mathbf{y}^\transpose \mathbf{b} <0$. 
We call such a vector $\mathbf y$ a \textbf{certificate of infeasibility} for the linear system $\mathbf{x}^\transpose A=\mathbf{b}^\transpose,\: \mathbf{x} \geqslant \mathbf{0}$.
Note that one can find certificates of infeasibility using linear programming techniques.

Consider the set $\mathfrak{F}$ of interlacing characteristic polynomials for $p(x)$.
By Theorem~\ref{thm:sumofsubpoly}, if there exists a Seidel matrix $S$ such that $\operatorname{Char}_S(x) = p(x)$ then there exist nonnegative integers $n_{\mathfrak f}$ for each $\mathfrak f(x) \in \mathfrak F$ such that
\begin{equation}\label{eq:sumofsubpoly_dim17}
\sum_{\mathfrak f(x) \in \mathfrak F} n_{\mathfrak f} \cdot \mathfrak f(x)= p^\prime(x).
\end{equation}

Let $\mathbf n$ be the vector indexed by $\mathfrak F$ whose $\mathfrak f(x)$-entry is $n_{\mathfrak f}$, for each $\mathfrak f(x) \in \mathfrak F$ and let $A = A(\mathfrak F)$.
Then $\mathbf n \geqslant \mathbf 0$ is a solution to the linear system $\mathbf{n}^\transpose A=\mathbf{b}^\transpose$ where $\mathbf{b}$ is the coefficient vector for $p^\prime(x)$.
We call the vector $\mathbf n$ an \textbf{interlacing configuration} for $\mathfrak F$.
Thus, to show that no Seidel matrix $S$ exists having $\operatorname{Char}_S(x) = p(x)$, it suffices to show that there does not exist an interlacing configuration for $\mathfrak F$.
Hence, it suffices to provide a certificate of infeasibility $\mathbf c$ for the linear system above.
We call $\mathbf c$ a \textbf{certificate of infeasibility} for $p(x)$.

\begin{lemma}
\label{lem:firstTable}
There does not exist a Seidel matrix $S$ whose characteristic polynomial is equal to any of the polynomials in Table~\ref{tab:164polydim17}.
\end{lemma}
\begin{proof}
    Each polynomial in Table~\ref{tab:164polydim17} has a certificate of infeasibility listed in Table~\ref{tab:164polydim17infeas}. 
\end{proof}

Most of the computation time used in this paper was dedicated to finding the interlacing characteristic polynomials in order to prove Lemma~\ref{lem:firstTable}.
We make the following remark about the computation time taken for Lemma~\ref{lem:firstTable}.

\begin{remark}
    The total time to compute the interlacing characteristic polynomials for each of the 164 polynomials of Table~\ref{tab:164polydim17} took about 6 hours and 17 minutes.
    The time taken varied dramatically for the 164 polynomials.  
    As we can see in Figure~\ref{fig:timeTaken164}, there are six outliers, which correspond to the indices 65, 87, 88, 90, 101, and 112 in Table~\ref{tab:164polydim17}.
    The vast majority of the computation time was spent on finding the interlacing characteristic polynomials for these six candidate characteristic polynomials.
\end{remark}
\begin{figure}
    \centering
   \begin{tikzpicture}

\begin{groupplot}[
       group style={
       group name=plot,
       group size=2 by 1,
       xlabels at=edge bottom,
       ylabels at=edge left,
       horizontal sep=0pt,
       vertical sep=0pt,
       /pgf/bar width=10pt},
ylabel=Time taken (seconds),
major x tick style=transparent,
ybar= \pgflinewidth,
ymin=0,
x axis line style={opacity=0},
x tick label style={rotate=0, anchor=center, xshift=0.005cm, scale=1},
xmin=0, xmax=164, 
xtick={0,20,40,60,80,100,120,140,160},
xlabel=Index (Table~\ref{tab:164polydim17}),
ymajorgrids=true,
grid style=dotted,
nodes near coords,
scale only axis,
point meta=explicit symbolic,
enlarge x limits = {abs=1},
cycle list={
  draw=blue,thick,fill=blue,fill opacity=0.6,nodes near coords style={blue!60}\\
  draw=orange,thick,fill=orange,fill opacity=0.6,nodes near coords style={orange}\\
  },
legend columns=-1,
]
\nextgroupplot[
     ytick pos=left
      ]
\addplot [mark=none, ybar interval=0.5]
	coordinates {
( 1 , 0.140 )
( 2 , 0.140 )
( 3 , 14.750 )
( 4 , 1.740 )
( 5 , 14.680 )
( 6 , 14.830 )
( 7 , 14.840 )
( 8 , 14.770 )
( 9 , 14.770 )
( 10 , 14.880 )
( 11 , 1.720 )
( 12 , 14.830 )
( 13 , 14.940 )
( 14 , 0.140 )
( 15 , 95.140 )
( 16 , 1.690 )
( 17 , 14.610 )
( 18 , 14.800 )
( 19 , 1.290 )
( 20 , 94.580 )
( 21 , 14.720 )
( 22 , 14.660 )
( 23 , 14.730 )
( 24 , 14.890 )
( 25 , 0.170 )
( 26 , 14.610 )
( 27 , 14.590 )
( 28 , 0.610 )
( 29 , 473.770 )
( 30 , 94.930 )
( 31 , 14.330 )
( 32 , 14.260 )
( 33 , 469.070 )
( 34 , 0.630 )
( 35 , 14.790 )
( 36 , 56.830 )
( 37 , 14.820 )
( 38 , 94.220 )
( 39 , 1.640 )
( 40 , 457.100 )
( 41 , 468.560 )
( 42 , 14.940 )
( 43 , 3.760 )
( 44 , 10.080 )
( 45 , 96.570 )
( 46 , 15.050 )
( 47 , 57.930 )
( 48 , 96.880 )
( 49 , 97.090 )
( 50 , 14.860 )
( 51 , 96.220 )
( 52 , 14.660 )
( 53 , 474.300 )
( 54 , 0.070 )
( 55 , 57.380 )
( 56 , 95.790 )
( 57 , 14.840 )
( 58 , 29.990 )
( 59 , 10.050 )
( 60 , 57.660 )
( 61 , 95.410 )
( 62 , 95.950 )
( 63 , 14.860 )
( 64 , 98.300 )
( 65 , 1973.350 )
( 66 , 471.260 )
( 67 , 56.500 )
( 68 , 14.460 )
( 69 , 57.470 )
( 70 , 57.650 )
( 71 , 476.850 )
( 72 , 93.920 )
( 73 , 93.580 )
( 74 , 0.390 )
( 75 , 244.500 )
( 76 , 56.490 )
( 77 , 3.610 )
( 78 , 56.320 )
( 79 , 56.970 )
( 80 , 94.340 )
( 81 , 250.870 )
( 82 , 57.150 )
( 83 , 56.360 )
( 84 , 56.610 )
( 85 , 471.160 )
( 86 , 0.290 )
( 87 , 1882.780 )
( 88 , 1813.660 )
( 89 , 91.280 )
( 90 , 1817.470 )
( 91 , 3.520 )
( 92 , 1.210 )
( 93 , 54.470 )
( 94 , 9.490 )
( 95 , 90.860 )
( 96 , 0.370 )
( 97 , 237.030 )
( 98 , 54.560 )
( 99 , 54.560 )
( 100 , 112.130 )
( 101 , 1813.480 )
( 102 , 0.290 )
( 103 , 1.250 )
( 104 , 5.510 )
( 105 , 54.350 )
( 106 , 90.870 )
( 107 , 28.380 )
( 108 , 54.350 )
( 109 , 237.570 )
( 110 , 54.480 )
( 111 , 0.680 )
( 112 , 1816.240 )
( 113 , 0.330 )
( 114 , 5.460 )
( 115 , 54.660 )
( 116 , 28.490 )
( 117 , 28.310 )
( 118 , 54.670 )
( 119 , 54.870 )
( 120 , 872.360 )
( 121 , 1.950 )
( 122 , 112.240 )
( 123 , 54.660 )
( 124 , 55.000 )
( 125 , 54.550 )
( 126 , 11.720 )
( 127 , 872.670 )
( 128 , 1.990 )
( 129 , 0.200 )
( 130 , 54.520 )
( 131 , 237.720 )
( 132 , 54.640 )
( 133 , 11.590 )
( 134 , 111.770 )
( 135 , 0.900 )
( 136 , 367.540 )
( 137 , 0.260 )
( 138 , 0.010 )
( 139 , 0.780 )
( 140 , 54.400 )
( 141 , 2.620 )
( 142 , 11.660 )
( 143 , 28.390 )
( 144 , 0.080 )
( 145 , 0.000 )
( 146 , 0.440 )
( 147 , 2.520 )
( 148 , 28.290 )
( 149 , 3.930 )
( 150 , 11.730 )
( 151 , 54.500 )
( 152 , 11.920 )
( 153 , 0.950 )
( 154 , 0.070 )
( 155 , 11.770 )
( 156 , 0.170 )
( 157 , 0.020 )
( 158 , 2.510 )
( 159 , 0.030 )
( 160 , 0.450 )
( 161 , 0.190 )
( 162 , 0.060 )
( 163 , 0.060 )
( 164 , 0.440 )};
\end{groupplot}
\end{tikzpicture}
    \caption{A bar chart to display the time taken to compute the interlacing characteristic polynomial for each of the 164 polynomials in Lemma~\ref{lem:firstTable}.}
    \label{fig:timeTaken164}
\end{figure}
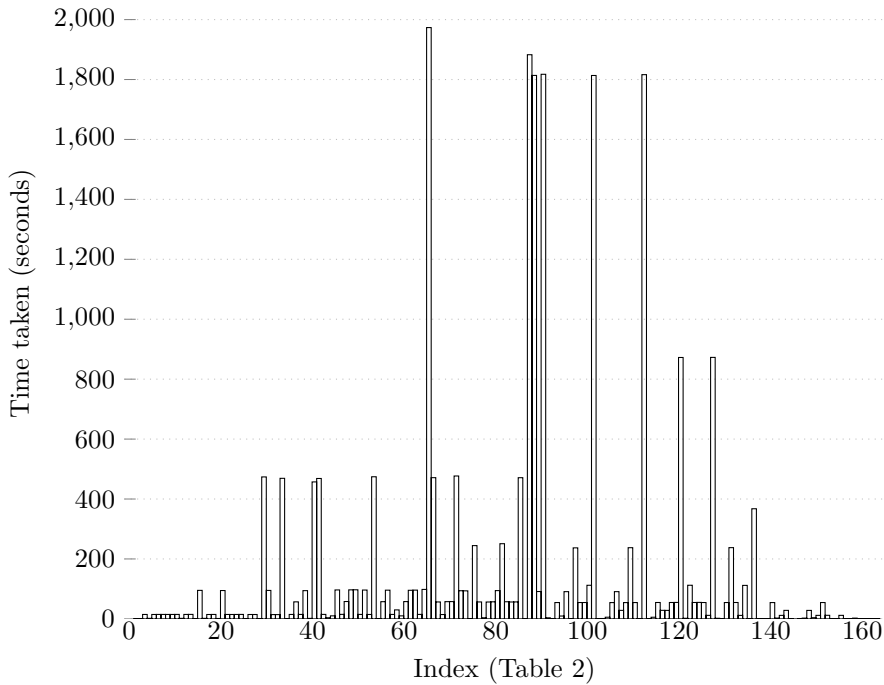

Suppose the linear system $\mathbf{n}^\transpose A=\mathbf{b}^\transpose$ has at least one nonnegative real solution.
Then a subset $\mathfrak W \subset \mathfrak F$ is called a \textbf{warranted} subset of interlacing characteristic polynomials if there is no nonnegative solution $\mathbf n$ to the subsystem  $\mathbf{n}^\transpose B=\mathbf{b}^\transpose$, where the matrix $B$ is obtained from $A$ by removing the rows corresponding to the polynomials in  $\mathfrak W$.
If $\mathfrak W$ consists of a single polynomial $\mathfrak w(x)$ then we call $\mathfrak w(x)$ a \textbf{warranted} interlacing characteristic polynomial.
Equivalently, an interlacing characteristic polynomial $\mathfrak w(x) \in \mathfrak F$ is called warranted if the $\mathfrak w(x)$-entry of every interlacing configuration for $\mathfrak F$ is positive.
We can show that $\mathfrak W \subset \mathfrak F$ is warranted by providing a certificate of infeasibility $\mathbf c$ for $\mathfrak F \backslash \mathfrak W$.
We call such a $\mathbf c$ a \textbf{certificate of warranty} for the subset $\mathfrak W$.
Equivalently, the entries of $ A\mathbf c$ that correspond to $\mathfrak W$ are negative, while the rest of the entries of $A\mathbf c$, which correspond to the entries of $B\mathbf c$, are nonnegative.

Once we have a warranted subset of interlacing characteristic polynomials $\mathfrak W \subset \mathfrak F$, we know that, for any interlacing configuration $\mathbf n$ for $\mathfrak F$, the entry $n_{\mathfrak f} \geqslant 1$ for some $\mathfrak f \in \mathfrak W$.
Furthermore, by Corollary~\ref{cor:Seidel-compatible}, if $n_{\mathfrak f} > 0$ then $\mathfrak f$ must be Seidel-compatible with some polynomials in the warranted subset $\mathfrak W$.
Let $\mathfrak C \subset \mathfrak F$ be the subset of polynomials that are Seidel-compatible with at least one polynomial from $\mathfrak W$.
Instead of looking for solutions to \eqref{eq:sumofsubpoly_dim17}, we need only restrict our attention to solving the system
\begin{equation}\label{eq:sumofsubpoly_dim17Compat}
\sum_{\mathfrak f \in \mathfrak C } n_{\mathfrak f} \cdot \mathfrak f(x)=p^\prime(x).
\end{equation}
Similar to the above, to show that no Seidel matrix $S$ exists having $\operatorname{Char}_S(x) = p(x)$, such that $\mathfrak W$ is a warranted subset of interlacing characteristic polynomials for $p(x)$, it suffices to provide a certificate of infeasibility $\mathbf c$ for the linear system \eqref{eq:sumofsubpoly_dim17Compat}.
We call $\mathbf c$ a \textbf{certificate of infeasibility} for $p(x)$ with respect to $\mathfrak W$.

Given a putative characteristic polynomial for a Seidel matrix, we first find its interlacing characteristic polynomials.
If no certificate of infeasibility exists, then we look for warranted polynomials or warranted subsets consisting of two polynomials.
We find warranted subsets by removing them from the set of interlacing characteristic polynomials and then finding a certificate of infeasibility for the remaining subset of interlacing characteristic polynomials.

\begin{lemma}
\label{lem:secondTable}
There does not exist a Seidel matrix $S$ whose characteristic polynomial is equal to any of the 11 polynomials in Table~\ref{tab:11polydim17}.
\end{lemma}
\begin{proof}
For each polynomial $p(x)$ in Table~\ref{tab:11polydim17}, the corresponding linear system has at least one nonnegative real solution.
In Table~\ref{tab:11polydim17warrant}, we list two warranted interlacing characteristic polynomials $\mathfrak f(x)$ and $\mathfrak g(x)$ for each $p(x)$ and their certificates of warranty.
However, for each $p(x)$ we have that $\mathfrak f(x)$ is not Seidel-compatible with $\mathfrak g(x)$, which contradicts Corollary~\ref{cor:Seidel-compatible}.
\end{proof}

The computations required in Lemma~\ref{lem:secondTable} (to compute the interlacing characteristic polynomials and check Seidel-compatibility) take about 2 minutes and 5 seconds in total.

\begin{lemma}
\label{lem:thirdTable}
There does not exist a Seidel matrix $S$ whose characteristic polynomial is equal to any of the eight polynomials in Table~\ref{tab:8polydim17}.
\end{lemma}
\begin{proof}
For each polynomial in Table~\ref{tab:8polydim17} there exists a warranted polynomial $\mathfrak f(x)$, listed in Table~\ref{tab:8polydim17warrant} together with its certificate of warranty.
Next we find all interlacing characteristic polynomials that are Seidel-compatible with $\mathfrak f(x)$ and we find that there is no solution to \eqref{eq:sumofsubpoly_dim17Compat}.
The Seidel-compatible interlacing characteristic polynomials and the certificate of infeasibility with respect to $\mathfrak f(x)$ are listed in Table~\ref{tab:8polydim17compat}. 
\end{proof}

The computations required in Lemma~\ref{lem:thirdTable} (to compute the interlacing characteristic polynomials and Seidel-compatible subsets) take about 2 minutes and 16 seconds in total.

In the remainder of this paper we are solely occupied with ruling out the remaining set $\mathcal C$ of $8$ polynomials from Lemma~\ref{lem:candpols}.
The total time required to compute the interlacing characteristic polynomials and Seidel-compatible subsets for the remainder of this paper's results takes less than five minutes on a modern PC.

\begin{lemma} \label{lem:last8_quad92}
There does not exist a Seidel matrix $S$ with characteristic polynomial
$$
\operatorname{Char}_S(x) = (x+5)^{32} (x-9)^{13} (x-11)^2 (x^2-21x+92).
$$
\end{lemma}

\begin{proof}
Suppose a Seidel matrix $S$ has characteristic polynomial
$$
\operatorname{Char}_S(x) = (x+5)^{32} (x-9)^{13} (x-11)^2 (x^2-21x+92).
$$
There are 51 interlacing characteristic polynomials and the set
\begin{align*}
\{ & (x+5)^{31} (x-9)^{12} (x-11) (x^4-36x^3+454x^2-2356x+4241),\\
& (x+5)^{31} (x-9)^{12} (x-11)^2 (x^3-25x^2+179x-379) \}.
\end{align*}
is warranted with certificate $(-17507602, 0, 0, -9281, -1031)$.
From the 51 interlacing characteristic polynomials, we keep only those polynomials that are Seidel-compatible with at least one of the polynomials in the set above.
As a result, we are left with five polynomials:
\begin{align*}
    \mathfrak{f}_1(x) &= (x+5)^{31} (x-9)^{12} (x-11) (x^4-36x^3+454x^2-2364x+4281), \\
    \mathfrak{f}_2(x) &= (x+5)^{31} (x-9)^{12} (x-11) (x^4-36x^3+454x^2-2356x+4241), \\
    \mathfrak{f}_3(x) &= (x+5)^{31} (x-9)^{12} (x-11)^2 (x^3-25x^2+179x-379), \\
    \mathfrak{f}_4(x) &= (x+5)^{31} (x-9)^{14} (x-11) (x^2-18x+49), \\
    \mathfrak{f}_5(x) &= (x+5)^{31} (x-3) (x-9)^{13} (x-11) (x^2-24x+139).
\end{align*}
These five polynomials are pairwise Seidel-compatible and we find that there are seven possible interlacing configurations $(n_{\mathfrak f_1},n_{\mathfrak f_2},n_{\mathfrak f_3},n_{\mathfrak f_4},n_{\mathfrak f_5})$:
\begin{align*}
    & (2,43,1,3,0), (6,34,7,2,0), (6,37,4,1,1), (6,40,1,0,2), \\
    & (10,25,13,1,0), (10,28,10,0,1), (14,16,19,0,0).
\end{align*}
Let $(\lambda_1,\dots,\lambda_5)=\left((21-\sqrt{73})/2,(21+\sqrt{73})/2,11,-5,9\right)$ be a 5-tuple of the five distinct eigenvalues of $S$.
Firstly, we compute the following angles
$$ 
\begin{bmatrix}
\alpha_{\lambda_1}^2(\mathfrak f_1) & \alpha_{\lambda_2}^2(\mathfrak f_1) & \alpha_{\lambda_3}^2(\mathfrak f_1)  \\
\alpha_{\lambda_1}^2(\mathfrak f_2) & \alpha_{\lambda_2}^2(\mathfrak f_2) & \alpha_{\lambda_3}^2(\mathfrak f_2)  \\
\alpha_{\lambda_1}^2(\mathfrak f_3) & \alpha_{\lambda_2}^2(\mathfrak f_3) & \alpha_{\lambda_3}^2(\mathfrak f_3)  \\
\alpha_{\lambda_1}^2(\mathfrak f_4) & \alpha_{\lambda_2}^2(\mathfrak f_4) & \alpha_{\lambda_3}^2(\mathfrak f_4)  \\
\alpha_{\lambda_1}^2(\mathfrak f_5) & \alpha_{\lambda_2}^2(\mathfrak f_5) & \alpha_{\lambda_3}^2(\mathfrak f_5)
\end{bmatrix}
= 
\begin{bmatrix}
(949+107\sqrt{73})/97236 & (949-107\sqrt{73})/97236 & 1/9 \\
(803-29\sqrt{73})/48618 & (803+29\sqrt{73})/48618 & 1/36 \\
(511-5\sqrt{73})/16206 & (511+5\sqrt{73})/16206 & 0 \\
(3869+385\sqrt{73})/48618 & (3869-385\sqrt{73})/48618 & 7/36  \\
(6059+427\sqrt{73})/48618 & (6059-427\sqrt{73})/48618 & 1/9
\end{bmatrix}.
$$
Note that
$$
(S+5I)(S-9I)=\frac{83-17\sqrt{73}}{2} \mathbf{u} \mathbf{u}^\transpose + \frac{83+17\sqrt{73}}{2} \mathbf{v} \mathbf{v}^\transpose + 32 P_{\lambda_3}
$$
where $\mathbf{u}$ and $\mathbf{v}$ are unit eigenvectors corresponding to the eigenvalues $\lambda_1$ and $\lambda_2$, respectively and $P_{\lambda_3}$ is the orthogonal projection onto $\mathcal{E}(\lambda_3)$.
Clearly, the matrix $(S+5I)(S-9I)$ must be an integer matrix.
For each of the seven possible interlacing configurations $(n_{\mathfrak f_1},n_{\mathfrak f_2},n_{\mathfrak f_3},n_{\mathfrak f_4},n_{\mathfrak f_5})$, we observe that both $n_{\mathfrak f_2}$ and $n_{\mathfrak f_3}$ are positive, while we also have that $\alpha_{\lambda_3}(\mathfrak{f}_3)=0$.
Suppose that the interlacing configuration is one of the seven possibilities.
Let $i,j \in \{1, \dots, 49\}$ such that $\operatorname{Char}_{S[i]}(x) = \mathfrak{f}_2(x)$ and $\operatorname{Char}_{S[j]}(x) = \mathfrak{f}_3(x)$.
Since $\alpha_{\lambda_3}(\mathfrak{f}_3)=0$, by Lemma~\ref{lem:anglesq_sqsum}, we obtain
\begin{equation} \label{eqn:59}
    (S+5I)(S-9I)_{i,j} = \pm \left( \frac{83-17\sqrt{73}}{2} \alpha_{\lambda_1}(\mathfrak{f}_2) \alpha_{\lambda_1}(\mathfrak{f}_3) \pm \frac{83+17\sqrt{73}}{2} \alpha_{\lambda_2}(\mathfrak{f}_2) \alpha_{\lambda_2}(\mathfrak{f}_3) \right).
\end{equation}
However, the right hand side of \eqref{eqn:59} is equal to either $\pm 65\sqrt{73}/219$ or $\pm 11/3$.
This contradicts the condition that $(S+5I)(S-9I)$ is an integer matrix.
\end{proof}

\section{Angle vectors with entries equal to 0}
\label{sec:angle0}

Before we can rule out the remaining seven polynomials from Lemma~\ref{lem:candpols}, we need to introduce some extra tools.
We use $O$ and $J$ to denote the zero matrix and all-ones matrix respectively, the dimensions of these matrices should be clear from the context.
We will write $J_n$ to denote the all-ones matrix of order $n$.
Let $M$ be a real symmetric matrix of order $n$.
For $k \in \{1,\dots,n\}$ and a subset $\{i_1,\dots,i_k\} \subset \{1,\dots,n\}$, denote by $M[i_1,\dots,i_k]$ the principal submatrix of $M$ obtained by deleting its rows and columns indexed by the set $\{i_1,\dots,i_k\}$.

\begin{lemma} \label{lem:zero_block_eigenvecs}
Let $M$ be a real symmetric matrix of order $n$ with eigenvalue $\lambda$.
Let $\mathbf{v}_1, \dots, \mathbf{v}_m$ be linearly independent $\lambda$-eigenvectors of $M$ and let $k < n$ be a positive integer.
Suppose that $\mathbf{v}_j (i) = 0$ for all $i \in \{1,\dots,k\}$ and for all $j \in \{1,\dots,m\}$.
For each $j \in \{1,\dots,m\}$, remove the first $k$ entries of $\mathbf{v}_j$ to obtain the vector $\mathbf{v}_j^*$.
Then $\mathbf{v}_1^*, \dots, \mathbf{v}_m^*$ are linearly independent $\lambda$-eigenvectors of the principal submatrix $M[1,\dots,k]$.
\end{lemma}

\begin{proof}
Clearly, the vectors $\mathbf{v}_1^*, \dots, \mathbf{v}_m^*$ retain linear independence from $\mathbf{v}_1, \dots, \mathbf{v}_m$.
Let $N$ be the $n \times m$ matrix whose $i$th column is $\mathbf{v}_i$ and let $N^*$ be the $(n-k) \times m$ matrix whose $i$th column is $\mathbf{v}_i^*$.
Observe that, we have $MN = \lambda N$.
Furthermore, writing $M$ and $N$ as block matrices yields
$$
\begin{bmatrix}
O \\
\lambda N^*
\end{bmatrix} = \lambda N = MN = 
\begin{bmatrix}
A & B \\
B^\transpose & M[1,\dots,k]
\end{bmatrix}
\begin{bmatrix}
O \\
N^*
\end{bmatrix}
= 
\begin{bmatrix}
BN^* \\
M[1,\dots,k]N^*
\end{bmatrix}.
$$
It follows that $M[1,\dots,k]N^* = \lambda N^*$, as required.
\end{proof}

Combining Lemma~\ref{lem:anglesq_formula}, \ref{lem:anglesq_sqsum}, and \ref{lem:zero_block_eigenvecs}, we have the following lemma.
\begin{lemma} \label{lem:extracting}
Let $M$ be a real symmetric matrix of order $n$ with eigenvalue $\lambda$ of multiplicity $m$.
Suppose there exists a $k$-subset $\mathcal{I}$ of $\{1,\dots,n\}$ such that for each $i \in \mathcal{I}$, the principal submatrix $M[i]$ has an eigenvalue $\lambda$ of multiplicity $m$ or $m+1$.
Then $M$ has a principal submatrix of order $n-k$ with eigenvalue $\lambda$ of multiplicity at least $m$.
\end{lemma}

\begin{proof}
If for some $i \in \{1,\dots,n\}$, the principal submatrix $M[i]$ has eigenvalue $\lambda$ of multiplicity $m$ or $m+1$, then we have the corresponding angle $\bm \alpha_{\lambda}(i) = 0$ by Lemma~\ref{lem:anglesq_formula}.
Without loss of generality, suppose that $\mathcal{I} = \{n-k+1,\dots,n\}$.
For each $i \in \mathcal{I}$, the principal submatrix $M[i]$ has eigenvalue $\lambda$ of multiplicity $m$ or $m+1$.
Let $\mathbf{u}_1, \mathbf{u}_2, \dots, \mathbf{u}_{m}$ be an orthonormal basis of the eigenspace $\mathcal{E} (\lambda)$.
By Lemma~\ref{lem:anglesq_sqsum}, we have that
\begin{align*}
\sum_{j=1}^{m} \mathbf{u}_j^2 (i) =     \bm \alpha_{\lambda}^2(i) = 0 \implies \mathbf{u}_1 (i) = \mathbf{u}_2 (i) = \dots =\mathbf{u}_{m} (i) = 0
\end{align*}
for all $i \in \mathcal{I}$.
Let $\mathbf{u}_1^*,$ $\mathbf{u}_2^*,$ $\dots,$ $\mathbf{u}_{m}^* \in \mathbb{R}^{n-k}$ be vectors such that for all $i \in \{1,\dots, n-k\}$ and $j \in \{1,\dots,m\}$, we have $\mathbf{u}_j^* (i) = \mathbf{u}_j(i)$.
By Lemma~\ref{lem:zero_block_eigenvecs}, the vectors $\mathbf{u}_1^*,$ $\mathbf{u}_2^*,$ $\dots,$ $\mathbf{u}_{m}^*$ are linearly independent $\lambda$-eigenvectors of the principal submatrix $M^* = M[n-k+1, \dots,n]$.
Therefore, the principal submatrix $M^*$ has eigenvalue $\lambda$ of multiplicity at least $m$.
\end{proof}

\begin{lemma} \label{lem:last8_fiveint}
There does not exist a Seidel matrix $S$ with characteristic polynomial
$$
\operatorname{Char}_S(x) = (x+5)^{32} (x-7) (x-9)^{14} (x-12) (x-15).
$$
\end{lemma}

\begin{proof}
Suppose a Seidel matrix $S$ has characteristic polynomial
$$
\operatorname{Char}_S(x) = (x+5)^{32} (x-7) (x-9)^{14} (x-12) (x-15).
$$
There are 36 interlacing characteristic polynomials and one of them is warranted:
$$\mathfrak{f}(x) = (x+5)^{31} (x-7) (x-9)^{13} (x^3-31x^2+291x-757)$$
with certificate of warranty $(-60980320, 0, 0, -26357, -2396)$.
Out of the 36 interlacing characteristic polynomials, only three are Seidel-compatible with $\mathfrak{f}(x)$:
\begin{align*}
    \mathfrak{f}_1(x) &= (x+5)^{31} (x-7) (x-9)^{13} (x-15) (x^2-16x+51), \\
    \mathfrak{f}_2(x) &= (x+5)^{31} (x-7) (x-9)^{13} (x^3-31x^2+291x-757), \\
    \mathfrak{f}_3(x) &= (x+5)^{31} (x-9)^{14} (x^3-29x^2+247x-531).
\end{align*}
These three polynomials are pairwise Seidel-compatible and we find that there is only one possible interlacing configuration $(n_{\mathfrak f_1},n_{\mathfrak f_2},n_{\mathfrak f_3}) = (9,36,4)$.
By Lemma~\ref{lem:extracting}, we conclude that there exists a Seidel matrix $\widehat{S}$ of order $4$ with eigenvalue $7$.
This is a contradiction since $\tr \widehat{S}^2=12< 7^2$. 
\end{proof}

\begin{lemma} \label{lem:last8_sixint_ev15}
There does not exist a Seidel matrix $S$ with characteristic polynomial
$$
\operatorname{Char}_S(x) = (x+5)^{32} (x-7) (x-8) (x-9)^{12} (x-11)^2 (x-15).
$$
\end{lemma}

\begin{proof}
Suppose a Seidel matrix $S$ has characteristic polynomial
$$
\operatorname{Char}_S(x) = (x+5)^{32} (x-7) (x-8) (x-9)^{12} (x-11)^2 (x-15).
$$
There are 56 interlacing characteristic polynomials and the set
\begin{align*}
\{ & (x+5)^{31} (x-7) (x-9)^{11} (x-11) (x^4-38x^3+512x^2-2874x+5583),\\
& (x+5)^{31} (x-9)^{11} (x-11) (x^5-45x^4+778x^3-6442x^2+25397x-37673) \}.
\end{align*}
is warranted with certificate $(-114666126726, 0, 0, -29336361, -3575825, -437857)$.
From the 56 interlacing characteristic polynomials, we keep only those polynomials that are Seidel-compatible with at least one of the polynomials in the set above.
As a result, we are left with eight polynomials:
\begin{align*}
    \mathfrak{f}_1(x) &= (x+5)^{31} (x-7) (x-9)^{12} (x-11) (x-15) (x^2-14x+41), \\
    \mathfrak{f}_2(x) &= (x+5)^{31} (x-7) (x-9)^{11} (x-11) (x^4-38x^3+512x^2-2874x+5583), \\
    \mathfrak{f}_3(x) &= (x+5)^{31} (x-9)^{12} (x-11) (x^4-36x^3+454x^2-2364x+4201), \\
    \mathfrak{f}_4(x) &= (x+5)^{31} (x-9)^{12} (x-11) (x^4-36x^3+454x^2-2356x+4129), \\
    \mathfrak{f}_5(x) &= (x+5)^{31} (x-9)^{11} (x-11) (x^5-45x^4+778x^3-6442x^2+25397x-37673), \\
    \mathfrak{f}_6(x) &= (x+5)^{31} (x-9)^{12} (x-11) (x^4-36x^3+454x^2-2348x+4089), \\
    \mathfrak{f}_7(x) &= (x+5)^{31} (x-7) (x-9)^{12} (x-11)^2 (x^2-18x+53), \\
    \mathfrak{f}_8(x) &= (x+5)^{31} (x-9)^{12} (x-11) (x^4-36x^3+454x^2-2332x+3961).
\end{align*}
For these eight polynomials, there is only one possible interlacing configuration $(n_{\mathfrak f_1},n_{\mathfrak f_2},\dots,n_{\mathfrak f_8}) = (0, 42, 0, 0, 0, 0, 1, 6)$.
The three polynomials that correspond to the nonzero entries in the interlacing configuration are pairwise Seidel-compatible.
By Lemma~\ref{lem:extracting}, we conclude that there exists a Seidel matrix $\widehat{S}$ of order 6 with eigenvalue $7$.
This is a contradiction since $\tr \widehat{S}^2=30<49=7^2$.
\end{proof}

\begin{lemma} \label{lem:last8_sixint_ev13}
There does not exist a Seidel matrix $S$ with characteristic polynomial
$$
\operatorname{Char}_S(x) = (x+5)^{32} (x-7)^2 (x-8) (x-9)^{10} (x-11)^2 (x-13)^2.
$$
\end{lemma}

\begin{proof}
Suppose a Seidel matrix $S$ has characteristic polynomial
$$
\operatorname{Char}_S(x) = (x+5)^{32} (x-7)^2 (x-8) (x-9)^{10} (x-11)^2 (x-13)^2.
$$
There are 17 interlacing characteristic polynomials and one of them is warranted:
$$(x+5)^{31} (x-7)^2 (x-9)^9 (x-11)^2 (x-13) (x^3-25x^2+191x-439)$$
with certificate of warranty $(-1368582672, 0, 0, -369129, -41014, -4558)$.
Using Seidel-compatibility, we reduce from 17 to eight polynomials:
\begin{align*}
    \mathfrak{f}_1(x) =& (x+5)^{31} (x-7) (x-9)^9 (x-11) (x-13) (x^5-43x^4+718x^3-5802x^2+22577x-33547), \\
    \mathfrak{f}_2(x) =& (x+5)^{31} (x-7)^2 (x-9)^9 (x-11)^2 (x-13) (x^3-25x^2+191x-439), \\
    \mathfrak{f}_3(x) =& (x+5)^{31} (x-7) (x-9)^{10} (x-11) (x-13) (x^4-34x^3+412x^2-2086x+3659), \\
    \mathfrak{f}_4(x) =& (x+5)^{31} (x-7)^2 (x-9)^9 (x-11)^2 (x-13) (x^3-25x^2+191x-431), \\
    \mathfrak{f}_5(x) =& (x+5)^{31} (x-7) (x-9)^{10} (x-11) (x-13) (x^4-34x^3+412x^2-2078x+3571), \\
    \mathfrak{f}_6(x) =& (x+5)^{31} (x-7) (x-9)^9 (x-11)^2 (x-13) (x^2-12x+31) (x^2-20x+95), \\
    \mathfrak{f}_7(x) =& (x+5)^{31} (x-7) (x-9)^{10} (x-11)^2 (x-13) (x^3-23x^2+159x-321), \\
    \mathfrak{f}_8(x) =& (x+5)^{31} (x-7) (x-9)^{10} (x-11)^2 (x-13) (x^3-23x^2+159x-313).
\end{align*}
These eight polynomials are pairwise Seidel-compatible and we find that there are 193 possible interlacing configurations.
Furthermore, each interlacing configuration $(n_{\mathfrak f_1},n_{\mathfrak f_2},\dots,n_{\mathfrak f_8})$ has the property that $n_{\mathfrak f_2}+n_{\mathfrak f_4}+n_{\mathfrak f_6}+n_{\mathfrak f_7}+n_{\mathfrak f_8}=33$.
By Lemma~\ref{lem:extracting}, we conclude that there exists a Seidel matrix $\widehat{S}$ of order 16 with eigenvalue 11 of multiplicity at least 2.
This is a contradiction since $\tr \widehat{S}^2=16 \cdot 15=240 < 242=2 \cdot 11^2$.
\end{proof}

\section{Common interlacing and matrices with two simple eigenvalues}
\label{sec:common2simple}

Let $M$ be a real symmetric matrix of order $n$.
It follows immediately from  Lemma~\ref{thm:cauchyinterlace} (Cauchy's interlacing theorem) that,
for all $i, j \in \{1,\dots,n\}$,
the polynomial $\operatorname{Char}_{M[i,j]}(x)$ interlaces both $\operatorname{Char}_{M[i]}(x)$ and $\operatorname{Char}_{M[j]}(x)$.
This seemingly innocuous proposition can be leveraged to help us rule out two of the candidate characteristic polynomials from Lemma~\ref{lem:candpols}.
See Lemma~\ref{lem:last8_fourint} and Lemma~\ref{lem:last8_quad64}.

\begin{lemma} \label{lem:last8_fourint}
There does not exist a Seidel matrix $S$ with characteristic polynomial
$$
\operatorname{Char}_S(x) = (x+5)^{32} (x-4) (x-9)^{10} (x-11)^6.
$$
\end{lemma}

\begin{proof}
Suppose a Seidel matrix $S$ has characteristic polynomial
$$
\operatorname{Char}_S(x) = (x+5)^{32} (x-4) (x-9)^{10} (x-11)^6.
$$
There are five interlacing characteristic polynomials and one of them is warranted:
\begin{align*}
    (x+5)^{31} (x-5) (x-9)^9 (x-11)^5 (x^2-14x+37)
\end{align*}
with certificate of warranty $(11984, 0, 0, 67)$.
Using Seidel-compatibility, we reduce from five to two polynomials:
\begin{align*}
    \mathfrak{f}_1(x) &= (x+5)^{31} (x-7) (x-9)^9 (x-11)^5 (x^2-12x+23), \\
    \mathfrak{f}_2(x) &= (x+5)^{31} (x-5) (x-9)^9 (x-11)^5 (x^2-14x+37).
\end{align*}
We find that there is only one possible interlacing configuration $(7,42)$.

The polynomial $\mathfrak{f}_2(x)$ has 28 interlacing characteristic polynomials.
We find a warranted interlacing characteristic polynomial of $\mathfrak{f}_2(x)$:
$$
(x+5)^{30} (x-9)^8 (x-11)^4 (x^2-16x+59)(x^3-18x^2+91x-134)
$$
with certificate of warranty $(-632673042, 0, 0, -330026, -37186, -4649)$.
Using Seidel-compatibility, we reduce from 28 to seven polynomials:
\begin{align*}
    \mathfrak{f}_{2,1}(x) &= (x+5)^{30} (x-3) (x-5)^2 (x-9)^8 (x-10) (x-11)^5, \\
    \mathfrak{f}_{2,2}(x) &= (x+5)^{30} (x-6) (x-9)^8 (x-11)^4 (x^2-14x+37)^2, \\
    \mathfrak{f}_{2,3}(x) &= (x+5)^{30} (x-9)^8 (x-11)^4 (x^2-16x+59)(x^3-18x^2+91x-134), \\
    \mathfrak{f}_{2,4}(x) &= (x+5)^{30} (x-3) (x-7) (x-9)^8 (x-11)^5 (x^2-13x+34), \\
    \mathfrak{f}_{2,5}(x) &= (x+5)^{30} (x-9)^8 (x-11)^4 (x^2-12x+23) (x^3-22x^2+151x-318), \\
    \mathfrak{f}_{2,6}(x) &= (x+5)^{30} (x-7) (x-9)^8 (x-11)^4 (x^4-27x^3+249x^2-889x+978), \\
    \mathfrak{f}_{2,7}(x) &= (x+5)^{30} (x-5) (x-9)^9 (x-11)^4 (x^3-20x^2+113x-142).
\end{align*}
These seven polynomials are pairwise Seidel-compatible and we find that there are 16 possible interlacing configurations $(n_{\mathfrak f_{2,1}},\dots,n_{\mathfrak f_{2,7}})$:
\begin{align*}
    & (0, 16, 27, 2, 1, 2, 0), (0, 19, 24, 2, 0, 2, 1), (2, 14, 27, 1, 3, 1, 0), (2, 17, 24, 1, 2, 1, 1), \\
    & (2, 20, 21, 1, 1, 1, 2), (2, 22, 18, 1, 0, 5, 0), (2, 23, 18, 1, 0, 1, 3), (4, 12, 27, 0, 5, 0, 0), \\
    & (4, 15, 24, 0, 4, 0, 1), (4, 18, 21, 0, 3, 0, 2), (4, 20, 18, 0, 2, 4, 0), (4, 21, 18, 0, 2, 0, 3), \\
    & (4, 23, 15, 0, 1, 4, 1), (4, 24, 15, 0, 1, 0, 4), (4, 26, 12, 0, 0, 4, 2), (4, 27, 12, 0, 0, 0, 5).
\end{align*}
Furthermore, only $\mathfrak{f}_{2,5}(x)$, $\mathfrak{f}_{2,6}(x)$, and $\mathfrak{f}_{2,7}(x)$ interlace both $\mathfrak{f}_1(x)$ and $\mathfrak{f}_2(x)$.
Suppose $(n_{\mathfrak f_{2,1}},\dots,n_{\mathfrak f_{2,7}})$ is an interlacing configuration for $\mathfrak{f}_2(x)$.
It follows that $n_{\mathfrak f_{2,5}}+ n_{\mathfrak f_{2,6}}+ n_{\mathfrak f_{2,7}} \geqslant 7$ since there is only one possible interlacing configuration $(7,42)$ for $\operatorname{Char}_S(x)$.
However, each of the 16 possible interlacing configurations above satisfies $n_{\mathfrak f_{2,5}}+ n_{\mathfrak f_{2,6}}+ n_{\mathfrak f_{2,7}} < 7$, which is a contradiction.
\end{proof}

Next we provide a result for matrices that have two simple eigenvalues that will help us eliminate some potential interlacing configurations.
We denote Schur (Hadamard, or entrywise) multiplication for two vectors $\mathbf v$ and $\mathbf w$ as $\mathbf v \circ \mathbf w$.

\begin{lemma} \label{lem:intercon}
Let $M$ be an integer symmetric matrix of order $n$.
Suppose $\lambda$ and $\mu$ are simple eigenvalues of $M$ such that $(x-\lambda)(x-\mu) \in \mathbb Z[x]$ is irreducible.
Set
$$q(x) = \prod_{\substack{\nu \in \Lambda(M) \\ \nu \not \in \{\lambda,\mu\}}} (x - \nu).$$
Suppose that $M\bm \alpha_\lambda = \lambda \bm \alpha_\lambda$ and all entries of the matrix $2 q(\lambda)\bm \alpha_\lambda \bm \alpha_\lambda^\transpose$ are in $\mathbb R \backslash \mathbb Z$.
Then there exists a unique $\bm \delta \in \{\pm 1\}^n$ such that $\bm \delta(1) = 1$ and
\[
q(\lambda)\bm \alpha_\lambda \bm \alpha_\lambda^\transpose + q(\mu)(\bm \delta \circ  \bm \alpha_\mu) (\bm \delta \circ  \bm \alpha_\mu)^\transpose
\]
is an integer matrix.
Furthermore, $\bm \delta \circ \bm \alpha_\mu$ is a $\mu$-eigenvector of $M$ and, for all $i, j \in \{1,\dots,n\}$ such that $\bm \alpha_\lambda(i) = \bm \alpha_\lambda(j)$, we have $\bm \delta(i) = \bm \delta(j)$.
\end{lemma}

\begin{proof}
    Set
    \[
\Psi (\bm \delta) \coloneqq q(\lambda)\bm \alpha_\lambda \bm \alpha_\lambda^\transpose + q(\mu)(\bm \delta \circ  \bm \alpha_\mu) (\bm \delta \circ  \bm \alpha_\mu)^\transpose.
\]
    Clearly there must exist $\bm \delta \in \{\pm 1\}^n$ such that $\bm \delta(1) = 1$ and $\Psi (\bm \delta) = q(M)$ is an integer matrix.
Indeed, take $\bm \delta = \bm \delta_0 \in \{\pm 1\}^n$ such that $\bm \delta_0(1) = 1$ and $\bm \delta_0 \circ  \bm \alpha_\mu$ is a unit eigenvector for $\mu$.
It remains to show that such $\bm \delta$ is unique.

Suppose that $\bm \delta_1 \in \{\pm 1\}^n$ satisfies $\bm \delta_1(1) = 1$ and $\Psi (\bm \delta_1) = q(M)$ is an integer matrix where $\bm \delta_1(j) \neq \bm \delta_0(j)$ for some $j > 1$.
Then, by considering the entry $(1,j)$ of $\Psi (\bm \delta)$, we conclude that both
$$ q(\lambda)\bm \alpha_\lambda(1)\bm \alpha_\lambda(j) \pm q(\mu)\bm \alpha_\mu(1)\bm \alpha_\mu(j) \in \mathbb{Z}, $$
which is not possible since each entry of the matrix $2 q(\lambda)\bm \alpha_\lambda \bm \alpha_\lambda^\transpose$ is not an integer.

Let $K = \mathbb Q(\lambda)$ be the splitting field of $(x-\lambda)(x-\mu)$.
Take $\sigma \in \Gal(K/\mathbb{Q})$ such that $\sigma(\mu) = \nu$.
Let $i,j \in \{1,\dots,n\}$ such that $\bm \alpha_\lambda(i) = \bm \alpha_\lambda(j)$.
By Lemma~\ref{lem:anglesq_formula}, we have 
\[
\frac{f_i(\lambda)}{\operatorname{Min}_M^\prime(\lambda)} = \frac{f_j(\lambda)}{\operatorname{Min}_M^\prime(\lambda)}.
\]
Applying the automorphism $\sigma$ and using Lemma~\ref{lem:anglesq_formula} again yields $\bm \alpha_\mu(i) = \bm \alpha_\mu(j)$.
Suppose, to the contrary, that $\bm \delta(i) \ne \bm \delta(j)$ then $\bm \delta(i) + \bm \delta(j) = 0$.
It follows that
\begin{align*}
   \Psi (\bm \delta)_{1,i} + \Psi (\bm \delta)_{1,j} =  2q(\lambda)\bm \alpha_\lambda(1)\bm \alpha_\lambda(i)
\end{align*}
is an integer, which is a contradiction.
Therefore, we have $\bm \delta(i) = \bm \delta(j)$, as required.
\end{proof}

We use Lemma~\ref{lem:intercon} in the subsequent nonexistence results.
One important utility of Lemma~\ref{lem:intercon}, as shown in the forthcoming proofs, is that it allows us to exhaustively check for all possibilities for $\bm \delta$.

\begin{lemma} \label{lem:last8_quad76}
There does not exist a Seidel matrix $S$ with characteristic polynomial
$$
\operatorname{Char}_S(x) = (x+5)^{32} (x-7) (x-9)^{10} (x-11)^4 (x^2-19x+76).
$$
\end{lemma}

\begin{proof}
Suppose a Seidel matrix $S$ has characteristic polynomial
$$
\operatorname{Char}_S(x) = (x+5)^{32} (x-7) (x-9)^{10} (x-11)^4 (x^2-19x+76).
$$
There are 99 interlacing characteristic polynomials and one of them is warranted:
\begin{align*}
    (x+5)^{31} (x-7) (x-9)^9 (x-11)^3 (x^4-34x^3+408x^2-2030x+3543)
\end{align*}
with certificate of warranty $(0, 0, 0, -31002, -10684, -1527)$.
Using Seidel-compatibility, we reduce from 99 to five polynomials:
\begin{align*}
    \mathfrak{f}_1(x) &= (x+5)^{31} (x-7) (x-9)^9 (x-11)^3 (x^4-34x^3+408x^2-2030x+3543), \\
    \mathfrak{f}_2(x) &= (x+5)^{31} (x-9)^9 (x-11)^3 (x^2-16x+59) (x^3-25x^2+187x-411), \\
    \mathfrak{f}_3(x) &= (x+5)^{31} (x-7) (x-9)^9 (x-11)^3 (x^4-34x^3+408x^2-2022x+3455), \\
    \mathfrak{f}_4(x) &= (x+5)^{31} (x-9)^{11} (x-11)^3 (x^3-23x^2+151x-289), \\
    \mathfrak{f}_5(x) &= (x+5)^{31} (x-7)^2 (x-9)^9 (x-11)^4 (x^2-16x+43).
\end{align*}
These five polynomials are pairwise Seidel-compatible and we find that there are two possible interlacing configurations $(n_{\mathfrak f_1},n_{\mathfrak f_2},n_{\mathfrak f_{3}},n_{\mathfrak f_{4}},n_{\mathfrak f_{5}})$: $(24, 12, 6, 0, 7)$ and $(27, 11, 2, 1, 8)$.
    
Let $(\lambda_1,\dots,\lambda_6)=\left( (19-\sqrt{57})/2, (19+\sqrt{57})/2, -5, 7, 9, 11 \right)$ be a 6-tuple of the six distinct eigenvalues of $S$.
Firstly, we compute the following angles
$$ 
\begin{bmatrix}
    \alpha_{\lambda_1}^2(\mathfrak f_1) & \alpha_{\lambda_2}^2(\mathfrak f_1) \\
    \alpha_{\lambda_1}^2(\mathfrak f_2) & \alpha_{\lambda_2}^2(\mathfrak f_2) \\
    \alpha_{\lambda_1}^2(\mathfrak f_3) & \alpha_{\lambda_2}^2(\mathfrak f_3) \\
    \alpha_{\lambda_1}^2(\mathfrak f_4) & \alpha_{\lambda_2}^2(\mathfrak f_4) \\
    \alpha_{\lambda_1}^2(\mathfrak f_5) & \alpha_{\lambda_2}^2(\mathfrak f_5) \\
\end{bmatrix}
= 
\begin{bmatrix}
    (323+\sqrt{57})/39102 & (323-\sqrt{57})/39102  \\
    (437-39\sqrt{57})/52136 & (437+39\sqrt{57})/52136  \\
    (589+22\sqrt{57})/19551 & (589-22\sqrt{57})/19551  \\
    (95+11\sqrt{57})/3192 & (95-11\sqrt{57})/3192  \\
    (969+3\sqrt{57})/13034 & (969-3\sqrt{57})/13034
\end{bmatrix}.
$$

By Corollary~\ref{cor:alpha_eigenvec}, without loss of generality, we can assume that $\bm \alpha_{\lambda_1}$ is an eigenvector for $S$ with eigenvalue $\lambda_1$.
Using the angles calculated above, we can check that $2(\lambda_1+5)(\lambda_1-7)(\lambda_1-9)(\lambda_1-11)\bm \alpha_{\lambda_1}\bm \alpha_{\lambda_1}^\transpose$ does not contain any element in $\mathbb Z$.
By Lemma~\ref{lem:intercon}, there exists $\bm \delta \in \{\pm 1\}^{49}$ such that $\bm \delta(1) = 1$ and $\bm \alpha_{\lambda_1}$ is orthogonal to $\bm \delta \circ \bm \alpha_{\lambda_2}$.
Furthermore, instead of checking $2^{48}$ possibilities for $\bm \delta$, we need only check $2^4$ possibilities for $\bm \delta$.
However, we find that for the two possible interlacing configurations $(24, 12, 6, 0, 7)$ and $(27, 11, 2, 1, 8)$ there does not exist a $\bm \delta \in \{\pm 1\}^{49}$ such that $\bm \alpha_{\lambda_1}$ is orthogonal to $\bm \delta \circ \bm \alpha_{\lambda_2}$, which contradicts Lemma~\ref{lem:intercon}.
\end{proof}

\section{Euler graphs}
\label{sec:euler}

In this section, we introduce one last tool that will enable us to dispose of the last two remaining polynomials from Lemma~\ref{lem:candpols}.
Let $S$ be a Seidel matrix and denote by $\Gamma(S)$ the \textbf{underlying graph} of $S$: the graph whose adjacency matrix is given by $(J-I-S)/2$.
The \textbf{switching class} of $S$ is defined as the set consisting of all underlying graphs $\Gamma(DSD)$ where $D$ is a diagonal matrix with entries in $\{\pm 1\}$.
An \textbf{Euler graph} is a graph all of whose vertices have even degree.
Seidel~\cite{Seidel74} showed that if $n$ is odd then there is precisely one Euler graph in the switching class of $S$.
Let $m$ be a positive integer.
If $A$ and $B$ are integer matrices of the same size, then $A \equiv_m B$ indicates that, for all $i$ and $j$, the $(i,j)$-entry of $A$ is congruent to the $(i,j)$-entry of $B$ modulo $m$.
    
\begin{lemma} \label{lem:euler_mod4}
Let $S$ be a Seidel matrix of order $n$ odd.
Let $\lambda$, $\mu$, and $\nu$ be integers with $\lambda$ and $\mu$ odd.
Then the underlying graph of $S$ is the unique Euler graph contained in the switching class of $S$ if and only if
\begin{align*}
    (S-\lambda I)(S-\mu I) &\equiv_4 (n-2-\lambda-\mu) J \text{ and} \\
    (S-\lambda I)(S-\mu I)(S-\nu I) &\equiv_4 (n-2-\lambda-\mu) (n-1-\nu) J.
\end{align*}
\end{lemma}
    
\begin{proof}
Let $S=J-I-2A$.
Expanding $(S-\lambda I)(S-\mu I)$, we obtain
\begin{align*}
    (S-\lambda I)(S-\mu I) = & (n-2-\lambda-\mu)J + (\lambda+1)(\mu+1)I + 2(\lambda+\mu+2)A + 4A^2 -2(J A + A J).
\end{align*}
Since $\lambda$ and $\mu$ are odd, it follows that
\begin{align} 
    (S-\lambda I)(S-\mu I) \equiv_4 & (n-2-\lambda-\mu)J - 2(J A + A J). \label{gen_quadraticS_mod4}
\end{align}
Suppose the underlying graph of $S$ is the unique Euler graph contained in the switching class of $S$.
Note that $JA \equiv_2 AJ \equiv_2 O$, since $A$ is the adjacency matrix of an Euler graph.
Then from \eqref{gen_quadraticS_mod4}, it is clear that
\begin{align*}
    (S-\lambda I)(S-\mu I) \equiv_4 & (n-2-\lambda-\mu)J.
\end{align*}
Therefore,
\begin{align*}
    (S-\lambda I)(S-\mu I)(S-\nu I) 
    &\equiv_4 (n-2-\lambda-\mu) J (J-(\nu+1)I-2A) \\
    &\equiv_4 (n-2-\lambda-\mu) (J^2-(\nu+1)J-2 J A) \\
    &\equiv_4 (n-2-\lambda-\mu) (n-1-\nu) J.
\end{align*}

Conversely, suppose that
\begin{align}
    (S-\lambda I)(S-\mu I) &\equiv_4 (n-2-\lambda-\mu) J \label{quadraticS_mod4} \\
    (S-\lambda I)(S-\mu I)(S-\nu I) &\equiv_4 (n-2-\lambda-\mu) (n-1-\nu) J.
    \label{cubicS_mod4}
\end{align}
From \eqref{quadraticS_mod4}, we have
\begin{align*}
    (S-\lambda I)(S-\mu I)(S-\nu I)
    &\equiv_4 (n-2-\lambda-\mu) (n-1-\nu) J - 2 (n-2-\lambda-\mu) J A.
\end{align*}
Combining with \eqref{cubicS_mod4}, we obtain
\begin{align*}
    2(n-2-\lambda-\mu) J A \equiv_4 O \implies (n-2-\lambda-\mu) J A \equiv_2 O
\end{align*}
Since $n$, $\lambda$, and $\mu$ are all odd, $n-2-\lambda-\mu$ is also odd.
It follows that $J A \equiv_2 O$ and therefore, the underlying graph of $S$ is an Euler graph, which is unique in the switching class of $S$.
\end{proof}

\begin{lemma} \label{lem:last8_quad72}
There does not exist a Seidel matrix $S$ with characteristic polynomial
$$
\operatorname{Char}_S(x) = (x+5)^{32} (x-9)^{12} (x-11)^3 (x^2-19x+72).
$$
\end{lemma}

\begin{proof}
Suppose a Seidel matrix $S$ has characteristic polynomial
$$
\operatorname{Char}_S(x) = (x+5)^{32} (x-9)^{12} (x-11)^3 (x^2-19x+72).
$$
There are 53 interlacing characteristic polynomials for $\operatorname{Char}_S(x)$, two of which are $\mathfrak{f}_3(x)$ and $\mathfrak{f}_4(x)$:
\begin{align*}
    \mathfrak{f}_3(x) &= (x+5)^{31} (x-9)^{11} (x-11)^2 (x^4-34x^3+404x^2-1966x+3323), \\
    \mathfrak{f}_4(x) &= (x+5)^{31} (x-9)^{11} (x-11)^2 (x^4-34x^3+404x^2-1966x+3339).
\end{align*}
The set $\{ \mathfrak{f}_3(x), \mathfrak{f}_4(x) \}$ is warranted with certificate $(-37937620,$ $0,$ $0,$ $-29081,$ $-5816)$.
However, $\mathfrak{f}_3(x)$ and $\mathfrak{f}_4(x)$ are not Seidel-compatible.
Hence we can consider Seidel-compatibility with $\mathfrak{f}_3(x)$ and $\mathfrak{f}_4(x)$ separately.
From the 53 interlacing characteristic polynomials $\operatorname{Char}_S(x)$, only three are Seidel-compatible with $\mathfrak{f}_3(x)$:
\begin{align*}
    \mathfrak{f}_3(x) &= (x+5)^{31} (x-9)^{11} (x-11)^2 (x^4-34x^3+404x^2-1966x+3323), \\
    \mathfrak{f}_{10}(x) &= (x+5)^{31} (x-9)^{11} (x-11)^3 (x^3-23x^2+151x-297), \\
    \mathfrak{f}_{41}(x) &= (x+5)^{31} (x-9)^{11} (x-11)^2 (x-13) (x^3-21x^2+131x-215).
\end{align*}
These three polynomials are pairwise Seidel-compatible but we find that there is only one possible solution to \eqref{eq:sumofsubpoly_dim17Compat}: $(91/2,1,5/2)$, which is a contradiction.
On the other hand, only four out of the 53 polynomials are Seidel-compatible with $\mathfrak{f}_4(x)$:
\begin{align*}
    \mathfrak{f}_1(x) &= (x+5)^{31} (x-9)^{11} (x-11)^2 (x^4-34x^3+404x^2-1966x+3291), \\
    \mathfrak{f}_4(x) &= (x+5)^{31} (x-9)^{11} (x-11)^2 (x^4-34x^3+404x^2-1966x+3339), \\
    \mathfrak{f}_{24}(x) &= (x+5)^{31} (x-9)^{11} (x-11)^2 (x^4-34x^3+404x^2-1942x+3075), \\
    \mathfrak{f}_{32}(x) &= (x+5)^{31} (x-3) (x-7) (x-9)^{11} (x-11)^3 (x-13).
\end{align*}
These four polynomials are pairwise Seidel-compatible and we find that there are two possible interlacing configurations $(n_{\mathfrak f_1},n_{\mathfrak f_4},n_{\mathfrak f_{24}},n_{\mathfrak f_{32}})$: $(16, 28, 4, 1)$ and $(17, 28, 0, 4)$.
    
Let $(\lambda_1,\dots,\lambda_5)=\left( (19-\sqrt{73})/2, (19+\sqrt{73})/2, 11, -5, 9 \right)$ be a 5-tuple of the five distinct eigenvalues of $S$.
Firstly, we compute the following angles
$$ 
\begin{bmatrix}
    \alpha_{\lambda_1}^2(\mathfrak f_1) &    \alpha_{\lambda_2}^2(\mathfrak f_1) &    \alpha_{\lambda_3}^2(\mathfrak f_1) \\
    \alpha_{\lambda_1}^2(\mathfrak f_4) &    \alpha_{\lambda_2}^2(\mathfrak f_4) &    \alpha_{\lambda_3}^2(\mathfrak f_4) \\
    \alpha_{\lambda_1}^2(\mathfrak f_{24}) & \alpha_{\lambda_2}^2(\mathfrak f_{24}) & \alpha_{\lambda_3}^2(\mathfrak f_{24}) \\
    \alpha_{\lambda_1}^2(\mathfrak f_{32}) & \alpha_{\lambda_2}^2(\mathfrak f_{32}) & \alpha_{\lambda_3}^2(\mathfrak f_{32})
\end{bmatrix}
= 
\begin{bmatrix}
    (219+23\sqrt{73})/14016 & (219-23\sqrt{73})/14016 &
    1/8 \\
    (803-65\sqrt{73})/56064 & (803+65\sqrt{73})/56064 &
    1/32 \\
    (3723+71\sqrt{73})/56064 & (3723-71\sqrt{73})/56064 &
    1/32 \\
    (73+\sqrt{73})/876 & 
    (73-\sqrt{73})/876 &
    0
\end{bmatrix}.
$$

By Corollary~\ref{cor:alpha_eigenvec}, without loss of generality, we can assume that $\bm \alpha_{\lambda_1}$ is an eigenvector for $S$ with eigenvalue $\lambda_1$.
Using the angles calculated above, we can check that $2(\lambda_1+5)(\lambda_1-9)(\lambda_1-11)\bm \alpha_{\lambda_1}\bm \alpha_{\lambda_1}^\transpose$ does not contain any element in $\mathbb Z$.
By Lemma~\ref{lem:intercon}, there exists $\bm \delta \in \{\pm 1\}^{49}$ such that $\bm \delta(1) = 1$ and $\bm \alpha_{\lambda_1}$ is orthogonal to $\bm \delta \circ \bm \alpha_{\lambda_2}$.
Furthermore, instead of checking $2^{48}$ possibilities for $\bm \delta$, we need only check $2^3$ possibilities for $\bm \delta$.
We find that, out of the two possible interlacing configurations $(16,28,4,1)$ and $(17, 28, 0, 4)$, for only $(16, 28, 4, 1)$ does there exist a $\bm \delta \in \{\pm 1\}^{49}$ such that $\bm \delta(1) = 1$ and $\bm \alpha_{\lambda_1}$ is orthogonal to $\bm \delta \circ \bm \alpha_{\lambda_2}$.
Furthermore, this $\bm \delta$ is unique.
    
Suppose $S$ has interlacing configuration $(n_{\mathfrak f_1},n_{\mathfrak f_4},n_{\mathfrak f_{24}},n_{\mathfrak f_{32}}) = (16, 28, 4, 1)$.
Let $\mathcal{I}_1 = \{1,\dots,16\}$, $\mathcal{I}_4 = \{17,\dots,44\}$, $\mathcal{I}_{24} = \{45,\dots,48\}$, and $\mathcal{I}_{32} = \{ 49 \}$ and fix the ordering such that $\bm \alpha_{\lambda_i}(j) = \alpha_{\lambda_i}(\mathfrak f_k)$ for each $i \in \{1,\dots,5\}$ and each $j \in \mathcal{I}_k$.
Suppose we switch \(S\) to \(S_{\epsilon}\) such that the underlying graph of \(S_{\epsilon}\) is the unique Euler graph contained in the switching class of $S$.
Let \(T_{\epsilon} = (S_{\epsilon}+5I)(S_{\epsilon}-9I)(S_{\epsilon}-11I)\) and by Lemma~\ref{lem:euler_mod4}, we have $T_{\epsilon} \equiv_4 3J$.
We find that there exists an $\bm \epsilon \in \{\pm 1\}^{49}$ such that $\bm \epsilon(1) = 1$ and 
\begin{align*}
    T_{\epsilon} = \eta (\bm \epsilon \circ \bm \alpha_{\lambda_1}) (\bm \epsilon \circ \bm \alpha_{\lambda_1})^\transpose + \overline \eta ( \bm \epsilon \circ \bm \delta \circ \bm \alpha_{\lambda_2}) (\bm \epsilon \circ \bm \delta \circ \bm \alpha_{\lambda_2})^\transpose = 
    \begin{pmatrix}
        7J_{16} & -J & -13J & -9J \\
        -J & 7J_{28} & -5J & 15J  \\
        -13J & -5J & 31J_4 & 3J \\
        -9J & 15J & 3J & 39
    \end{pmatrix},
\end{align*}
where $\eta = 3(157-\sqrt{73})/2$ and $\overline \eta = 3(157+\sqrt{73})/2$.
Furthermore, such an \(\bm \epsilon\) is unique.
Indeed, any diagonal $\{\pm 1\}$-matrix \(D\) satisfying \(DT_{\epsilon}D \equiv_4 T_{\epsilon}\) is either $\pm I$.
Set $\mathbf u = \bm \epsilon \circ \bm \alpha_{\lambda_1}$ and $\mathbf v = \bm \epsilon \circ \bm \delta \circ \bm \alpha_{\lambda_2}$.
Let $\mathbf{a} = (a_1,\dots,a_{49})^\transpose$, $\mathbf{b} = (b_1,\dots,b_{49})^\transpose$, and $\mathbf{c} = (c_1,\dots,c_{49})^\transpose$ be an orthonormal basis for the eigenspace $\mathcal{E}(\lambda_3)$. 
By Lemma~\ref{lem:anglesq_sqsum}, the entries $a_i$, $b_i$, and $c_i$ satisfy $a_i^2 + b_i^2 + c_i^2 = 1/8$ for all $i \in \mathcal{I}_1$, $a_i^2 + b_i^2 + c_i^2 = 1/32$ for all $i \in \mathcal{I}_4 \cup \mathcal{I}_{24}$, and $a_{49} = b_{49} = c_{49} = 0$.
By the Cauchy-Schwarz inequality,
\begin{align}
    (a_i^2 + b_i^2 + c_i^2)(a_j^2 + b_j^2 +c_j^2) \geqslant (a_i a_j + b_i b_j +c_i c_j)^2
    \label{cs_ineq}
\end{align}
for all $i, j \in \{1,\dots,49\}$.
Now, by Lemma~\ref{lem:euler_mod4},
\begin{align*}
    (S_{\epsilon}+5I)(S_{\epsilon}-9I) &= \frac{3}{2} (17-5\sqrt{73}) \mathbf u \mathbf u^\transpose + \frac{3}{2} (17+5\sqrt{73})  \mathbf v \mathbf v^\transpose
    + 32P_{\lambda_3} \\
    &= \begin{pmatrix}
   -J_{16} & J & J & 3J \\
        J & 2J_{28} & -4J & 3J \\
        J & -4J & 2J_4 & -9J \\
        3J & 3J & -9J & 3
    \end{pmatrix} + 32(\mathbf{a} \mathbf{a}^\transpose + \mathbf{b} \mathbf{b}^\transpose + \mathbf{c} \mathbf{c}^\transpose) \equiv_4 3J.
\end{align*}
From \eqref{cs_ineq}, we have
\begin{align*}
    -4 &\leqslant 32(a_i a_j + b_i b_j + c_i c_j) \leqslant 4 \; \text{for all } i, j \in \mathcal{I}_1, \\
    -1 &\leqslant 32(a_i a_j + b_i b_j + c_i c_j) \leqslant 1 \; \text{for all } i, j \in \mathcal{I}_4 \cup \mathcal{I}_{24}, \\
    -2 &\leqslant 32(a_i a_j + b_i b_j + c_i c_j) \leqslant 2 \; \text{for all } i \in \mathcal{I}_1 \; \text{and for all } j \in \mathcal{I}_4 \cup \mathcal{I}_{24}.
\end{align*}
Therefore,
\begin{align}
    32(\mathbf{a} \mathbf{a}^\transpose + \mathbf{b} \mathbf{b}^\transpose + \mathbf{c} \mathbf{c}^\transpose) = \begin{pmatrix}
    M_{1,1} & M_{1,4} & M_{1,24} & O \\
    M_{1,4}^\transpose & J_{28} & -J & O  \\
    M_{1,24}^\transpose & -J & J_4 & O \\
    O & O & O & 0
    \end{pmatrix}
    \label{32P11_block_entries}
\end{align}
where the entries of $M_{1,1}$ are $-4$, $0$, or $4$ and the entries of $M_{1,4}$ and $M_{1,24}$ are $-2$ or $2$.

Let $(a_{17}, b_{17}, c_{17})=(a, b, c)$ for some $a, b, c \in \mathbb{R}$.
Using \eqref{32P11_block_entries}, for $i, j \in \mathcal{I}_4$ we have
\begin{align*}
    a_i^2 + b_i^2 +c_i^2 = a_j^2 + b_j^2 +c_j^2 = a_i a_j + b_i b_j + c_i c_j = \frac{1}{32},
\end{align*}
from which it follows that $(a_i, b_i, c_i)=(a_j, b_j, c_j)$.
Similarly, for $i \in \mathcal{I}_4$ and $j \in \mathcal{I}_{24}$, we have
\begin{align*}
    a_i^2 + b_i^2 +c_i^2 = a_j^2 + b_j^2 +c_j^2 = -(a_i a_j + b_i b_j + c_i c_j) = \frac{1}{32},
\end{align*}
from which it follows that $(a_i, b_i, c_i)=(-a_j, -b_j, -c_j)$.
So far, we have obtained 
\begin{align*}
    \mathbf{a}^\transpose &= (\underbrace{a_1, \dots, a_{16}}_{16}, \underbrace{a, \dots, a}_{28}, -a, -a, -a, -a ,0), \\
    \mathbf{b}^\transpose &= (\underbrace{b_1, \dots, b_{16}}_{16}, \underbrace{b, \dots, b}_{28}, -b, -b, -b, -b ,0), \\
    \mathbf{c}^\transpose &= (\underbrace{c_1, \dots, c_{16}}_{16}, \underbrace{c, \dots, c}_{28}, -c, -c, -c, -c ,0).
\end{align*}
Similarly again, using \eqref{32P11_block_entries}, for $i \in \mathcal{I}_1$, we have 
\begin{align*}
    a_i^2 + b_i^2 +c_i^2 = 4(a^2 + b^2 + c^2) = \frac{1}{8} \; 
    \text{ and } \;
    a a_i + b b_i + c c_i = \pm \frac{1}{16}.
\end{align*}
Therefore, $(a_i, b_i, c_i)= \pm 2(a, b, c)$ for all $i \in \mathcal{I}_1$.
Finally, consider the matrix
$B_{\mathcal{E}(\lambda_3)} = [\mathbf{a} \;\; \mathbf{b} \;\; \mathbf{c}]$.
Since the dimension of $\mathcal{E}(\lambda_3)$ is $3$ so is the rank of $B_{\mathcal{E}(\lambda_3)}$.
However, the row space of $B_{\mathcal{E}(\lambda_3)}$ is generated by $(a, b, c)$, which is a contradiction. 
\end{proof}

Finally, we dismiss the most stubborn candidate characteristic polynomial for a Seidel matrix corresponding to $49$ equiangular lines in $\mathbb R^{17}$.

\begin{lemma} \label{lem:last8_quad64}
There does not exist a Seidel matrix $S$ with characteristic polynomial
$$
\operatorname{Char}_S(x) = (x+5)^{32} (x-9)^{13} (x-13)^2 (x^2-17x+64).
$$
\end{lemma}

\begin{proof}
Suppose a Seidel matrix $S$ has characteristic polynomial
$$
\operatorname{Char}_S(x) = (x+5)^{32} (x-9)^{13} (x-13)^2 (x^2-17x+64).
$$
There are 26 interlacing characteristic polynomials and one of them is warranted:
$$(x+5)^{31} (x-9)^{12} (x-13) (x^4-34x^3+408x^2-2022x+3503)$$
with certificate of warranty $(-12716952, 0, 0, -7781, -864)$.
Using Seidel-compatibility, we reduce from 26 to six polynomials:
\begin{align*}
    \mathfrak{f}_1(x) &= (x+5)^{31}  (x-9)^{12} (x-13) (x^4-34x^3+408x^2-2022x+3503), \\
    \mathfrak{f}_2(x) &= (x+5)^{31} (x-7) (x-9)^{12} (x-13)^2 (x^2-14x+37), \\
    \mathfrak{f}_3(x) &= (x+5)^{31} (x-9)^{12} (x-13) (x^4-34x^3+408x^2-2014x+3415), \\
    \mathfrak{f}_4(x) &= (x+5)^{31} (x-9)^{12} (x-13) (x^4-34x^3+408x^2-2006x+3295), \\
    \mathfrak{f}_5(x) &= (x+5)^{31} (x-3) (x-9)^{12} (x-11) (x-13) (x^2-20x+95), \\
    \mathfrak{f}_6(x) &= (x+5)^{31} (x-9)^{13} (x-11) (x-13) (x^2-14x+29).
\end{align*}
These six polynomials are pairwise Seidel-compatible and we find that there are thirteen possible interlacing configurations $(n_{\mathfrak f_1},n_{\mathfrak f_2},n_{\mathfrak f_3},n_{\mathfrak f_4},n_{\mathfrak f_5},n_{\mathfrak f_6})$:
\begin{align*}
    & (30, 3, 16, 0, 0, 0), (33, 1, 12, 3, 0, 0), (34, 2, 11, 1, 1, 0), (37, 0, 7, 4, 1, 0), \\
    & (37, 2, 8, 1, 0, 1), (38, 1, 6, 2, 2, 0), (39, 2, 5, 0, 3, 0), (40, 0, 4, 4, 0, 1), \\
    & (41, 1, 3, 2, 1, 1), (42, 0, 1, 3, 3, 0), (42, 2, 2, 0, 2, 1), (43, 1, 0, 1, 4, 0), \\
    & (44, 1, 0, 2, 0, 2).
\end{align*}
Let $(\lambda_1,\dots,\lambda_5)=\left( (17-\sqrt{33})/2, (17+\sqrt{33})/2, 13, -5, 9 \right)$ be a 5-tuple of the five distinct eigenvalues of $S$.
Firstly, we compute the following angles
$$ 
\begin{bmatrix}
\alpha_{\lambda_1}^2(\mathfrak f_1) & \alpha_{\lambda_2}^2(\mathfrak f_1) & \alpha_{\lambda_3}^2(\mathfrak f_1)  \\
\alpha_{\lambda_1}^2(\mathfrak f_2) & \alpha_{\lambda_2}^2(\mathfrak f_2) & \alpha_{\lambda_3}^2(\mathfrak f_2)  \\
\alpha_{\lambda_1}^2(\mathfrak f_3) & \alpha_{\lambda_2}^2(\mathfrak f_3) & \alpha_{\lambda_3}^2(\mathfrak f_3)  \\
\alpha_{\lambda_1}^2(\mathfrak f_4) & \alpha_{\lambda_2}^2(\mathfrak f_4) & \alpha_{\lambda_3}^2(\mathfrak f_4)  \\
\alpha_{\lambda_1}^2(\mathfrak f_5) & \alpha_{\lambda_2}^2(\mathfrak f_5) & \alpha_{\lambda_3}^2(\mathfrak f_5)  \\
\alpha_{\lambda_1}^2(\mathfrak f_6) & \alpha_{\lambda_2}^2(\mathfrak f_6) & \alpha_{\lambda_3}^2(\mathfrak f_6)
\end{bmatrix}
= 
\begin{bmatrix}
(33-2\sqrt{33})/2871 & (33+2\sqrt{33})/2871 & 1/27  \\
(33-2\sqrt{33})/319 & (33+2\sqrt{33})/319 & 0  \\
(165+19\sqrt{33})/7656 & (165-19\sqrt{33})/7656 & 1/18  \\
(495-\sqrt{33})/5742 & (495+\sqrt{33})/5742 & 1/27  \\
(165+19\sqrt{33})/2088 & (165-19\sqrt{33})/2088 & 5/54  \\
(627+107\sqrt{33})/5742 & (627-107\sqrt{33})/5742 & 4/27
\end{bmatrix}.
$$

Using Lemma~\ref{lem:intercon}, we can reduce from 13 possible interlacing configurations to just three.
Indeed, by Corollary~\ref{cor:alpha_eigenvec}, without loss of generality, we can assume that $\bm \alpha_{\lambda_1}$ is an eigenvector for $S$ with eigenvalue $\lambda_1$.
Using the angles calculated above, we can check that $2(\lambda_1+5)(\lambda_1-9)(\lambda_1-13)\bm \alpha_{\lambda_1}\bm \alpha_{\lambda_1}^\transpose$ does not contain any element in $\mathbb Z$.
By Lemma~\ref{lem:intercon}, there exists $\bm \delta \in \{\pm 1\}^{49}$ such that $\bm \delta(1) = 1$ and $\bm \alpha_{\lambda_1}$ is orthogonal to $\bm \delta \circ \bm \alpha_{\lambda_2}$.
Furthermore, instead of checking $2^{48}$ possibilities for $\bm \delta$, we need only check $2^5$ possibilities for $\bm \delta$.
Now, out of the 13 possible interlacing configurations, for only
\begin{align*}
    (33, 1, 12, 3, 0, 0), (37, 0, 7, 4, 1, 0), (40, 0, 4, 4, 0, 1)
\end{align*}
do there exist a $\bm \delta \in \{\pm 1\}^{49}$ such that $\bm \delta(1) = 1$ and 
$$
(S+5I)(S-9I)(S-13I)= \frac{3}{2} (67+19\sqrt{33}) \bm \alpha_{\lambda_1} \bm \alpha_{\lambda_1}^\transpose + \frac{3}{2} (67-19\sqrt{33}) (\bm \delta \circ \bm \alpha_{\lambda_2})(\bm \delta \circ \bm \alpha_{\lambda_2})^\transpose
$$ 
is an integer matrix and $\bm \alpha_{\lambda_1}$ is orthogonal to $\bm \delta \circ \bm \alpha_{\lambda_2}$.
Furthermore, this $\bm \delta$ is unique for each of the three interlacing configurations.

First suppose that $S$ has interlacing configuration $(n_{\mathfrak f_1},n_{\mathfrak f_2},n_{\mathfrak f_3},n_{\mathfrak f_4},n_{\mathfrak f_5},n_{\mathfrak f_6}) = (33, 1, 12, 3, 0, 0)$.
Let $\mathcal{I}_1 = \{1,\dots,33\}$, $\mathcal{I}_2 = \{34\}$, $\mathcal{I}_{3} = \{35,\dots,46\}$, and $\mathcal{I}_{4} = \{ 47,48,49 \}$ and fix the ordering such that $\bm \alpha_{\lambda_i}(j) = \alpha_{\lambda_i}(\mathfrak f_k)$ for each $i \in \{1,\dots, 5\}$ and each $j \in \mathcal{I}_k$.
Suppose we switch \(S\) to \(S_{\epsilon}\) such that the underlying graph of \(S_{\epsilon}\) is the unique Euler graph contained in the switching class of $S$.
Let \(T_{\epsilon} = (S_{\epsilon}+5I)(S_{\epsilon}-9I)(S_{\epsilon}-13I)\) and by Lemma~\ref{lem:euler_mod4}, we have $T_{\epsilon} \equiv_4 J$.
We find that there exists $\bm \epsilon \in \{\pm 1\}^{49}$ such that $\bm \epsilon(1) = 1$ and \begin{align*}
    T_{\epsilon} = \eta ( \bm \epsilon \circ \bm \alpha_{\lambda_1}) (\bm \epsilon \circ \bm \alpha_{\lambda_1})^\transpose + \overline \eta( \bm \epsilon \circ \bm \delta \circ \bm \alpha_{\lambda_2} )(\bm \epsilon \circ \bm \delta \circ \bm \alpha_{\lambda_2})^\transpose = 
    \begin{pmatrix}
    J_{33} & -3J & 5J & 9J \\
    -3J & 9 & -15J & -27J  \\
    5J & -15J & 9J_{12} & 13J \\
    9J & -27J & 13J & 17J_{3}
    \end{pmatrix},
\end{align*}
where $\eta = 3(67+19\sqrt{33})/2$ and $\overline \eta = 3(67-19\sqrt{33})/2$.
Furthermore, such an \(\bm \epsilon\) is unique.
Indeed, any diagonal $\{\pm 1\}$-matrix \(D\) satisfying \(DT_{\epsilon}D \equiv_4 T_{\epsilon}\) is either $\pm I$.
Set $\mathbf u = \bm \epsilon \circ \bm \alpha_{\lambda_1}$ and $\mathbf v = \bm \epsilon \circ \bm \delta \circ \bm \alpha_{\lambda_2}$.
Let $\mathbf{a} = (a_1,\dots,a_{49})^\transpose$ and $\mathbf{b} = (b_1,\dots,b_{49})^\transpose$ be an orthonormal basis for the eigenspace $\mathcal{E}(\lambda_3)$. 
By Lemma~\ref{lem:anglesq_sqsum}, the entries $a_i$ and $b_i$ satisfy $a_i^2 + b_i^2= 1/27$ for all $i \in \mathcal{I}_1 \cup \mathcal{I}_4$, $a_{34} = b_{34} = 0$, and $a_i^2 + b_i^2 = 1/18$ for all $i \in \mathcal{I}_3$.
By the Cauchy-Schwarz inequality,
\begin{align}
    (a_i^2 + b_i^2)(a_j^2 + b_j^2 ) \geqslant (a_i a_j + b_i b_j)^2
    \label{cs_ineq_2d}
\end{align}
for all $i, j \in \{1,\dots,49\}$.
Now, by Lemma~\ref{lem:euler_mod4},
\begin{align*}
    (S_{\epsilon}+5I)(S_{\epsilon}-9I) &= \frac{3-13\sqrt{33}}{2} \mathbf{u} \mathbf{u}^\transpose + \frac{3+13\sqrt{33}}{2} \mathbf{v} \mathbf{v}^\transpose + 72P_{\lambda_3} \\
    &= \begin{pmatrix}
    1/3J_{33} & -J & -J & -7/3J \\
    -J & 3& 3J & 7J  \\
    -J & 3J & -J_{12} & -J \\
    -7/3J & 7J & -J & 1/3J_3
    \end{pmatrix} + 72(\mathbf{a} \mathbf{a}^\transpose + \mathbf{b} \mathbf{b}^\transpose) \equiv_4 3J.
\end{align*}
From \eqref{cs_ineq_2d}, we have
\begin{align*}
    -8/3 &\leqslant 72(a_i a_j + b_i b_j) \leqslant 8/3, & \text{for all } i, j \in \mathcal{I}_1 \cup \mathcal{I}_4; \\
    -4 &\leqslant 72(a_i a_j + b_i b_j) \leqslant 4, & \text{for all } i, j \in \mathcal{I}_3; \\
    -4\sqrt{6}/3 &\leqslant 72(a_i a_j + b_i b_j) \leqslant 4\sqrt{6}/3, & \text{for all } i \in \mathcal{I}_1 \cup \mathcal{I}_4 \; \text{and } j \in \mathcal{I}_3.
\end{align*}
Consequently, we have
\begin{align}
    72(\mathbf{a} \mathbf{a}^\transpose + \mathbf{b} \mathbf{b}^\transpose ) = \begin{pmatrix}
    M_{1,1} & O & O & M_{1,4} \\
    O & 0 & O & O  \\
    O & O & M_{3,3} & O \\
    M_{1,4}^\transpose & O & O & M_{4,4}
    \end{pmatrix}
    \label{72P13_block_entries}
\end{align}
where the entries of the matrices $M_{1,1}$ and $M_{4,4}$ of order $33$ and $3$ respectively are $-4/3$ or $8/3$, the entries of $M_{1,4}$ are $-8/3$ or $4/3$ and the entries of the matrix $M_{3,3}$ of order $12$ are $-4$, $0$, or $4$.
    
Let $(a_1, b_1) = (a, b)$ for some $a, b \in \mathbb{R}$.
Using \eqref{72P13_block_entries}, for $i \in \mathcal{I}_1 \cup \mathcal{I}_4$ and $j \in \mathcal{I}_3$, we have $a_i^2 a_j^2 = b_i^2 b_j^2$, from which it follows that 
\begin{align}
    a_i^2 (a_j^2 + b_j^2) &= b_j^2 (a_i^2 + b_i^2) \implies 3a_i^2 = 2b_j^2;
    \label{add_ai_bj} \\
    a_j^2 (a_i^2 + b_i^2) &= b_i^2 (a_j^2 + b_j^2) \implies 2a_j^2 = 3b_i^2.
    \label{add_aj_bi}
\end{align}

Using \eqref{add_ai_bj} we see that $a_i^2 = a^2$ for all $i \in \mathcal{I}_1 \cup \mathcal{I}_4$ and $b_j^2 = 3a^2/2$ for all $j \in \mathcal{I}_3$.
Similarly, from \eqref{add_aj_bi} we have $b_i^2 = b^2$ for all $i \in \mathcal{I}_1 \cup \mathcal{I}_4$ and $a_j^2 = 3b^2/2$ for all $j \in \mathcal{I}_3$.
Therefore,
\begin{align*}
    (\mathbf{a} \circ \mathbf{a})^\transpose &= (\underbrace{a^2, \dots, a^2}_{33}, 0, \underbrace{3b^2/2, \dots, 3b^2/2 }_{12}, a^2, a^2, a^2 ), \\
    (\mathbf{b} \circ \mathbf{b})^\transpose &= (\underbrace{b^2, \dots, b^2}_{33}, 0, \underbrace{3a^2/2, \dots, 3a^2/2}_{12}, b^2, b^2, b^2 ).
\end{align*}
Since $\mathbf{a}$ and $\mathbf{b}$ are unit vectors, we must have $a^2 = b^2 = 1/54$.
Based on the entries of $M_{1,1}$ in \eqref{72P13_block_entries} we have
\begin{align*}
    72(a_1 a_i + b_1 b_i) = -\frac{4}{3} \; \text{ or } \; \frac{8}{3} \implies a a_i + b b_i = -\frac{1}{54} \; \text{ or } \; \frac{1}{27}
\end{align*}
for all $i \in \mathcal{I}_1 \setminus \{1\}$.
The options are $a_i = \pm a$ and $b_i = \pm b$ and it is easy to check that only $a_i = a$ and $b_i = b$ are possible.
Similarly, if $i \in \mathcal{I}_4$ then $a_i = -a$ and $b_i = -b$ based on the entries of $M_{1,4}$ in \eqref{72P13_block_entries}.
Furthermore, if $i \in \mathcal{I}_3$ then
\begin{align*}
    & a a_i + b b_i = 0 \implies a a_i b_i = -b b_i^2 
    \implies a_i b_i = -\frac{3}{2} ab
\end{align*}
as $a \neq 0$.
Therefore, 
\begin{align*}
    \mathbf{a} \cdot \mathbf{b} = 33ab + 0 - 12 \cdot \frac{3}{2} ab + 3ab = 18 ab.
\end{align*}
Since $a^2=b^2=1/54$, then $ab=\pm 1/54$.
Hence we obtain $\mathbf{a} \cdot \mathbf{b} = \pm 1/3$, which contradicts the fact that $\mathbf{a}$ and $\mathbf{b}$ are orthogonal.

Secondly, suppose that $S$ has interlacing configuration $(n_{\mathfrak f_1},n_{\mathfrak f_2},n_{\mathfrak f_3},n_{\mathfrak f_4},n_{\mathfrak f_5},n_{\mathfrak f_6}) = (37, 0, 7, 4, 1, 0)$.
Let $\mathcal{I}_1 = \{1,\dots,37\}$, $\mathcal{I}_{3} = \{38,\dots,44\}$, $\mathcal{I}_{4} = \{ 45,\dots,48 \}$, and $\mathcal{I}_5 = \{49\}$ and fix the ordering such that $\bm \alpha_{\lambda_i}(j) = \alpha_{\lambda_i}(\mathfrak f_k)$ for each $i \in \{1,\dots, 5\}$ and each $j \in \mathcal{I}_k$.
Suppose we switch \(S\) to \(S_{\epsilon}\) such that the underlying graph of \(S_{\epsilon}\) is the unique Euler graph contained in the switching class of $S$.
Let \(T_{\epsilon} = (S_{\epsilon}+5I)(S_{\epsilon}-9I)(S_{\epsilon}-13I)\) and by Lemma~\ref{lem:euler_mod4}, we have $T_{\epsilon} \equiv_4 J$.
We find that there exists an $\bm \epsilon \in \{\pm 1\}^{49}$ such that $\bm \epsilon(1) = 1$ and \begin{align*}
    T_{\epsilon} = \eta ( \bm \epsilon \circ \bm \alpha_{\lambda_1}) (\bm \epsilon \circ \bm \alpha_{\lambda_1})^\transpose + \overline \eta( \bm \epsilon \circ \bm \delta \circ \bm \alpha_{\lambda_2} )(\bm \epsilon \circ \bm \delta \circ \bm \alpha_{\lambda_2})^\transpose = 
    \begin{pmatrix}
        J_{37} & 5J & 9J & -7J \\
        5J & 9J_7 & 13J & -19J  \\
        9J & 13J & 17J_4 & -31J \\
        -7J & -19J & -31J & 33
    \end{pmatrix},
\end{align*}
where $\eta = 3(67+19\sqrt{33})/2$ and $\overline \eta = 3(67-19\sqrt{33})/2$.
Furthermore, such an \(\bm \epsilon\) is unique.
Indeed, any diagonal $\{\pm 1\}$-matrix \(D\) satisfying \(DT_{\epsilon}D \equiv_4 T_{\epsilon}\) is either $\pm I$.
Set $\mathbf u = \bm \epsilon \circ \bm \alpha_{\lambda_1}$ and $\mathbf v = \bm \epsilon \circ \bm \delta \circ \bm \alpha_{\lambda_2}$.
Let $\mathbf{a} = (a_1,\dots,a_{49})^\transpose$ and $\mathbf{b} = (b_1,\dots,b_{49})^\transpose$ be an orthonormal basis for the eigenspace $\mathcal{E}(\lambda_3)$.
By Lemma~\ref{lem:anglesq_sqsum}, the entries $a_i$ and $b_i$
satisfy $a_i^2 + b_i^2= 1/27$ for all $i \in \mathcal{I}_1 \cup \mathcal{I}_4$, $a_i^2 + b_i^2 = 1/18$ for all $i \in \mathcal{I}_3$, and $a_{49}^2 + b_{49}^2 = 5/54$.
By the Cauchy-Schwarz inequality,
\begin{align}
    (a_i^2 + b_i^2)(a_j^2 + b_j^2 ) \geqslant (a_i a_j + b_i b_j)^2
    \label{second_cs_ineq_2d}
\end{align}
for all $i, j \in \{1,\dots,49\}$.
Now, by Lemma~\ref{lem:euler_mod4},
\begin{align*}
    (S_{\epsilon}+5I)(S_{\epsilon}-9I) &= \frac{3-13\sqrt{33}}{2} \mathbf{u} \mathbf{u}^\transpose + \frac{3+13\sqrt{33}}{2} \mathbf{v} \mathbf{v}^\transpose + 72P_{\lambda_3} \\
    &= \begin{pmatrix}
        1/3J_{37} & -J & -7/3J & 1/3J \\
        -J & -J_7 & -J & 3J \\
        -7/3J & -J & 1/3J_4 & 17/3J \\
        1/3J & 3J & 17/3J & -11/3
    \end{pmatrix} + 72(\mathbf{a} \mathbf{a}^\transpose + \mathbf{b} \mathbf{b}^\transpose) \equiv_4 3J.
\end{align*}
From \eqref{second_cs_ineq_2d}, we have
\begin{align*}
    -8/3 &\leqslant 72(a_i a_j + b_i b_j) \leqslant 8/3, & \text{for all } i, j \in \mathcal{I}_1 \cup \mathcal{I}_4; \\
    -4 &\leqslant 72(a_i a_j + b_i b_j) \leqslant 4, & \text{for all } i, j \in \mathcal{I}_3; \\
    -4\sqrt{6}/3 &\leqslant 72(a_i a_j + b_i b_j) \leqslant 4\sqrt{6}/3, & \text{for all } i \in \mathcal{I}_1 \cup \mathcal{I}_4 \; \text{and }  j \in \mathcal{I}_3; \\
    -4\sqrt{10}/3 &\leqslant 72(a_i a_{49} + b_i b_{49}) \leqslant 4\sqrt{10}/3, & \text{for all } i \in \mathcal{I}_1 \cup \mathcal{I}_4; \\
    -4\sqrt{15}/3 &\leqslant 72(a_i a_{49} + b_i b_{49}) \leqslant 4\sqrt{15}/3, & \text{for all } i \in \mathcal{I}_3.
\end{align*}
Consequently, we have
\begin{align}
    72(\mathbf{a} \mathbf{a}^\transpose + \mathbf{b} \mathbf{b}^\transpose ) = \begin{pmatrix}
    M_{1,1} & O & M_{1,4} & M_{1,5} \\
    O & M_{3,3} & O & M_{3,5} \\
    M_{1,4}^\transpose & O & M_{4,4} & M_{4,5} \\
    M_{1,5}^\transpose & M_{3,5}^\transpose & M_{4,5}^\transpose & 20/3
    \end{pmatrix}
    \label{second_72P13_block_entries}
\end{align}
where the entries of $M_{1,1}$, $M_{1,5}$, and $M_{4,4}$ are $-4/3$ or $8/3$, the entries of $M_{1,4}$ and $M_{4,5}$ are $-8/3$ or $4/3$, and the entries of $M_{3,3}$ and $M_{3,5}$ are $-4$, $0$, or $4$.
Note that the matrices $M_{1,1}$, $M_{3,3}$, and $M_{4,4}$ have orders $37$, $7$, and $4$ respectively.

Let $(a_1, b_1) = (a, b)$ for some $a, b \in \mathbb{R}$.
Using \eqref{second_72P13_block_entries}, for $i \in \mathcal{I}_1 \cup \mathcal{I}_4$ and $j \in \mathcal{I}_3$, we have $a_i^2 a_j^2 = b_i^2 b_j^2$.
Following the same reasoning as for the previous interlacing configuration, we obtain
\begin{align*}
    (\mathbf{a} \circ \mathbf{a})^\transpose &= (\underbrace{a^2, \dots, a^2}_{37}, \underbrace{3b^2/2, \dots, 3b^2/2}_{7}, a^2, a^2, a^2, a^2, a_{49}^2 ), \\
    (\mathbf{b} \circ \mathbf{b})^\transpose &= (\underbrace{b^2, \dots, b^2}_{37}, \underbrace{3a^2/2, \dots, 3a^2/2}_{7}, b^2, b^2, b^2, b^2, b_{49}^2 ).
\end{align*}
    
We observe that $a$ and $b$ are both nonzero.
Otherwise, by Lemma~\ref{lem:zero_block_eigenvecs}, there would be a Seidel matrix $\widehat{S}$ of order 8 with eigenvalue 13 of multiplicity at least one.
This is a contradiction since $\tr \widehat{S}^2 = 56 < 169 = 13^2$.

Now let $i \in \mathcal{I}_1 \setminus \{1\}$ and hence,
\begin{align*}
    72(a_1 a_i + b_1 b_i) = -\frac{4}{3} \; \text{ or } \; \frac{8}{3} \implies a a_i + b b_i = -\frac{1}{54} \; \text{ or } \; \frac{1}{27}
\end{align*}
based on the entries of $M_{1,1}$ in \eqref{second_72P13_block_entries}.
The options are $a_i = \pm a$ and $b_i = \pm b$.
We can exclude the possibility $a_i = -a$ and $b_i = -b$, otherwise $a a_i + b b_i = -1/27$.
    
Suppose $a_i = a$ and $b_i = -b$ so we have $a a_i + b b_i = a^2-b^2$.
If $a^2 - b^2 = 1/27$ then $b=0$, which is a contradiction.
If $a^2 - b^2 = -1/54$ then $a^2 = 1/108$ and $b^2 = 1/36$.
However, this implies that $37b^2 > 1$, which contradicts that $\mathbf{b}$ is a unit vector.
    
Similarly, suppose $a_i = -a$ and $b_i = b$ so we have $a a_i + b b_i = b^2-a^2$.
If $b^2 - a^2 = 1/27$ then $a=0$, which is a contradiction.
If $b^2 - a^2 = -1/54$ then $a^2 = 1/36$ and $b^2 = 1/108$.
However, this implies that $37a^2 > 1$, which contradicts that $\mathbf{a}$ is a unit vector.
Thus, we conclude that $a_i = a$ and $b_i = b$ for all $i \in \mathcal{I}_1$.
In a similar fashion, we also conclude that $a_i = -a$ and $b_i = -b$ for all $i \in \mathcal{I}_4$ based on the entries of $M_{1,4}$ in \eqref{second_72P13_block_entries}.

Furthermore, if $i \in \mathcal{I}_3$ then from $a a_i + b b_i = 0$ we derive $a_i b_i = -3ab/2$ since $a, b \neq 0$.
Combining with $\mathbf{a} \cdot \mathbf{b} = 0$, we obtain
\begin{align}
 a_{49}^2 b_{49}^2 &= \frac{3721}{4} a^2 \left( \frac{1}{27}-a^2 \right)
    \label{eq:a_b_ortho}.
\end{align}
Since $\mathbf{a}$ and $\mathbf{b}$ are unit vectors, we obtain
\begin{align}
    a_{49}^2 &= 1-41a^2-\frac{21}{2} b^2 = 1-41a^2-\frac{21}{2} \left( \frac{1}{27}-a^2 \right) = \frac{11}{18} - \frac{61}{2} a^2, \label{eq:a_unit} \\
    b_{49}^2 &= 1-41b^2-\frac{21}{2} a^2 = 1-41 \left( \frac{1}{27}-a^2 \right) - \frac{21}{2} a^2 = \frac{61}{2} a^2 -\frac{14}{27}.
    \label{eq:b_unit}
\end{align}
But together \eqref{eq:a_unit} and \eqref{eq:b_unit} are inconsistent with \eqref{eq:a_b_ortho}. 
Therefore, we arrive at a contradiction.
    
Lastly, suppose we have $(40, 0, 4, 4, 0, 1)$ as the interlacing configuration.
Here we employ different strategy compared with the previous two interlacing configurations.
There are 73 interlacing characteristic polynomials corresponding to the polynomial $\mathfrak{f}_6(x)$.
We note that the corresponding entry in the interlacing configuration is $1$ so suppose $\operatorname{Char}_{S[1]}(x) = \mathfrak{f}_6(x)$. 
For $i \in \{2,\dots,49\}$ we have $\operatorname{Char}_{S[i]}(x) = \mathfrak{f}_1(x)$, $\mathfrak{f}_3(x)$, or $\mathfrak{f}_4(x)$.
It follows that $\operatorname{Char}_{S[1,i]}(x)$ interlaces both $\mathfrak{f}_6(x)$ and $\operatorname{Char}_{S[i]}(x)$.
Out of the 73 polynomials, only 25 interlace at least one of $\mathfrak{f}_1(x)$, $\mathfrak{f}_3(x)$, $\mathfrak{f}_4(x)$.
Among them are
\begin{align*}
    \mathfrak{f}_{6,2}(x) &= (x+5)^{30} (x-9)^{12} (x-11) (x^4-31x^3+321x^2-1245x+1506),\\
    \mathfrak{f}_{6,3}(x) &= (x+5)^{30} (x-9)^{12} (x-11) (x^4-31x^3+321x^2-1241x+1462),
\end{align*}
where the set $\{ \mathfrak{f}_{6,2}(x), \mathfrak{f}_{6,3}(x) \}$ is warranted with certificate $(5716582502$, $0$, $0$, $1614215$, $145454$, $13175)$.
However, we also have that $\mathfrak{f}_{6,2}(x)$ and $\mathfrak{f}_{6,3}(x)$ are not Seidel-compatible.
Hence, we can consider Seidel-compatibility with $\mathfrak{f}_{6,2}(x)$ and $\mathfrak{f}_{6,3}(x)$ separately.
From the 25 polynomials, only two are Seidel-compatible with $\mathfrak{f}_{6,2}(x)$:
\begin{align*}
    \mathfrak{f}_{6,1}(x) &= (x+5)^{30} (x-9)^{12} (x-13) (x^2-14x+29) (x^2-15x+46), \\
    \mathfrak{f}_{6,2}(x) &= (x+5)^{30} (x-9)^{12} (x-11) (x^4-31x^3+321x^2-1245x+1506).
\end{align*}
But we have a certificate of infeasibility $(38051376, 0, 0, 10571, 823, 3)$.
On the other hand, only four out of the 25 polynomials are Seidel-compatible with $\mathfrak{f}_{6,3}(x)$:
\begin{align*}
    \mathfrak{f}_{6,1}(x) &= (x+5)^{30} (x-9)^{12} (x-13) (x^2-14x+29) (x^2-15x+46), \\
    \mathfrak{f}_{6,3}(x) &= (x+5)^{30} (x-9)^{12} (x-11) (x^4-31x^3+321x^2-1241x+1462), \\
    \mathfrak{f}_{6,11}(x) &= (x+5)^{30} (x-9)^{12} (x-13) (x^4-29x^3+285x^2-1047x+982), \\
    \mathfrak{f}_{6,18}(x) &= (x+5)^{30} (x-9)^{12} (x-11) (x^4-31x^3+321x^2-1209x+1046).
\end{align*}
But we have a certificate of infeasibility $(-7998906611, 0, 0, -2236872, -195255, -17128)$.
Therefore, we arrive at a contradiction.
\end{proof}

Thus the proof of Theorem~\ref{thm:main} is complete.

\section{Conclusion}

    Our proof of $N(17) = 48$ consists of two main parts.
    First we enumerate all the polynomials that can potentially be the characteristic polynomial of a Seidel matrix corresponding to an equiangular line system of cardinality $49$ in $\mathbb R^{17}$.
    We find that there are $194$ such polynomials.
    Second, we show that none of these polynomials can be the characteristic polynomial of a Seidel matrix.
    In principle, this technique can be applied in general to find which polynomials can be the characteristic polynomial of a Seidel matrix that corresponds to an equiangular line system of cardinality $n$ in $\mathbb R^{d}$.
    However, the amount of computation required to generate all of the polynomials in the first part can become too expensive for various combinations of $n$ and $d$.
    Furthermore, given a list of candidate characteristic polynomials, it can be very challenging to determine whether or not there exists a corresponding Seidel matrix for each of the polynomials.
    For example, to show that the polynomial $(x+5)^{32}(x-9)^{16}(x-16)$ cannot be the characteristic polynomial of a Seidel matrix, it is necessary to show the nonexistence of a strongly regular graph of order $49$~\cite{GG18}.
    The question of the existence of certain strongly regular graphs is a notoriously difficult problem - an entire paper was devoted to showing the nonexistence of a strongly regular graph of order $49$~\cite{srg49}, which was implicitly required to show that $N(17) = 48$.
    
    If we apply our techniques to the next open case, dimension $18$, we can enumerate all the candidate characteristic polynomials that correspond to a Seidel matrix for an equiangular line system of cardinality $60$ in $\mathbb R^{18}$.
    However, our current techniques cannot dispose of all the candidates that we find.
    In particular, the polynomial $(x+5)^{42}(x-11)^{15}(x-15)^3$, which was first identified in \cite{GG18} as a candidate for the characteristic polynomial for a putative Seidel matrix, corresponds to an equiangular line system of cardinality $60$ in $\mathbb R^{18}$.
    Looking further, enumerating all candidates for characteristic polynomials of a putative Seidel matrix for an equiangular line system of cardinality $59$ in $\mathbb R^{18}$ may be possible on a super computer, but it appears to be out of reach of our current methods using a personal computer.
    
    Lastly, we believe the methods we use may have more general applications for studying the existence of Hermitian matrices whose entries are restricted to a certain discrete ring and whose eigenvalues are geometrically constrained.

\section{Acknowledgements}

We are grateful to Yufei Zhao for his comments on an earlier version of this paper and to the referees, whose comments led to various improvements to this paper.

\bibliographystyle{amsplain}

\appendix

\section{Tables}

In this appendix we provide tables of polynomials, certificates of infeasibility, and certificates of warranty referenced in Lemma~\ref{lem:candpols}, Lemma~\ref{lem:firstTable}, Lemma~\ref{lem:secondTable}, and Lemma~\ref{lem:thirdTable}.

\begin{center}

\footnotesize
\begin{longtable}{@{\makebox[3em][r]{\rownumber\space}} | l}

\endfirsthead

\multicolumn{1}{c}
{{\bfseries \tablename\ \thetable{} -- continued from previous page}} \\ \hline 
\endhead

\hline \multicolumn{1}{r}{{Continued on next page}} \\ \hline
\endfoot

\endlastfoot

\multicolumn{1}{@{\makebox[3em][r]{Index~}} | l}{Candidate characteristic polynomial}\\

\hline

   $(x+5)^{32} (x-9)^{14} (x^3-34x^2+369x-1292)$ \\

     $(x+5)^{32} (x-9)^{14} (x^3-34x^2+369x-1280)$ \\

     $(x+5)^{32} (x-9)^{12} (x^5-52x^4+1062x^3-10664x^2+52713x-102724)$ \\

     $(x+5)^{32} (x-9)^{13} (x-11) (x^3-32x^2+323x-1028)$ \\

     $(x+5)^{32} (x-9)^{12} (x^5-52x^4+1062x^3-10656x^2+52569x-102092)$ \\

     $(x+5)^{32} (x-9)^{12} (x^5-52x^4+1062x^3-10652x^2+52481x-101608)$ \\

     $(x+5)^{32} (x-8) (x-9)^{12} (x^4-44x^3+710x^2-4972x+12737)$ \\

     $(x+5)^{32} (x-9)^{12} (x^5-52x^4+1062x^3-10648x^2+52393x-101140)$ \\

     $(x+5)^{32} (x-9)^{12} (x^5-52x^4+1062x^3-10648x^2+52425x-101428)$ \\

     $(x+5)^{32} (x-9)^{12} (x^5-52x^4+1062x^3-10644x^2+52321x-100832)$ \\

     $(x+5)^{32} (x-9)^{13} (x^4-43x^3+675x^2-4569x+11200)$ \\

     $(x+5)^{32} (x-9)^{12} (x^2-22x+109) (x^3-30x^2+293x-928)$ \\

     $(x+5)^{32} (x-9)^{12} (x^5-52x^4+1062x^3-10644x^2+52353x-101120)$ \\

     $(x+5)^{32} (x-9)^{14} (x-13) (x^2-21x+96)$ \\

     $(x+5)^{32} (x-7) (x-9)^{11} (x-11) (x^4-43x^3+679x^2-4665x+11756)$ \\

     $(x+5)^{32} (x-9)^{13} (x^4-43x^3+675x^2-4565x+11164)$ \\

     $(x+5)^{32} (x-7) (x-9)^{12} (x-13) (x^3-32x^2+331x-1108)$ \\

     $(x+5)^{32} (x-9)^{12} (x^2-22x+113) (x^3-30x^2+289x-892)$ \\

     $(x+5)^{32} (x-9)^{12} (x-11)^2 (x^3-30x^2+281x-824)$ \\

     $(x+5)^{32} (x-9)^{11} (x-11) (x^5-50x^4+980x^3-9414x^2+44331x-81896)$ \\

     $(x+5)^{32} (x-9)^{12} (x-11) (x^4-41x^3+611x^2-3915x+9096)$ \\

     $(x+5)^{32} (x-9)^{12} (x^5-52x^4+1062x^3-10636x^2+52161x-100024)$ \\

     $(x+5)^{32} (x-7) (x-9)^{12} (x-11) (x^3-34x^2+373x-1304)$ \\

     $(x+5)^{32} (x-9)^{12} (x^2-22x+113) (x^3-30x^2+289x-888)$ \\

     $(x+5)^{32} (x-8) (x-9)^{10} (x^3-31x^2+311x-1009)^2$ \\

     $(x+5)^{32} (x-9)^{12} (x^5-52x^4+1062x^3-10632x^2+52073x-99556)$ \\

     $(x+5)^{32} (x-9)^{12} (x^2-24x+139) (x^3-28x^2+251x-716)$ \\

     $(x+5)^{32} (x-7)^2 (x-9)^{10} (x-11)^3 (x^2-23x+124)$ \\

     $(x+5)^{32} (x-9)^{10} (x-11) (x^6-59x^5+1430x^4-18230x^3+128957x^2-480063x+734908)$ \\

     $(x+5)^{32} (x-9)^{11} (x^6-61x^5+1530x^4-20190x^3+147793x^2-568885x+899524)$ \\

     $(x+5)^{32} (x-9)^{12} (x^5-52x^4+1062x^3-10632x^2+52105x-99908)$ \\

     $(x+5)^{32} (x-9)^{12} (x^2-20x+87) (x^3-32x^2+335x-1148)$ \\

     $(x+5)^{32} (x-7) (x-9)^{10} (x-11) (x^2-22x+113) (x^3-30x^2+293x-932)$ \\

     $(x+5)^{32} (x-9)^{11} (x^2-20x+95)^2 (x^2-21x+100)$ \\

     $(x+5)^{32} (x-9)^{12} (x^5-52x^4+1062x^3-10628x^2+51969x-98896)$ \\

     $(x+5)^{32} (x-9)^{10} (x-11)^2 (x^5-48x^4+902x^3-8304x^2+37481x-66400)$ \\

     $(x+5)^{32} (x-9)^{12} (x^2-18x+73) (x^3-34x^2+377x-1360)$ \\

     $(x+5)^{32} (x-9)^{11} (x^6-61x^5+1530x^4-20186x^3+147653x^2-567289x+893584)$ \\

     $(x+5)^{32} (x-9)^{13} (x-13) (x^3-30x^2+285x-848)$ \\

     $(x+5)^{32} (x-7) (x-9)^{10} (x-11) (x^5-52x^4+1066x^3-10764x^2+53509x-104704)$ \\

     $(x+5)^{32} (x-9)^{10} (x-11) (x^2-20x+95) (x^4-39x^3+555x^2-3421x+7712)$ \\

     $(x+5)^{32} (x-9)^{12} (x-13) (x^4-39x^3+555x^2-3413x+7664)$ \\

     $(x+5)^{32} (x-7) (x-9)^{10} (x^2-20x+95)^2 (x^2-23x+128)$ \\

     $(x+5)^{32} (x-9)^{11} (x-11)^2 (x^4-39x^3+551x^2-3341x+7340)$ \\

     $(x+5)^{32} (x-9)^{11} (x-11) (x^5-50x^4+980x^3-9402x^2+44091x-80708)$ \\

     $(x+5)^{32} (x-9)^{12} (x^5-52x^4+1062x^3-10624x^2+51897x-98572)$ \\

     $(x+5)^{32} (x-9)^{10} (x-11)^2 (x^5-48x^4+902x^3-8300x^2+37441x-66316)$ \\

     $(x+5)^{32} (x-7) (x-9)^{11} (x-11) (x-12) (x^3-31x^2+307x-965)$ \\

     $(x+5)^{32} (x-9)^{11} (x-11) (x^5-50x^4+980x^3-9402x^2+44123x-81028)$ \\

     $(x+5)^{32} (x-9)^{12} (x-12) (x^2-18x+73) (x^2-22x+113)$ \\

     $(x+5)^{32} (x-9)^{11} (x-11) (x^5-50x^4+980x^3-9402x^2+44123x-80996)$ \\

     $(x+5)^{32} (x-9)^{12} (x-11) (x-13) (x^3-28x^2+247x-692)$ \\

     $(x+5)^{32} (x-9)^{10} (x^3-29x^2+271x-811) (x^4-41x^3+619x^2-4079x+9908)$ \\

     $(x+5)^{32} (x-7)^2 (x-9)^{12} (x-12) (x-13)^2$ \\

     $(x+5)^{32} (x-9)^{10} (x-11)^2 (x^5-48x^4+902x^3-8296x^2+37337x-65768)$ \\

     $(x+5)^{32} (x-9)^{11} (x-11) (x^5-50x^4+980x^3-9398x^2+44011x-80312)$ \\

     $(x+5)^{32} (x-9)^{12} (x-11) (x^4-41x^3+611x^2-3899x+8920)$ \\

     $(x+5)^{32} (x-7) (x-8) (x-9)^9 (x-11)^3 (x^3-31x^2+307x-965)$ \\

     $(x+5)^{32} (x-7) (x-9)^{11} (x-11)^2 (x^3-32x^2+327x-1048)$ \\

     $(x+5)^{32} (x-8) (x-9)^{10} (x-11)^2 (x^2-18x+73) (x^2-22x+113)$ \\

     $(x+5)^{32} (x-9)^{11} (x-11) (x^5-50x^4+980x^3-9398x^2+44043x-80664)$ \\

     $(x+5)^{32} (x-9)^{11} (x-11) (x^5-50x^4+980x^3-9398x^2+44043x-80632)$ \\

     $(x+5)^{32} (x-9)^{12} (x^5-52x^4+1062x^3-10620x^2+51841x-98504)$ \\

     $(x+5)^{32} (x-8) (x-9)^{11} (x-11) (x-13) (x^3-29x^2+267x-775)$ \\

     $(x+5)^{32} (x-9)^9 (x-11) (x^2-20x+95) (x^5-48x^4+906x^3-8408x^2+38373x-68920)$ \\

     $(x+5)^{32} (x-9)^{10} (x^3-31x^2+311x-1009) (x^4-39x^3+559x^2-3481x+7928)$ \\

     $(x+5)^{32} (x-9)^{10} (x-11)^2 (x^5-48x^4+902x^3-8292x^2+37249x-65284)$ \\

     $(x+5)^{32} (x-9)^{12} (x^5-52x^4+1062x^3-10616x^2+51721x-97620)$ \\

     $(x+5)^{32} (x-9)^{10} (x-11)^2 (x^5-48x^4+902x^3-8292x^2+37281x-65572)$ \\

     $(x+5)^{32} (x-9)^{10} (x-11)^2 (x^2-22x+113) (x^3-26x^2+217x-580)$ \\

     $(x+5)^{32} (x-9)^{10} (x-11) (x^6-59x^5+1430x^4-18214x^3+128493x^2-475631x+721004)$ \\

     $(x+5)^{32} (x-9)^{11} (x^2-20x+95) (x^4-41x^3+615x^2-3979x+9292)$ \\

     $(x+5)^{32} (x-9)^{11} (x-13) (x^5-48x^4+906x^3-8396x^2+38149x-67876)$ \\

     $(x+5)^{32} (x-7)^2 (x-9)^9 (x-11)^4 (x^2-21x+100)$ \\

     $(x+5)^{32} (x-7) (x-9)^9 (x-11)^2 (x^5-50x^4+984x^3-9522x^2+45287x-84668)$ \\

     $(x+5)^{32} (x-7) (x-9)^{10} (x-11)^2 (x^4-41x^3+615x^2-3987x+9404)$ \\

     $(x+5)^{32} (x-9)^{10} (x-13) (x^2-17x+68) (x^2-20x+95)^2$ \\

     $(x+5)^{32} (x-9)^{10} (x-11)^2 (x^2-16x+59) (x^3-32x^2+331x-1104)$ \\

     $(x+5)^{32} (x-9)^{10} (x-11)^2 (x^5-48x^4+902x^3-8288x^2+37193x-65104)$ \\

     $(x+5)^{32} (x-9)^{11} (x^6-61x^5+1530x^4-20170x^3+147157x^2-562217x+876448)$ \\

     $(x+5)^{32} (x-9)^9 (x-11)^2 (x^6-57x^5+1334x^4-16406x^3+111817x^2-400449x+588752)$ \\

     $(x+5)^{32} (x-9)^{10} (x-11)^2 (x^2-20x+87) (x^3-28x^2+255x-752)$ \\

     $(x+5)^{32} (x-9)^{10} (x-11)^2 (x^5-48x^4+902x^3-8288x^2+37225x-65392)$ \\

     $(x+5)^{32} (x-9)^{10} (x-11)^2 (x^2-20x+95) (x^3-28x^2+247x-688)$ \\

     $(x+5)^{32} (x-9)^{10} (x-11) (x^6-59x^5+1430x^4-18210x^3+128393x^2-474835x+719024)$ \\

     $(x+5)^{32} (x-7)^3 (x-9)^8 (x-11)^4 (x^2-23x+128)$ \\

     $(x+5)^{32} (x-9)^9 (x-11) (x^7-68x^6+1961x^5-31080x^4+292299x^3-1630868x^2+4997515x-6487424)$ \\

     $(x+5)^{32} (x-9)^9 (x-11) (x^3-29x^2+271x-811) (x^4-39x^3+559x^2-3489x+8000)$ \\

     $(x+5)^{32} (x-9)^{11} (x-13) (x^2-18x+73) (x^3-30x^2+293x-928)$ \\

     $(x+5)^{32} (x-7) (x-9)^9 (x-11) (x-13) (x^2-20x+95) (x^3-28x^2+255x-752)$ \\

     $(x+5)^{32} (x-9)^{10} (x-11) (x^2-19x+80) (x^2-20x+95)^2$ \\

     $(x+5)^{32} (x-9)^{12} (x-11)^2 (x^3-30x^2+281x-796)$ \\

     $(x+5)^{32} (x-9)^{10} (x-11)^2 (x^5-48x^4+902x^3-8284x^2+37121x-64796)$ \\

     $(x+5)^{32} (x-9)^{11} (x-11)^2 (x^4-39x^3+551x^2-3325x+7196)$ \\

     $(x+5)^{32} (x-9)^{11} (x-11) (x^5-50x^4+980x^3-9386x^2+43771x-79124)$ \\

     $(x+5)^{32} (x-9)^{12} (x-11) (x-13)^2 (x^2-15x+52)$ \\

     $(x+5)^{32} (x-9)^9 (x-11)^2 (x^6-57x^5+1334x^4-16402x^3+111709x^2-399493x+585980)$ \\

     $(x+5)^{32} (x-9)^{10} (x-11)^2 (x^2-18x+73) (x^3-30x^2+289x-892)$ \\

     $(x+5)^{32} (x-9)^{10} (x-11)^2 (x^5-48x^4+902x^3-8284x^2+37153x-65084)$ \\

     $(x+5)^{32} (x-7) (x-9)^8 (x-11)^3 (x^5-48x^4+906x^3-8404x^2+38309x-68668)$ \\

     $(x+5)^{32} (x-9)^9 (x-11) (x^7-68x^6+1961x^5-31076x^4+292147x^3-1628724x^2+4984227x-6456932)$ \\

     $(x+5)^{32} (x-9)^{10} (x-12) (x^3-29x^2+271x-811)^2$ \\

     $(x+5)^{32} (x-7)^2 (x-9)^8 (x-11)^2 (x-12) (x^2-20x+95)^2$ \\

     $(x+5)^{32} (x-9)^{10} (x-11)^3 (x^4-37x^3+495x^2-2835x+5864)$ \\

     $(x+5)^{32} (x-9)^{10} (x-11)^2 (x^5-48x^4+902x^3-8280x^2+37049x-64472)$ \\

     $(x+5)^{32} (x-9)^{11} (x-11) (x-13) (x^4-37x^3+499x^2-2895x+6056)$ \\

     $(x+5)^{32} (x-7) (x-9)^9 (x-11)^3 (x^4-39x^3+555x^2-3405x+7576)$ \\

     $(x+5)^{32} (x-9)^{10} (x-11)^2 (x^2-18x+73) (x^3-30x^2+289x-888)$ \\

     $(x+5)^{32} (x-9)^9 (x-11)^2 (x^6-57x^5+1334x^4-16398x^3+111601x^2-398521x+583064)$ \\

     $(x+5)^{32} (x-7) (x-9)^{10} (x-11)^2 (x-13) (x^3-28x^2+251x-712)$ \\

     $(x+5)^{32} (x-8) (x-9)^8 (x-11)^2 (x^3-29x^2+271x-811)^2$ \\

     $(x+5)^{32} (x-9)^9 (x-11) (x^2-20x+95) (x^5-48x^4+906x^3-8392x^2+38085x-67656)$ \\

     $(x+5)^{32} (x-7)^2 (x-8) (x-9)^6 (x-11)^4 (x^2-20x+95)^2$ \\

     $(x+5)^{32} (x-9)^{10} (x-11)^3 (x^4-37x^3+495x^2-2831x+5820)$ \\

     $(x+5)^{32} (x-9)^{10} (x-11)^2 (x^5-48x^4+902x^3-8276x^2+36961x-63988)$ \\

     $(x+5)^{32} (x-9)^9 (x-11)^3 (x^2-17x+68) (x^3-29x^2+267x-775)$ \\

     $(x+5)^{32} (x-9)^9 (x-11)^3 (x^5-46x^4+828x^3-7286x^2+31331x-52636)$ \\

     $(x+5)^{32} (x-9)^{10} (x-11)^2 (x^5-48x^4+902x^3-8276x^2+36993x-64340)$ \\

     $(x+5)^{32} (x-9)^{10} (x-11)^2 (x^2-18x+69) (x^3-30x^2+293x-932)$ \\

     $(x+5)^{32} (x-9)^8 (x-11)^2 (x^7-66x^6+1847x^5-28400x^4+259039x^3-1401034x^2+4158953x-5225156)$ \\

     $(x+5)^{32} (x-7)^2 (x-9)^8 (x-11)^4 (x-13) (x^2-17x+68)$ \\

     $(x+5)^{32} (x-7) (x-9)^8 (x-11)^3 (x^2-18x+73) (x^3-30x^2+293x-932)$ \\

     $(x+5)^{32} (x-9)^{10} (x-11)^2 (x^5-48x^4+902x^3-8272x^2+36873x-63520)$ \\

     $(x+5)^{32} (x-9)^{10} (x-11)^2 (x^5-48x^4+902x^3-8272x^2+36905x-63872)$ \\

     $(x+5)^{32} (x-9)^{10} (x-11)^2 (x^2-20x+95) (x^3-28x^2+247x-672)$ \\

     $(x+5)^{32} (x-7) (x-9)^8 (x-11)^4 (x^4-37x^3+499x^2-2903x+6128)$ \\

     $(x+5)^{32} (x-9)^8 (x-11)^2 (x^2-16x+59) (x^2-20x+95) (x^3-30x^2+293x-928)$ \\

     $(x+5)^{32} (x-9)^9 (x-11)^2 (x^2-17x+64) (x^2-20x+95)^2$ \\

     $(x+5)^{32} (x-7)^2 (x-9)^8 (x-11)^5 (x^2-19x+80)$ \\

     $(x+5)^{32} (x-9)^{10} (x-11)^2 (x^5-48x^4+902x^3-8268x^2+36801x-63212)$ \\

     $(x+5)^{32} (x-9)^9 (x-11)^2 (x^6-57x^5+1334x^4-16386x^3+111245x^2-395093x+572428)$ \\

     $(x+5)^{32} (x-9)^{10} (x-11)^2 (x-12) (x^4-36x^3+470x^2-2628x+5297)$ \\

     $(x+5)^{32} (x-9)^8 (x-11)^4 (x^2-16x+59) (x^3-28x^2+251x-724)$ \\

     $(x+5)^{32} (x-7) (x-9)^8 (x-11)^3 (x^5-48x^4+906x^3-8388x^2+37989x-67084)$ \\

     $(x+5)^{32} (x-9)^8 (x-11)^3 (x^2-15x+52) (x^2-20x+95)^2$ \\

     $(x+5)^{32} (x-9)^7 (x-11)^3 (x^3-29x^2+271x-811) (x^4-35x^3+451x^2-2533x+5228)$ \\

     $(x+5)^{32} (x-9)^{10} (x-11)^2 (x-12) (x^2-18x+73)^2$ \\

     $(x+5)^{32} (x-7)^4 (x-9)^6 (x-11)^6 (x-12)$ \\

     $(x+5)^{32} (x-9)^{11} (x-11)^3 (x^3-28x^2+243x-632)$ \\

     $(x+5)^{32} (x-9)^{10} (x-11)^2 (x^5-48x^4+902x^3-8264x^2+36729x-62888)$ \\

     $(x+5)^{32} (x-9)^9 (x-11)^4 (x^4-35x^3+443x^2-2401x+4712)$ \\

     $(x+5)^{32} (x-8) (x-9)^8 (x-11)^4 (x^4-36x^3+470x^2-2628x+5297)$ \\

     $(x+5)^{32} (x-7) (x-9)^9 (x-11)^3 (x^4-39x^3+555x^2-3389x+7400)$ \\

     $(x+5)^{32} (x-8) (x-9)^8 (x-11)^4 (x^2-18x+73)^2$ \\

     $(x+5)^{32} (x-7)^4 (x-8) (x-9)^4 (x-11)^8$ \\

     $(x+5)^{32} (x-9)^{10} (x-11)^4 (x^3-26x^2+209x-516)$ \\

     $(x+5)^{32} (x-7) (x-9)^9 (x-11)^4 (x^3-28x^2+247x-668)$ \\

     $(x+5)^{32} (x-9)^9 (x-11)^3 (x^5-46x^4+828x^3-7270x^2+31043x-51404)$ \\

     $(x+5)^{32} (x-7) (x-9)^7 (x-11)^5 (x^4-35x^3+447x^2-2465x+4948)$ \\

     $(x+5)^{32} (x-9)^8 (x-11)^4 (x^2-18x+73) (x^3-26x^2+217x-580)$ \\

     $(x+5)^{32} (x-9)^{10} (x-11)^2 (x^5-48x^4+902x^3-8256x^2+36553x-61936)$ \\

     $(x+5)^{32} (x-9)^8 (x-11)^4 (x^5-44x^4+758x^3-6380x^2+26193x-41920)$ \\

     $(x+5)^{32} (x-7) (x-9)^8 (x-11)^5 (x^3-26x^2+213x-544)$ \\

     $(x+5)^{32} (x-7)^2 (x-9)^7 (x-11)^6 (x^2-17x+64)$ \\

     $(x+5)^{32} (x-9)^8 (x-11)^4 (x^2-20x+95) (x^3-24x^2+183x-436)$ \\

     $(x+5)^{32} (x-9)^8 (x-11)^4 (x-12) (x^2-16x+59)^2$ \\

     $(x+5)^{32} (x-7)^2 (x-9)^6 (x-11)^7 (x^2-15x+52)$ \\

     $(x+5)^{32} (x-9)^9 (x-11)^4 (x^4-35x^3+443x^2-2385x+4568)$ \\

     $(x+5)^{32} (x-8) (x-9)^6 (x-11)^6 (x^2-16x+59)^2$ \\

     $(x+5)^{32} (x-9)^{10} (x-11)^4 (x^3-26x^2+209x-500)$ \\

     $(x+5)^{32} (x-9)^9 (x-11)^5 (x^3-24x^2+179x-412)$ \\

     $(x+5)^{32} (x-5) (x-9)^8 (x-11)^6 (x^2-17x+68)$ \\

     $(x+5)^{32} (x-7) (x-9)^8 (x-11)^6 (x^2-15x+48)$ \\

     $(x+5)^{32} (x-9)^{10} (x-11)^4 (x^3-26x^2+209x-488)$ \\

         \hline
             \caption{List of 164 candidate characteristic polynomials for 49 equiangular lines in $\mathbb R^{17}$.}
    \label{tab:164polydim17}
\end{longtable}
\end{center}

\begin{center}

\footnotesize
\begin{longtable}{@{\makebox[3em][r]{\rownumbers\space}} | l}

\endfirsthead

\multicolumn{1}{c}
{{\bfseries \tablename\ \thetable{} -- continued from previous page}} \\ \hline 
\endhead

\hline \multicolumn{1}{r}{{Continued on next page}} \\ \hline
\endfoot

\endlastfoot

\multicolumn{1}{@{\makebox[3em][r]{Index~}} | l}{Certificate of infeasibility}\\
\hline 
     $(-5488432, 0, 0, -2469, -274)$
    
     \\
    $(-63703708, 0, 0, -31134, -4447)$
    
     \\
    $(0, 0, 0, 0, -24516, -11277, -2933)$
    
     \\
    $(64094624850, 0, 0, 13282715, 1315121, 131512)$

     \\
    $(0, 0, 0, 0, -32892, -14918, -3771)$
    
     \\
    $(0, 0, 0, 2179953, 598246, 95219, 12792)$
    
     \\
    $(0, 0, 0, 0, -46941, -21298, -5385)$
    
     \\
    $(0, 0, 0, 2020374, 527071, 75373, 8375)$
    
     \\
    $(0, 0, 0, 0, -186488, -84667, -21412)$
    
     \\
    $(0, 0, 0, 2213481, 577512, 82588, 9177)$
    
     \\
    $(0, 0, 0, 262395, 85409, 13919)$
    
     \\
    $(0, 0, 0, 0, -50486, -22673, -5606)$
    
     \\
    $(0, 0, 0, 0, -166968, -76347, -19567)$

     \\
    $(0, 0, 0, -761, -407)$
    
     \\
    $(0, 0, 0, 0, 0, 48433, 27206, 8450)$
    
     \\
    $(43568416850, 0, 0, 8939852, 867947, 86794)$
    
     \\
    $(418335562857, 0, 0, 50603993, 5042602, 533922, 59325)$
    
     \\
    $(0, 0, 0, 2951038, 757444, 105637, 11738)$
    
     \\
    $(2170237224, 0, 0, 518291, 54779, 5899)$
    
     \\
    $(0, 0, 0, -102473516, -26628634, -4344437, -608687, -79085)$
    
     \\
    $(0, 0, 0, 3027878, 841889, 137248, 19051)$
    
     \\
    $(0, 0, 0, 3424439, 942309, 150950, 20564)$
    
     \\
    $(-350855974043, 0, 0, -36070743, -2365294, -118265, 0)$
    
     \\
    $(0, 0, 0, 19725270, 5270911, 802935, 104149)$
    
     \\
    $(0, 0, 0, 27926, 10159, 2010)$
    
     \\
    $(0, 0, 0, 2358490, 648675, 103793, 14105)$
    
     \\
    $(0, 0, 0, 0, -11689, -5069, -1158)$
    
     \\
    $(-4870391404, 0, 0, -1228192, -147036, -17299)$
    
     \\
    $(0, 0, 0, 0, 0, 0, 0, -2769, -2722)$
    
     \\
    $(0, 0, 0, 0, 0, 49785, 28128, 8819)$
    
     \\
    $(0, 0, 0, 2693660, 663478, 83879, 7626)$
    
     \\
    $(0, 0, 0, 2151377, 521244, 62701, 4824)$
    
     \\
    $(0, 0, 0, 0, 0, 2730500, 1336652, 374346, 80967)$
    
     \\
    $(0, 0, 0, -3968275, -1342119, -230256)$
    
     \\
    $(0, 0, 0, 0, -21822, -9463, -2157)$
    
     \\
    $(0, 0, 0, 0, -1569046, -573439, -115384, -17518)$
    
     \\
    $(4598007557155, 0, 0, 531045255, 47338744, 4256617, 386966)$
    
     \\
    $(0, 0, 0, 0, -3978492, -1477259, -314027, -53170)$
    
     \\
    $(20419347700, 0, 0, 4029188, 354569, 32233)$
    
     \\
    $(0, 0, 0, 0, 0, 5733949, 2909745, 855377, 195699)$
    
     \\
    $(0, 0, 0, 0, 0, 14228263, 7055510, 2012098, 444987)$
    
     \\
    $(0, 0, 0, 5766425, 1554289, 240786, 31992)$
    
     \\
    $(400497673181, 0, 0, 46699387, 3348961, 186053, 0)$
    
     \\
    $(695180798312, 0, 0, 89011625, 8091966, 735633, 66876)$
    
     \\
    $(0, 0, -83458849775, -11996698757, -1188052506, -118190915, -11828260, -1191996)$
    
     \\
    $(0, 0, 0, 0, -21332, -9279, -2128)$
    
     \\
    $(0, 0, 0, 0, 0, -48253, -17954, 0)$
    
     \\
    $(0, 0, -61036254227, -8328939996, -759497210, -69872634, -6493326, -610484)$
    
     \\
    $(0, 0, 0, -3673067283, -954462930, -155719105, -21817365, -2834629)$
    
     \\
    $(0, 0, 0, 3242353, 857188, 128112, 16317)$
    
     \\
    $(0, 0, 0, -136206225, -35394372, -5774558, -809058, -105118)$
    
     \\
    $(0, 0, 0, 0, 8891029, 3501307, 607418)$
    
     \\
    $(0, 0, 0, 0, 0, 5173067, 2563924, 730730, 161530)$
    
     \\
    $(-43386189, 0, 0, -21122, -1920)$
    
     \\
    $(0, 0, 0, -103422864, -26769889, -4290940, -589184, -74958)$
    
     \\
    $(0, 0, 0, -100965325, -26236718, -4280496, -599729, -77921)$
    
     \\
    $(-1299165141433, 0, 0, -149992114, -13410925, -1221759, -113651)$
    
     \\
    $(0, 0, 11091944806, 1414982133, 88436388, 0, -1403751, -356509)$
    
     \\
    $(-2370518793825, 0, 0, -304205617, -28068180, -2673161, -267316)$
    
     \\
    $(0, 0, 0, 0, -1329832, -537956, -127678, -24303)$
    
     \\
    $(0, 0, -77300917310, -10532936094, -956541670, -87328370, -8017435, -740071)$
    
     \\
    $(0, 0, 0, -3659439217, -950921617, -155141345, -21736416, -2824112)$
    
     \\
    $(0, 0, 0, 1667190, 446418, 67961, 8668)$
    
     \\
    $(0, 0, -27280263472, -3518058798, -293088094, -25041704, -2203669, -200334)$
    
     \\
    $(0, 0, 0, 0, 0, 0, -1064222, 0, 285444, 160397)$
    
     \\
    $(0, 0, 0, 0, 0, 2337512, 1158887, 330521, 73173)$
    
     \\
    $(0, 0, 0, 0, -1531587, -589537, -129964, -22695)$
    
     \\
    $(0, 0, 0, 0, -35406, -15155, -3334)$
    
     \\
    $(0, 0, 0, -71612334, -18804223, -3075477, -431366, -56080)$
    
     \\
    $(0, 0, 0, 0, 0, -12974, 0, 4359)$
    
     \\
    $(0, 0, 0, 0, 0, 4677012, 2215622, 593361, 122148)$
    
     \\
    $(0, 0, 0, -379426866, -98595447, -16085678, -2253718, -292816)$
    
     \\
    $(0, 0, 0, 0, -1001181, -347490, -65629, -9371)$
    
     \\
    $(-18918837294, 0, 0, -5728490, -787474, -108085)$
    
     \\
    $(0, 0, 0, -15349003083, -3936212014, -661763942, -97017114, -13355954, -1781204)$
    
     \\
    $(0, 0, 0, 0, 0, -60315, -31871, -8411)$
    
     \\
    $(0, 0, 0, 3059218, 768048, 97759, 8888)$
    
     \\
    $(0, 0, 0, -314129288, -78844475, -12081240, -1578458, -190430)$
    
     \\
    $(0, 0, 0, -64158067, -16846854, -2755344, -386464, -50242)$
    
     \\
    $(0, 0, 0, 0, -1704397, -620000, -127433, -20664)$
    
     \\
    $(0, 0, 0, 0, 0, 2880320, 1487377, 443235, 102448)$
    
     \\
    $(0, 0, 0, 0, 0, 262016, 152699, 49290)$
    
     \\
    $(0, 0, 0, -106796999, -28043138, -4586523, -643306, -83633)$
    
     \\
    $(0, 0, 0, 0, -6215064, -2516920, -598732, -114475)$
    
     \\
    $(0, 0, 0, 0, 0, 955948, 429395, 105520, 19167)$
    
     \\
    $(48031592256, 0, 0, 11040701, 1115611, 115408)$
    
     \\
    $(0, 0, 0, 0, 0, 0, -2143032, -1241360, -401833, -97740)$
    
     \\
    $(0, 0, 0, 0, 0, 0, -14047613, -8488569, -2909072, -758148)$
    
     \\
    $(0, 0, 0, 0, -1552628, -531841, -97912, -13407)$
    
     \\
    $(0, 0, 0, 0, 0, 0, -2618443, -1505910, -482407, -115715)$
    
     \\
    $(1617153799881, 0, 0, 197697391, 16639510, 1426243, 125110)$
    
     \\
    $(7063410652, 0, 0, 1589426, 144494, 13135)$
    
     \\
    $(0, 0, 0, -114103753, -29961502, -4900266, -687311, -89354)$
    
     \\
    $(0, 0, 0, 0, 31817, 13173, 2468)$
    
     \\
    $(0, 0, 0, 0, -10677039, -3964370, -842610, -142646)$
    
     \\
    $(0, 0, 0, 22638, 7941, 1375)$
    
     \\
    $(0, 0, 0, 0, 0, 6344046, 3245585, 956303, 218771)$
    
     \\
    $(0, 0, 0, -495338005, -121177788, -17614437, -2088461, -208847)$
    
     \\
    $(0, 0, 0, -393778194, -103398759, -16911081, -2371948, -308363)$
    
     \\
    $(0, 0, 0, 0, 0, 0, 12429, 8163, 2762)$
    
     \\
    $(0, 0, 0, 0, 0, 0, 0, 0, 1119, 1234)$
    
     \\
    $(-3885730122, 0, 0, -718522, -32660, -1)$
    
     \\
    $(-992098071642, 0, 0, -118216298, -9445854, -744219, -57247)$
    
     \\
    $(565131071791, 0, 0, 85703268, 9084766, 985095, 109456)$
    
     \\
    $(0, 0, 0, -138176827, -36282640, -5934101, -832316, -108205)$
    
     \\
    $(0, 0, 0, 0, -1213093, -449831, -95266, -15983)$
    
     \\
    $(0, 0, 0, 89931451, 24948054, 4341823, 647628, 89588)$
    
     \\
    $(0, 0, -19060138333, -2884207282, -300173384, -30699551, -3100964, -310097)$
    
     \\
    $(0, 0, 0, 0, 0, 9599576, 4809061, 1377672, 305452)$
    
     \\
    $(0, 0, -27750045610, -4064194735, -408278612, -41428056, -4245277, -439167)$
    
     \\
    $(0, 0, 0, 1087426, 314758, 52603, 7448)$
    
     \\
    $(0, 0, 0, 0, 0, 0, -1337775, -752950, -233517, -53538)$
    
     \\
    $(-2169430180, 0, 0, -198141, 4, 0, 0)$
    
     \\
    $(289768164194, 0, 0, 45345884, 5038432, 559825, 62203)$
    
     \\
    $(0, 0, 0, -70323109, -18465693, -3020109, -423600, -55070)$
    
     \\
    $(0, 0, 0, -89288753, -23286483, -3710063, -501572, -62051)$
    
     \\
    $(0, 0, 0, 0, 0, -15046, -7958, -2003)$
    
     \\
    $(0, 0, -55007201398, -7858513065, -756481319, -73298937, -7158493, -705767)$
    
     \\
    $(0, 0, 0, 0, -7607112, -2990049, -680855, -123814)$
    
     \\
    $(0, 0, 0, 0, 0, 0, -2570051, -1576448, -546592, -143665)$
    
     \\
    $(64738987904, 0, 0, 7963177, 461283, 0, -6621)$
    
     \\
    $(0, 0, -312204467074, -43839706334, -4468167233, -463940452, -49178872, -5332649, -592516)$ 
    
     \\
    $(0, 0, 0, -62996931, -16541959, -2705478, -379470, -49333)$
    
     \\
    $(0, 0, -34501373844, -4771922844, -433811167, -39437379, -3585216, -325929)$
    
     \\
    $(0, 0, 0, -1169094984, -287376982, -42813875, -5458939, -645630)$
    
     \\
    $(0, 0, 12829413037, 1887211369, 175111560, 14517041, 907320, 0)$
    
     \\
    $(0, 0, 0, 0, 0, 0, 0, -37148, -30261, -13559)$
    
     \\
    $(-178847202626, 0, 0, -25013595, -2273963, -206724, -18793)$
    
     \\
    $(1111960200, 0, 0, 365603, 48288, 6898)$
    
     \\
    $(0, 0, 0, -112031224, -29417295, -4811260, -674827, -87731)$
    
     \\
    $(0, 0, 0, 0, 0, 4443929, 2128486, 571736, 117372)$
    
     \\
    $(0, 0, -34375141077, -4754463451, -432223950, -39293087, -3572098, -324737)$
    
     \\
    $(0, 0, 0, 864770601, 222403153, 33916787, 4279939, 475548)$
    
     \\
    $(0, 0, 0, -3973255445, -998404736, -161564034, -22554302, -2924465, -362979)$
    
     \\
    $(0, 0, 0, 6247855, 1847930, 310209, 44003)$
    
     \\
    $(0, 0, 0, 0, 0, 0, 0, 0, 1137, 1309)$
    
     \\
    $(1294707168, 0, 0, 237202, 10317, 0)$
    
     \\
    $(39494208, 0, 0, 21372, 1645)$
    
     \\
    $(11007285490, 0, 0, 2883349, 285481, 28548)$
    
     \\
    $(0, 0, 0, 0, -16028637, -6112248, -1330674, -230596)$
    
     \\
    $(72932423060, 0, 0, 11832901, 1179751, 117975, 11798)$
    
     \\
    $(0, 0, 0, 0, -2848088, -1208753, -297136, -57413)$
    
     \\
    $(0, 0, 0, -364965233, -96892815, -15887789, -2230938, -290213)$
    
     \\
    $(0, 0, 0, -55660, -20011, -3553)$
    
     \\
    $(-1400886, 0, 0, -911, 0)$
    
     \\
    $(-2655400442, 0, 0, -824349, -91594, -10178)$
    
     \\
    $(163604694365, 0, 0, 26168863, 2536575, 246500, 23855)$
    
     \\
    $(0, 0, 0, -316467666, -81782200, -12927692, -1754267, -221407)$
    
     \\
    $(0, 0, 0, -411875739, -126767583, -24943894, -4253740, -676710)$
    
     \\
    $(0, 0, 0, 0, 0, 0, 2204, 2101)$
    
     \\
    $(0, 0, 0, 0, -2336823, -882211, -189036, -32146)$
    
     \\
    $(0, 0, 0, 0, -8943902, -3542882, -788607, -138366)$
    
     \\
    $(0, 0, 0, -2568167, -756561, -121188, -16321)$
    
     \\
    $(1294235800, 0, 0, 373032, 20725, -1)$
    
     \\
    $(0, 0, 0, 0, -1467454, -617730, -151051, -29543)$
    
     \\
    $(949975748, 0, 0, 235960, 21451, 1950)$
    
     \\
    $(0, 0, 0, 24726, 11067, 2893)$
    
     \\
    $(0, 0, 0, 1421580, 478465, 97821, 16654)$
    
     \\
    $(-364383063, 0, 0, -124960, -13884, -1543)$
    
     \\
    $(519279994, 0, 0, 161166, 17908, 1989)$
    
     \\
    $(0, 0, 0, -4803, -1591, -145)$
    
     \\
    $(-64725455, 0, 0, -25323, -2681, -298)$
    
     \\
    $(-3135702600, 0, 0, -1171021, -104988, -8077)$
    
     \\
    $(0, 0, 0, -84876, -27407, -2492)$ \\

         \hline
             \caption{ Certificates of infeasibility for each polynomial in Table~\ref{tab:164polydim17}.}
    \label{tab:164polydim17infeas}
\end{longtable}
\end{center}

\begin{center}
\footnotesize
    \begin{longtable}{r|l}

\endfirsthead

\multicolumn{2}{c}
{{\bfseries \tablename\ \thetable{} -- continued from previous page}} \\ \hline 
\endhead

\hline \multicolumn{2}{r}{{Continued on next page}} \\ \hline
\endfoot

\endlastfoot

	Index & Candidate characteristic polynomial \\
    \hline
     1 &  $(x+5)^{32} (x-9)^{14} (x-11) (x^2-23x+116)$  \\
 	 2 & $(x+5)^{32} (x-9)^{14} (x^3-34x^2+369x-1264)$ \\
	 3 & $(x+5)^{32} (x-9)^{12} (x^2-20x+95) (x^3-32x^2+327x-1068)$ \\
	 4 & $(x+5)^{32} (x-9)^{12} (x^5-52x^4+1062x^3-10636x^2+52193x-100376)$ \\
     5 & $(x+5)^{32} (x-9)^{12} (x^2-20x+95) (x^3-32x^2+327x-1052)$ \\
	 6 & $(x+5)^{32} (x-9)^{12} (x-11)^2 (x^3-30x^2+281x-812)$ \\
     7 & $(x+5)^{32} (x-9)^{12} (x-11)^2 (x^3-30x^2+281x-808)$ \\
     8 & $(x+5)^{32} (x-8) (x-9)^{11} (x-11)^3 (x^2-20x+83)$ \\
     9 & $(x+5)^{32} (x-7) (x-9)^{10} (x-11)^3 (x^3-30x^2+285x-852)$ \\
     10 & $(x+5)^{32} (x-9)^{11} (x-11)^2 (x-13) (x^3-26x^2+213x-556)$ \\
   	 11 & $(x+5)^{32} (x-5) (x-9)^{12} (x-11)^2 (x-12) (x-13)$ \\
    \hline
    
             \caption{List of 11 candidate characteristic polynomials for 49 equiangular lines in $\mathbb R^{17}$.}
    \label{tab:11polydim17}
\end{longtable}
\end{center}

\begin{center}

\footnotesize
    \begin{longtable}{r|l}

\endfirsthead

\multicolumn{2}{c}
{{\bfseries \tablename\ \thetable{} -- continued from previous page}} \\ \hline 
\endhead

\hline \multicolumn{2}{r}{{Continued on next page}} \\ \hline
\endfoot

\endlastfoot

		Index & Warranted polynomials and certificates of warranty \\
	\hline
        \multirow{4}{*}{1} &  
        $(x+5)^{31} (x-9)^{13} (x^4-38x^3+508x^2-2810x+5363)$ \\
        & $(-28491688, 0, 0, -13311, -1663)$ \\
	\cline{2-2}
        & $(x+5)^{31} (x-9)^{13} (x-11) (x^3-27x^2+211x-481)$ \\
        & $(-71206868, 0, 0, -31318, -3131)$ \\
    \hline
    
   \multirow{4}{*}{2} & 
    $(x+5)^{31} (x-9)^{13} (x^4-38x^3+508x^2-2802x+5339)$ \\
        & $(-31466028, 0, 0, -15413, -2201)$ \\
		\cline{2-2}
        &  $(x+5)^{31} (x-9)^{13} (x^4-38x^3+508x^2-2794x+5267)$ \\
        & $(-32393128, 0, 0, -14283, -1428)$ \\
	\hline
	
	  \multirow{4}{*}{3} & 
         $(x+5)^{31} (x-9)^{11} (x^6-56x^5+1273x^4-15008x^3+96523x^2-319816x+423835)$  \\
        & $(3081463054228, 0, 0, 329213615, 23608945, 1311608, 0)$ \\
	\cline{2-2}
        & $(x+5)^{31} (x-9)^{11} (x^2-20x+95) (x^4-36x^3+458x^2-2428x+4485)$ \\
        & $(0, 0, 0, 0, -239073, -107333, -26537)$ \\
    \hline

     \multirow{4}{*}{4} & 
    $(x+5)^{31} (x-9)^{11} (x^6-56x^5+1273x^4-15000x^3+96371x^2-319088x+423571)$  \\
   & $(0, 0, 0, 576862162, 159550736, 25867427, 3614120)$ \\
   	\cline{2-2}
   &  $(x+5)^{31} (x-9)^{11} (x^6-56x^5+1273x^4-15000x^3+96371x^2-319056x+423347)$ \\
   &     $(-900198292671, 0, 0, -70576836, 0, 678324, 136910)$ \\
	\hline
    
     \multirow{4}{*}{5} & 
     $(x+5)^{31} (x-9)^{11} (x^6-56x^5+1273x^4-14992x^3+96139x^2-316952x+417243)$  \\
      &   $(15798865000678, 0, 0, 1810892027, 159441460, 14258992, 1296273)$ \\
	   \cline{2-2}
      & $(x+5)^{31} (x-9)^{11} (x^2-20x+95) (x^4-36x^3+458x^2-2412x+4405)$ \\
      &   $(0, 0, 0, 8229987, 1467053, 0, -42311)$ \\
	\hline
	
	 \multirow{4}{*}{6} &  $(x+5)^{31} (x-9)^{11} (x-11) (x^2-14x+41) (x^3-31x^2+303x-921)$  \\
        
       & $(7723650238, 0, 0, 1795746, 177797, 17779)$ \\
        \cline{2-2} 
       &  $(x+5)^{31} (x-9)^{11} (x-11)^2 (x^2-14x+41) (x^2-20x+83)$ \\
        
       & $(1681764478, 0, 0, 411472, 45720, 5079)$ \\
    \hline 
	
	 \multirow{4}{*}{7} &   $(x+5)^{31} (x-9)^{11} (x-11) (x^5-45x^4+778x^3-6426x^2+25189x-37289)$  \\
        
      &  $(8311896038, 0, 0, 1932968, 192633, 19927)$  \\
        
       \cline{2-2}
		& $(x+5)^{31} (x-9)^{11} (x-11)^2 (x^4-34x^3+404x^2-1974x+3347)$ \\
        
       & $(2513477598, 0, 0, 614905, 68323, 7591)$ \\
    \hline 
	
	 \multirow{4}{*}{8} &  $(x+5)^{31} (x-9)^{10} (x-11)^2 (x^5-43x^4+710x^3-5618x^2+21257x-30675)$  \\
        
       & $(-615702770, 0, 0, -180140, -22209, -2468)$ \\
        
        \cline{2-2}
		
		& $(x+5)^{31} (x-9)^{10} (x-11)^3 (x^4-32x^3+358x^2-1672x+2753)$ \\
        
       & $(-3736443160, 0, 0, -1027960, -114217, -12691)$ \\
    \hline 
    
	 \multirow{4}{*}{9} &   $(x+5)^{31} (x-7) (x-9)^9 (x-11)^2 (x^5-45x^4+782x^3-6534x^2+26113x-39709)$  \\
        
      &  $(1614272396498, 0, 0, 236406108, 23278501, 2290900, 221701)$ \\
        
       \cline{2-2}
		
		& $(x+5)^{31} (x-9)^9 (x-11)^2 (x^6-52x^5+1097x^4-12000x^3+71603x^2-220076x+270467)$ \\
        
        & $(693739600051, 0, 0, 103053553, 10456851, 1071550, 110851)$ \\
    \hline 
    
	 \multirow{4}{*}{10} &  $(x+5)^{31} (x-9)^{10} (x-11) (x-13) (x^5-41x^4+650x^3-4970x^2+18293x-25869)$  \\
        
       & $(0, 0, 0, 0, 28753, 11854, 2194)$ \\
        
       \cline{2-2}
		
		& $(x+5)^{31} (x-9)^{10} (x-11)^2 (x^2-20x+95) (x^3-23x^2+155x-317)$ \\
        
       & $(0, 0, 0, 4052423, 1116206, 173878, 22882)$
    \\
    \hline 
	 
	 \multirow{4}{*}{11} &  $(x+5)^{31} (x-9)^{11} (x-11) (x-13) (x^4-32x^3+362x^2-1696x+2789)$  \\
        
        & $(-5637966054, 0, 0, -1233304, -103494, -8625)$ \\
        
       \cline{2-2}
		
		& $(x+5)^{31} (x-9)^{11} (x-11)^2 (x^2-16x+43) (x^2-18x+73)$ \\
        
       & $(-27025283304, 0, 0, -5833300, -470427, -37635)$ \\
         
         \hline
    
             \caption{Pairs of Seidel-incompatible warranted interlacing characteristic polynomials together with certificates of warranty for each of the polynomials listed in Table~\ref{tab:11polydim17}.}
    \label{tab:11polydim17warrant}
\end{longtable}
\end{center}

\begin{center}

\footnotesize
    \begin{longtable}{r|l}

\endfirsthead

\multicolumn{2}{c}
{{\bfseries \tablename\ \thetable{} -- continued from previous page}} \\ \hline 
\endhead

\hline \multicolumn{2}{r}{{Continued on next page}} \\ \hline
\endfoot

\endlastfoot

	Index & Candidate characteristic polynomial \\
    \hline
    1 & $(x+5)^{32} (x-9)^{12} (x-11) (x^4-41x^3+611x^2-3919x+9140)$ \\
    2 & $(x+5)^{32} (x-9)^{12} (x^5-52x^4+1062x^3-10640x^2+52249x-100508)$ \\
    3 & $(x+5)^{32} (x-7)^2 (x-9)^{11} (x-11)^2 (x^2-25x+152)$ \\
    4 & $(x+5)^{32} (x-9)^{12} (x^5-52x^4+1062x^3-10628x^2+52001x-99248)$ \\
    5 & $(x+5)^{32} (x-9)^{13} (x-11) (x-12) (x^2-20x+83)$ \\
    6 & $(x+5)^{32} (x-9)^{11} (x-11)^2 (x^4-39x^3+551x^2-3321x+7160)$ \\
    7 & $(x+5)^{32} (x-5) (x-8) (x-9)^{10} (x-11)^4 (x-13)$ \\
    8 & $(x+5)^{32} (x-9)^{11} (x-11)^4 (x^2-17x+56)$ \\
	\hline
    
             \caption{List of 8 candidate characteristic polynomials for 49 equiangular lines in $\mathbb R^{17}$.}
    \label{tab:8polydim17}
\end{longtable}
\end{center}

\begin{center}

\footnotesize
    \begin{longtable}{r|l}

\endfirsthead

\multicolumn{2}{c}
{{\bfseries \tablename\ \thetable{} -- continued from previous page}} \\ \hline 
\endhead

\hline \multicolumn{2}{r}{{Continued on next page}} \\ \hline
\endfoot

\endlastfoot

		Index & Warranted polynomial and certificate of warranty \\
	    \hline
    \multirow{2}{*}{1} & $(x+5)^{31} (x-9)^{11} (x-11) (x^5-45x^4+778x^3-6442x^2+25477x-38329)$ \\
    & $(0, 0, 0, 1746419, 491392, 81630, 11528)$ \\
	\hline

    \multirow{2}{*}{2} & $(x+5)^{31} (x-9)^{11} (x^6-56x^5+1273x^4-15000x^3+96339x^2-318544x+421459)$ \\
    & $(0, 0, 0, 2578890, 690376, 105095, 13403)$ \\
	\hline
    
    \multirow{2}{*}{3} & $(x+5)^{31} (x-7) (x-9)^{10} (x-11) (x^5-47x^4+854x^3-7466x^2+31161x-49015)$ \\
    & $(49638392174, 0, 0, 10008058, 946709, 90162)$ \\
	\hline
    
	\multirow{2}{*}{4} & $(x+5)^{31} (x-9)^{11} (x^6-56x^5+1273x^4-14992x^3+96155x^2-317304x+419371)$ \\
    & $(0, 0, 0, 3772987, 1029536, 162478, 21753)$ \\
    \hline
    
	\multirow{2}{*}{5} & $(x+5)^{31} (x-9)^{12} (x^5-47x^4+850x^3-7342x^2+30061x-46531)$ \\
    & $(-30845693350, 0, 0, -6031794, -513344, -42779)$ \\
    \hline
    
	\multirow{2}{*}{6} & $(x+5)^{31} (x-9)^{10} (x-11)^2 (x^5-43x^4+710x^3-5602x^2+21065x-30211)$ \\
    & $(0, 0, 0, 15489527, 4414523, 731381, 103163)$ \\
    \hline
    
	\multirow{2}{*}{7} & $(x+5)^{31} (x-9)^9 (x-11)^3 (x^5-41x^4+646x^3-4870x^2+17497x-23889)$ \\
    & $(-4146865216, 0, 0, -1459666, -205965, -26866)$ \\
    \hline
    
	\multirow{2}{*}{8} & $(x+5)^{31} (x-9)^{10} (x-11)^3 (x^2-14x+37) (x^2-18x+69)$ \\
    & $(13345166, 0, 0, 9135, 610)$ \\
	\hline
    
             \caption{Warranted interlacing characteristic polynomials together with certificates of warranty for each of the polynomials listed in Table~\ref{tab:8polydim17}.}
    \label{tab:8polydim17warrant}
\end{longtable}
\end{center}

\begin{center}

\footnotesize
    \begin{longtable}{r|l}

\endfirsthead

\multicolumn{2}{c}
{{\bfseries \tablename\ \thetable{} -- continued from previous page}} \\ \hline 
\endhead

\hline \multicolumn{2}{r}{{Continued on next page}} \\ \hline
\endfoot

\endlastfoot

		Index & Seidel-compatible interlacing characteristic polynomials and certificate of infeasibility \\
	    \hline
    \multirow{7}{*}{1} & $(x+5)^{31} (x-9)^{11} (x^6-56x^5+1273x^4-15000x^3+96275x^2-317328x+415699)$ \\
    & $(x+5)^{31}(x-7) (x-9)^{11}(x-11)^2 (x^3-27x^2+215x-493)$ \\
    & $(x+5)^{31} (x-9)^{11}(x-11) (x^5-45x^4+778x^3-6442x^2+25477x-38329)$ \\
    & $(x+5)^{31} (x-9)^{12} (x-11)(x^4-36x^3+454x^2-2332x+3929)$ \\
    & $(x+5)^{31} (x-9)^{12} (x^5-47x^4+850x^3-7326x^2+29629x-43715)$ \\
    & $(x+5)^{31} (x-9)^{11} (x-11) (x^2-22x+113) (x^3-23x^2+159x-313)$ \\
    \cline{2-2}
    & $(-733717562599, 0, 0, -88445449, -8599751, -854743, -87218)$ \\
    \hline
    
    \multirow{5}{*}{2}  & $(x+5)^{31} (x-9)^{12} (x^5-47x^4+850x^3-7350x^2+30141x-46283)$ \\
    & $(x+5)^{31} (x-9)^{11} (x^6-56x^5+1273x^4-15000x^3+96339x^2-318544x+421459)$ \\
     & $(x+5)^{31} (x-9)^{12} (x^5-47x^4+850x^3-7342x^2+29981x-45651)$ \\
      & $(x+5)^{31} (x-9)^{12} (x^5-47x^4+850x^3-7318x^2+29437x-42539)$ \\
       \cline{2-2}
    & $(1986879326641, 0, 0, 183663110, 8328599, 0, -81386)$ \\
	\hline
    
    \multirow{8}{*}{3} & $(x+5)^{31}(x-7) (x-9)^{11} (x-11)^2 (x^3-27x^2+215x-493)$ \\
    & $(x+5)^{31} (x-7) (x-9)^{10} (x-11) (x^5-47x^4+854x^3-7466x^2+31161x-49015)$ \\
    & $(x+5)^{31}(x-7) (x-9)^{12} (x-11) (x^3-29x^2+251x-591)$ \\
    & $(x+5)^{31} (x-7)(x-9)^{10} (x-11)^2(x^4-36x^3+458x^2-2420x+4397)$ \\
    & $(x+5)^{31} (x-7)(x-9)^{10} (x-11)^2(x^4-36x^3+458x^2-2420x+4429)$ \\
    & $(x+5)^{31} (x-7)(x-9)^{11} (x-11) (x^2-24x+139) (x^2-14x+37)$ \\
    & $(x+5)^{31}(x-7) (x-9)^{11} (x-11)^2 (x^3-27x^2+215x-445)$ \\
     \cline{2-2}
    & $(0, 0, 0, 58308, 28387, 9164)$ \\
    \newpage
	\multirow{7}{*}{4} & $(x+5)^{31} (x-9)^{11} (x^6-56x^5+1273x^4-14992x^3+96155x^2-317304x+419371)$ \\
	 & $(x+5)^{31}  (x-9)^{12} (x^5-47x^4+850x^3-7334x^2+29869x-45291)$ \\
	  & $(x+5)^{31} (x-9)^{13}(x^4-38x^3+508x^2-2754x+4843)$ \\
	  & $(x+5)^{31}  (x-9)^{12} (x^5-47x^4+850x^3-7310x^2+29277x-41779)$ \\
	  & $(x+5)^{31}  (x-9)^{12} (x^5-47x^4+850x^3-7310x^2+29341x-42483)$ \\
	  & $(x+5)^{31} (x-9)^{12}(x-13)(x^4-34x^3+408x^2-1990x+3119)$ \\
	  \cline{2-2}
    & $(877829032285, 0, 0, 0, -19643059, -4005362, -619899)$ \\
    \hline
    
	\multirow{4}{*}{5} & $(x+5)^{31} (x-9)^{12} (x^5-47x^4+850x^3-7342x^2+30061x-46531)$ \\
	& $(x+5)^{31} (x-9)^{12}(x-11)(x^4-36x^3+454x^2-2316x+3881)$ \\
	 & $(x+5)^{31} (x-3)(x-9)^{12} (x-11)^2 (x^2-22x+113)$ \\
	 \cline{2-2}
    & $(0, 0, 0, 17473, 5644, 889)$ \\
    \hline
    
	\multirow{10}{*}{6} & $(x+5)^{31} (x-9)^{10} (x-11)^2 (x^5-43x^4+710x^3-5602x^2+21049x-30131)$ \\
	& $(x+5)^{31} (x-9)^{10} (x-11)^2 (x^5-43x^4+710x^3-5602x^2+21065x-30211)$ \\
	& $(x+5)^{31}(x-7) (x-9)^{10} (x-11)^3 (x^3-25x^2+183x-367)$ \\
	& $(x+5)^{31}(x-9)^{11} (x-11)^2(x-13) (x^3-21x^2+131x-247)$ \\
	& $(x+5)^{31}(x-9)^{10} (x-11)^2(x^2 -18x+73) (x^3-25x^2+187x-395)$ \\
	& $(x+5)^{31} (x-9)^{10} (x-11)^2 (x^5-43x^4+710x^3-5578x^2+20553x-27611)$ \\
	 & $(x+5)^{31} (x-9)^{11}(x-11)^2(x^4-34x^3+404x^2-1934x+2939)$ \\
	 & $(x+5)^{31} (x-9)^{11}(x-11)^2(x^4-34x^3+404x^2-1926x+2931)$ \\
	 & $(x+5)^{31}(x-9)^{11} (x-11)^3(x^3-23x^2+151x-241)$ \\
	  \cline{2-2}
    & $(0, 0, 0, 0, 39830, 17496, 3881)$ \\
    \hline
    
	\multirow{4}{*}{7} & $(x+5)^{31} (x-9)^9 (x-11)^3 (x^5-41x^4+646x^3-4870x^2+17497x-23889)$ \\
	& $(x+5)^{31}(x-5)(x-9)^{9}(x-11)^4 (x-11)^3(x^3-25x^2+191x-431)$ \\
	& $(x+5)^{31} (x-9)^{10}(x-11)^3(x^4-32x^3+358x^2-1624x+2449)$ \\
	\cline{2-2}
    & $(0, 0, 0, -312193, -119497, -23881)$ \\
    \hline
    
	\multirow{2}{*}{8} & $(x+5)^{31} (x-9)^{10} (x-11)^3 (x^2-14x+37) (x^2-18x+69)$ \\
    & $(-11430609, 0, 0, -7006, 0)$ \\
	\hline
    
             \caption{Lists of Seidel-compatible interlacing characteristic polynomials together with certificates of infeasibility for each polynomial in 
             Table~\ref{tab:8polydim17} with respect to its warranted polynomial in Table~\ref{tab:8polydim17warrant}.}
    \label{tab:8polydim17compat}
\end{longtable}
\end{center}

\end{document}